\documentclass[reqno]{amsart}
\usepackage{amsfonts}
\usepackage{dsfont}
\usepackage{amssymb,amscd,amsthm, verbatim,amsmath,color, fancyhdr, mathrsfs}
\usepackage{comment}
\usepackage{graphicx}
\usepackage{turnstile}
\usepackage{cite}
\usepackage{tikz-cd}
\usepackage{thmtools}
\usepackage{mathtools}
\usepackage{thm-restate}

\usepackage{float}

\usepackage[margin=1.25in]{geometry}
\usepackage[font=small,labelfont=bf]{caption}
\usepackage{stmaryrd}
\usepackage{pict2e}
\usepackage[english]{babel}
\usepackage[hidelinks]{hyperref}
\hypersetup{colorlinks}
\definecolor{darkred}{rgb}{0.5,0,0}
\definecolor{darkgreen}{rgb}{0,0.5,0}
\definecolor{darkblue}{rgb}{0,0.3,0.8}
\hypersetup{colorlinks, linkcolor=black, filecolor=darkgreen, urlcolor=darkred, citecolor=darkblue}
\usepackage{cleveref}
\usepackage{kpfonts}
\usepackage[titletoc]{appendix}
\usepackage[numbers,square]{natbib} 
\usepackage{xcolor}
\usepackage[all,cmtip]{xy}
\usepackage{enumitem}
\usepackage{cases}

\numberwithin{equation}{section}
\begingroup\lccode`~=`& \lowercase{\endgroup
\newcommand{\spmat}[1]{%
  \left(
  \let~=&
  \begin{smallmatrix}#1\end{smallmatrix}
  \right)
}}

\makeatletter
\DeclareRobustCommand{\intprod}{%
  \mathbin{\mathpalette\int@prod{(0.1,0)(0.9,0)(0.9,0.8)}}%
}
\DeclareRobustCommand{\intprodr}{%
  \mathbin{\mathpalette\int@prod{(0.1,0.8)(0.1,0)(0.9,0)}}}

\newcommand{\int@prod}[2]{%
  \begingroup
  \sbox\z@{$\m@th#1+$}%
  \setlength\unitlength{\wd\z@}%
  \begin{picture}(1,1)
  \roundcap
  \polyline#2
  \end{picture}%
  \endgroup
}
\makeatother
\newtheorem*{exr*}{Exercise}
\newtheorem*{thm*}{Theorem}
\newtheorem{thm}{Theorem}[section]
\newtheorem{prop}[thm]{Proposition}
\newtheorem{lem}[thm]{Lemma}
\newtheorem{rem}[thm]{Remark}

\newtheorem{cor}[thm]{Corollary}

\newtheorem{defn}[thm]{Definition}

\theoremstyle{definition}
\newtheorem{example}[thm]{Example}

\newcommand{\Z}{\mathbb{Z}}
\newcommand{\R}{\mathbb{R}}
\newcommand{\U}{\mathbb{U}}
\newcommand{\Q}{\mathbb{Q}}
\newcommand{\C}{\mathbb{C}}
\newcommand{\bS}{\mathbb{S}}

\newcommand{\bb}{\boldsymbol{b}}

\newcommand{\CF}{\mathrm{CF}}

\newcommand{\cY}{\mathcal{Y}}
\newcommand{\bL}{\mathbb{L}}

\newcommand{\A}{\mathbb{A}}

\newcommand{\bM}{\mathcal{M}}
\newcommand{\cF}{\mathscr{F}}

\newcommand{\cS}{\mathcal{S}}

\newcommand{\cX}{\mathcal{X}}

\newcommand{\bF}{\mathbf{F}}
\newcommand{\op}{\textrm{op}}
\newcommand{\fr}{\textrm{fr}}

\newcommand{\Aut}{\mathrm{Aut}}

\newcommand{\Hom}{\mathrm{Hom}}

\newcommand{\Fuk}{\mathrm{Fuk}}

\newcommand{\GL}{\mathrm{GL}}

\newcommand{\cE}{\mathcal{E}}

\newcommand{\one}{\mathbf{1}}

\newcommand{\bP}{\mathbb{P}}

\newcommand{\cA}{\mathcal{A}}

\newcommand{\cL}{\mathcal{L}}

\newcommand{\MF}{\mathrm{MF}}
\newcommand{\Tw}{\mathrm{Tw}}

\newcommand{\val}{\mathrm{val}}

\newcommand{\Id}{\mathrm{Id}}

\newcommand{\lmod}{{\operatorname{-mod}}}

\graphicspath{ {./pics/} }
\vspace{-2em}

\title{Mirror construction for Nakajima quiver varieties}

\author{Jiawei Hu, Siu-Cheong Lau and Ju Tan}

\date{\today}

\begin{document}

\begin{abstract}
	In this paper, we construct the ADHM quiver representations and the corresponding sheaves as the mirror objects of formal deformations of the framed immersed Lagrangian sphere decorated with flat bundles.  More generally, we construct Nakajima quiver varieties as localized mirrors of framed nodal unions of Lagrangian spheres in dimension two.  This produces a mirror functor from the Fukaya category of a framed plumbing of surfaces to the dg category of complexes of bundles over the corresponding Nakajima quiver varieties.  
	
	For affine ADE quivers in specific multiplicities, the corresponding (unframed) Lagrangian immersions are homological tori, whose moduli of stable deformations are asymptotically locally Euclidean (ALE) spaces.  We show that framed stable Lagrangian branes are transformed into monadic complexes of framed torsion-free sheaves over the ALE spaces.
	
	A main ingredient is the notion of framed Lagrangian immersions and their Maurer-Cartan deformations. Moreover, using the formalism of quiver algebroid stacks, we find isomorphisms between the moduli of stable Lagrangian immersions and that of special Lagrangian fibers of an SYZ fibration in the affine $A_n$ cases.


\end{abstract}

\maketitle 
{\vspace{-2em} \hypersetup{hidelinks} \tableofcontents }

\section{Introduction} 

The celebrated work of Atiyah-Drinfeld-Hitchin-Manin \cite{ADHM78} showed that all anti-self dual solutions to Yang-Mills equations over the sphere $\bS^4$ for compact classical Lie groups can be encoded as linear data.  It reduces a highly non-trivial nonlinear partial differential equation to a much more manageable algebraic equation. Using the Penrose fibration, Donaldson \cite{Don84} established a one-to-one correspondence between the isomorphism classes of the linear data and the isomorphism classes of framed holomorphic vector bundles over the complex projective plane $\C \bP^2$ (where a framed bundle means the trivialization of the bundle is fixed at the line at infinity $l_\infty \subset \C \bP^2$). Subsequently, Nakajima formalized these linear data by introducing quiver representations \cite{Nak94} and use them \cite{Nak99} to construct framed torsion-free sheaves over $\C \bP^2$. These works explain the profound connections between Yang-Mills instantons, framed representations of the ADHM quiver $Q^{\textrm{ADHM}}$ and framed sheaves of $\C\bP^2$, , unifying methods from gauge theory, algebraic geometry, and representation theory.

A motivation of this work comes from an observation of the relation between the ADHM quiver $Q^{\textrm{ADHM}}$ and the framed immersed nodal sphere $\bL^\fr_{\textrm{ADHM}}$.  See Figure \ref{fig:framedImmSph}.  This connection provides a symplecto-geometric perspective on the quiver $Q^{\textrm{ADHM}}$ and its representations.

The moduli space of representations of $Q^{\textrm{ADHM}}$ (or generally a framed double quiver) is a hyperK\"ahler manifold constructed as a  hyperK\"ahler quotient using the complex moment map of the gauge group.  A central idea of this paper is to reconstruct this hyperK\"ahler manifold as the (unobstructed) stable deformation space of $\bL^\fr_{\textrm{ADHM}}$ (or generally a framed nodal surface) within its own symplectic neighborhood.
In particular, we establish a direct correspondence between the complex moment map in the hyperK\"ahler quotient construction and the obstructions that arise in Lagrangian deformation theory. 
This approach provides a new perspective on hyperK\"ahler geometry by situating it within the symplectic and Lagrangian deformation framework.

\begin{figure}[htb!]
	\includegraphics[scale=0.6]{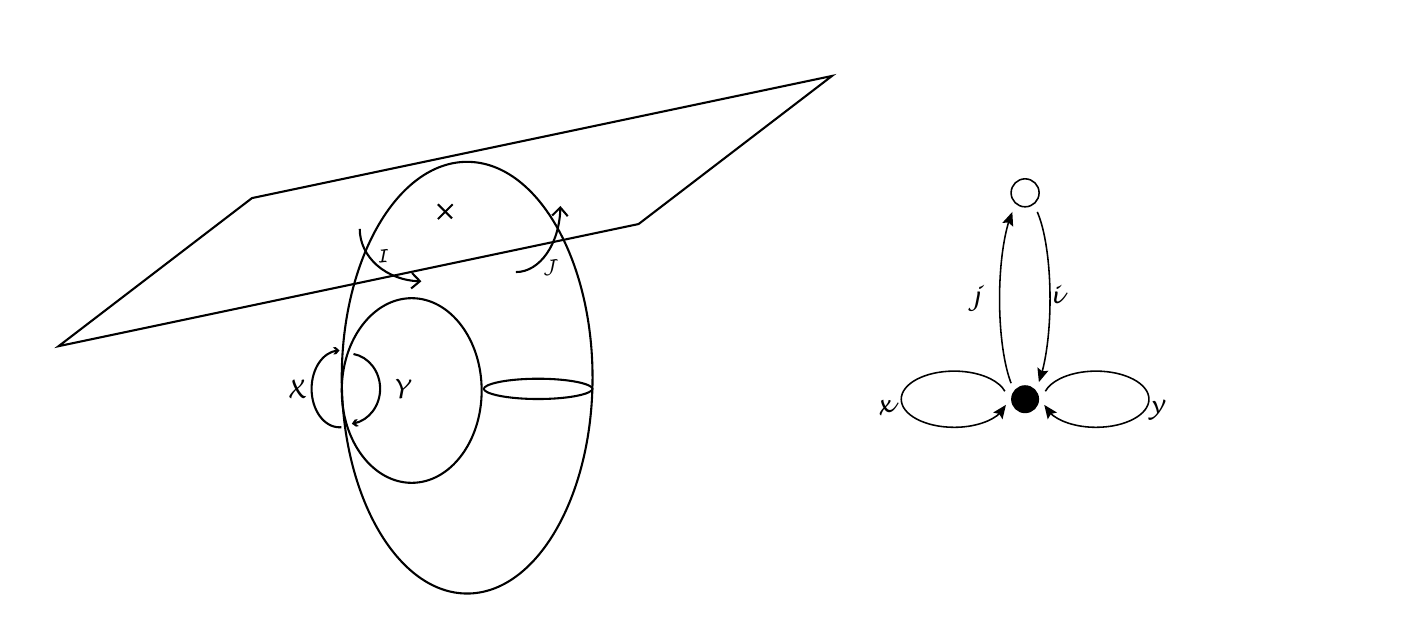}
	\caption{Framed immersed sphere and the ADHM quiver.}
	\label{fig:framedImmSph}
\end{figure}


\begin{thm}[Theorem \ref{thm:nc-framed}] \label{thm:intro-ADHM}
	Let $\Gamma Q^\mathrm{ADHM}$ be the completed path algebra of $ Q^\mathrm{ADHM}$. There exists a coordinate change on $\Gamma Q^\mathrm{ADHM}$ such that 
	the formal deformation space of the framed Lagrangian immersion $\mathbb{L}^{\mathrm{fr}}_{\mathrm{ADHM}}$ equals
	$$
	\A^\fr:=\Gamma Q^\mathrm{ADHM}\,\big/\,(xy-yx+ij).
	$$
\end{thm}

In above, $xy - yx + ij = 0$ gives the complex moment map equation.  HyperK\"ahler moduli of representations of $Q^\mathrm{ADHM}$ is a quotient of the affine variety ${xy - yx + ij = 0}$ by the gauge group of the quiver.  We show that this equation comes from obstructions of Lagrangian deformations.

This fits into the framework of mirror symmetry, which is a duality between symplectic geometry and complex geometry.  
According to homological mirror symmetry conjecture \cite{Kon95}, Lagrangian submanifolds in a symplectic manifold are dual to coherent sheaves on the mirror complex manifold.  The conjecture was proved for several important classes of geometries \cite{Polishchuk-Zaslow,Abo06, AKO08, Sei11,She11,FLTZ12, SeiK3,She15}. In this context, our work contributes to the study of mirror symmetry by connecting framed Lagrangian immersions to their mirror objects in the complex geometry setting.

In the above example of ADHM quiver, we establish the duality between the framed immersed sphere $\bL^\fr_{\textrm{ADHM}}$ and torsion-free sheaves of $\bP^2$ framed at the line at infinity $l_\infty$.  The duality is established using the localized mirror functor \cite{CHL21} constructed from the immersed sphere.  Notably, the functor directly maps framed Lagrangians to monads, which are essential tools in the study of torsion-free sheaves. For a review of the localized mirror functor, see Section \ref{section:nc mirror}.

\begin{thm}[Theorem \ref{thm:frash}]
	A stable framed Lagrangian brane supported on $\bL^\fr_{\mathrm{ADHM}}$ is transformed to a framed torsion-free sheaf over $(\bP^2,l_\infty)$ in the form of a monad under the localized mirror functor.
\end{thm}

An important concept in this theory is a framed Lagrangian brane.  
First, we will define a framing Lagrangian $F$ of a convex symplectic manifold $M$ to be a conical Lagrangian diffeomorphic to $\mathbb{R}^n$, which is equipped with a Morse function $f_F$ with exactly one maximal point (and no other critical point). 
This condition ensures the framing Lagrangian has a simple topological structure compatible with the convexity of $M$ (see Definition \ref{def:framing}).

A framed symplectic manifold is then defined as a convex symplectic manifold equipped with a collection $\bF$ of framing Lagrangians. Building on this, we define a framed Lagrangian brane as a union of two components: a compact Lagrangian immersion and the framing Lagrangian $\bF$, together with a formal deformation by a boundary cochain supported at the intersection points between the compact Lagrangian and the framing Lagrangians. When the Lagrangian submanifolds are equipped with flat vector bundles, these boundary cochains are represented by matrices. The underlying submanifold of the compact Lagrangian immersion is called to be the support of the Lagrangian brane.

\begin{rem}
	In Seidel's Picard-Lefschetz theory \cite{Sei01}, a framed Lagrangian sphere is defined as a Lagrangian submanifold together with a fixed diffeomorphism class to a smooth sphere.  To distinguish from this, we use the terms framed Lagrangian immersion and framed Lagrangian brane.  Our definition incorporates additional brane structures such as boundary cochains and flat bundles.  
\end{rem}


Now we move to general Nakajima quiver varieties, which are the moduli spaces of stable representations of a framed double quiver.  Nakajima found a geometric construction of representations of the quantum groups using Hecke correspondence of holomorphic Lagrangians in quiver varieties \cite{Nak94, Nak98, Nakajima-tensor, Nakajima-affine}, see also \cite{Beilinson-Lusztig-MacPherson, Rin90, Ginzburg91, Lusztig91, Grojnowski-Lusztig, Lus93, Kashiwara-Saito} for representations of quantum groups and canonical bases.

We will construct the hyperK\"ahler gauge theory using Lagrangian deformation theory of nodal surfaces. For each Nakajima's framed quiver, we construct a corresponding framed symplectic surface by plumbing.  First, we will show that Nakajima quiver varieties can be constructed as the unobstructed deformation space of framed Lagrangian nodal surfaces.

\begin{thm}[Theorem \ref{thm:preproj} and Proposition \ref{prop:nc-framed}]
	\label{thm:introplum}
	Let $D:=(I,E)$ be a graph and $\bL$ be the compact Lagrangian immersion obtained from the plumbing construction. Then the associated quiver $Q$ is the double quiver of the graph $D$. Moreover, up to a change of coordinates, the deformation space $\A$ of $\bL$ is the preprojective algebra of $Q$. Namely, $$\A= \A_{\text{free}}/J,$$ where $J$ is the two-sided ideal generated by $\sum_{t(a)=v} \epsilon(a)x_{\bar{a}}x_{a}$ for all $v \in I$. 
	
	Furthermore, the deformation space $\A^{\textrm{fr}}$ of the framed Lagrangian immersion $\bL^{\textrm{fr}}$ is $$\A^{\textrm{fr}}= \A^\fr_{\text{free}}/ J^{\textrm{fr}}, $$  where $J^{\textrm{fr}}$ is a two-sided ideal generated by $(\sum_{t(a)=v} \epsilon(a) x_{\bar{a}} x_{a})_v + (ij)_v$ for all $v \in I$.
	
	In particular, the Maurer-Cartan deformation space of a framed Lagrangian brane $(\bL^\fr,\cE)$ is a Nakajima quiver variety.
\end{thm}

In the study of quiver varieties, 
a main ingredient is the following monadic complex defined algebraically by Nakajima, see for instance the Equation 1.7 in \cite{Nak07}:

\begin{equation}
	\bigoplus \limits_{v\in I} \mathrm{Hom}(V^{1}_v,V^{2}_v) \to \bigoplus \limits_{a} \Hom(V^1_{t(a)},V^2_{h(a)}) \oplus \bigoplus_{v \in I} \mathrm{Hom}(V^{1}_v,W^{2}_v) \to \bigoplus \limits_{v\in I} \mathrm{Hom}(V^{1}_v,V^{2}_v), 
	\label{eq:monadic}
\end{equation} where $V^1_v$ is the tautological bundle associated with the vertex $v$ over the quiver variety $\bM(V^k,0)$, $V^2$ and $W^2$ are the trivial vector bundles.
Note that this complex always have three terms, even though the quiver variety is in higher dimensions.
It possesses profound applications and appears at various points in the study of quiver varieties, such as the ADHM construction of instantons over an asymptotically locally Euclidean (ALE) spaces space \cite{Kro89,KN90}, Kashiwara crystal structure \cite{Nak98}, and the definition of the Hecke correspondence \cite{Nak98}. 
More specifically, monad is a three-term complex that has cohomology concentrated in the middle term.  The cohomology of a monad yields a torsion-free sheaf, or a holomorphic vector bundle (with an induced connection from the trivial one) under a stronger fiberwise condition.  Moreover, quiver representations also naturally arise from monads. See for instance \cite{Don84,Nak99,KKO01,BGK02,Nak07,CLS11,Hen14,BBR15,BBLR17}.


\begin{thm}[Theorem \ref{thm:frmonad}]
	A framed Lagrangian brane supported on $\bL^{\textrm{fr}}$ is transformed to a monadic complex over the corresponding Nakajima quiver variety under the localized mirror functor $\cF^{\bL^\fr}$. 
\end{thm}

 In particular, our mirror construction provides a reason why three-term complexes are crucial for Nakajima quiver varieties, namely, Floer cochain complexes for a Lagrangian brane in a symplectic surface always have at most length three.


In summary, Nakajima framed quiver varieties are identified with the moduli spaces of stable unobstructed deformations of framed Lagrangian branes.  From this point of view, quiver theory arises to encode the stacky structure of deformation spaces of a Lagrangian immersion $\bL$, when $\bL$ has more than one component or when $\bL$ is equipped with a higher-rank flat bundle.   In such situations, the deformation space is a quotient stack by a non-compact gauge group, and quiver encodes this for bundles in all ranks simultaneously.

\begin{rem}
	The relation between symplectic topology and preprojective algebras of quivers is a very interesting subject. In \cite{BK16}, Bezrukavnikov-Kapranov studied the moduli space of microlocal sheaves on a nodal curve and found its intriguing connection to  multiplicative Nakajima quiver varieties.  In \cite{EL19}, Etg{\"u}-Lekili computed the wrapped Fukaya category of a plumbing of Riemann surfaces using Legendrian geometry and obtained a multiplicative preprojective algebra.  It has a close relation with the singular string topology introduced by \cite{CELN17}. Besides, there is a Koszul duality between compact and non-compact Lagrangians \cite{EL17,Hon23}.  Etg{\"u}-Lekili studied Koszul duality among these algebras in \cite{EL17}. This approach was further developed by \cite{RET22}.  More recently (after the first preprint version of the current work), Karabas and Lee \cite{KL24} computed the wrapped Fukaya category of the plumbings of spheres in general dimension along any quiver and proved the quasi-equivalence with the derived category of the multiplicative preprojective algebra.
	
	Our study in this paper is based on Lagrangian Floer theory \cite{FOOO09,AJ10} and moduli of local deformations of Lagrangian immersions, which is more geometric in the sense that 
	this uses gauge theory on nodal surfaces whose stable deformation spaces are naturally endowed with holomorphic symplectic forms.  The associated functor of \cite{CHL17,CHL21} produces monadic complexes over quiver varieties that are crucial in the study of torsion-free sheaves and the construction of Hecke correspondence for quantum groups. Note that since we work with local Lagrangian deformations, we obtain the original quiver varieties rather than the multiplicative ones. Nakajima's algebro-geometric results are readily applied in this setup. An important additional ingredient is framing on symplectic manifolds. 
	
\end{rem}




Among all quivers, the affine ADE quivers are particularly important. Kronheimer \cite{Kro89} constructed the ALE spaces using the moduli of the representations of affine ADE quivers. These spaces form an important class of resolutions of singularities and are the main objects of study in McKay correspondence.  Its three-dimensional analogs also provide a rich source of geometric surgeries, see for instance \cite{Fri86,KM92,Tia92, BKR01,STY02,MR2057015}.

From the perspective of Lagrangian Floer theory, affine ADE quivers in specific dimension vectors are special in the sense that the corresponding immersed Lagrangian branes have torus-type Lagrangian deformations (see Proposition \ref{prop:definite}).  In these cases, we can show that the moduli spaces are ALE spaces (Corollary \ref{prop:moduli}). From this perspective, ALE spaces are mirror to the Weinstein neighborhoods of affine ADE-type immersed Lagrangians. 

In addition, \cite{CS98,Sha05} found the connection of the (multiplicative) preprojective algebra for affine ADE quivers to Kleinian singularities. In particular, this implies that the loop algebras are commutative (which is not the case for general quivers). 

Motivated by these results, we expect that for affine ADE quivers, stable charts in a suitable rank (defined in Section \ref{Sec:Frcat}) are commutative.  We check this by explicit computations for affine $A_n$ and affine $D_4$ type quivers. 


\begin{prop}[Proposition \ref{Prop: D4}]
	Let $\cA_0$ be the path algebra of the affine $D_4$ type quiver and $\vec{\delta}=(2,1,1,1,1)$. Then there exists some stable commutative affine charts of the stable family $\cA_0$ over $\cA$ of rank $\vec{\delta}$, where $\cA$ is the noncommutative parameter space defined at the end of Section \ref{sec:stable}, which can be glued into the minimal resolution of $D_4$ singularity. 
\end{prop}

\begin{prop}[Theorem \ref{thm:qstack}]
	Let $\A$ be the path algebra of the affine $A_n$ type quiver and $\vec{\delta}=(1,\cdots,1)$. Then there exists some stable commutative affine charts of the stable family $\A$ over $\Lambda$ of rank $\vec{\delta}$, which can be glued into the minimal resolution of $A_n$ singularity.
\end{prop} 	


Stability conditions play an important role in resolving stacky singularities of moduli spaces. In our setup, we take families of stable unobstructed deformations of a Lagrangian brane to construct mirror affine charts.  On the other hand, differential geometric meaning of the stable deformation poses a highly challenging problem, see for instance the formulation of Joyce \cite{Joy15}.  

In \cite{TY02}, Thomas-Yau compared special Lagrangian spheres and stable objects for smoothing of an $A_n$ singularity.  
In this paper, we consider the affine $A_n$ case, in which Lagrangian tori and their degenerations play an important role.  We find a family of Fukaya isomorphisms between stable objects and special Lagrangians.

\begin{thm}[Theorem \ref{thm:qstack} and Theorem \ref{thm:qstack2}]
	Consider the affine $A_n$ surface $$X=\left\{(x,y,z)\in \mathbb{C}^3\mid xy = \prod_{k=1}^{n+1}(z-a_{k}) \text{ and }z\neq 0\right\},$$ where $-\infty < a_{n+1}<\cdots < a_{1}<0$ are distinct real numbers. This space admits a special Lagrangian fibration. Let $\bL$ be a cycle of vanishing spheres, depicted by Figure \ref{fig:An}, $\cS_i$ be the $i-$th singular special Lagrangian fiber, and $T_i$ be a family of smooth  fibers, which are special Lagrangian tori intersecting the $i-$th vanishing spheres. Then there exist Fukaya-isomorphisms between stable deformations of $\bL$ and deformations of special Lagrangians $\cS_i$ and $T_i$ over a quiver stack $\hat{\cY}$ for $i=1, \cdots, n$. 
\end{thm}

In the above theorem, we apply the technique of solving Fukaya isomorphism equations over the quiver algebroid stack, see \cite{LNT23} or Section \ref{sec:qstack} for the notion of quiver algebroid stack. A similar result was found in the case of deformed conifolds and Atiyah flop \cite{FHLY17}.

\subsection*{Notations} 
We will work over Novikov ring (field). We use the following notations: 
\begin{align*}
	\Lambda_{+} & =\left\{\sum\limits_{i=0}^{\infty}a_{i}T^{A_{i}}\mid A_{i}> 0\text{ increases to } +\infty, a_{i}\in \mathbb{C}\right\};\\
	\Lambda_{0} & =\left\{\sum\limits_{i=0}^{\infty}a_{i}T^{A_{i}}\mid A_{i}\geq 0\text{ increases to } +\infty, a_{i}\in \mathbb{C}\right\};\\
	\Lambda & =\left\{\sum\limits_{i=0}^{\infty}a_{i}T^{A_{i}}\mid A_{i}\text{ increases to } +\infty, a_{i}\in \mathbb{C}\right\}.
\end{align*}

They admit a non-archimedean valuation map (resp. norm) $$\mathrm{val}(x)=\min\{\lambda_{i}\mid a_{i}\neq 0\} \quad (\text{resp. } |x|_{\mathfrak{v}}=e^{-\mathrm{val}(x)}). $$

We will extend  Floer group and Fukaya algebra over path algebras. Recall a quiver is a quadruple $$Q = (I,E,h,t)$$where $I$ (and $E$) is the set of vertices (resp. arrows), and $$h,t:E\rightarrow I$$ are the maps assigning an arrow to its head (resp. tail). The functions $h,t$ easily generalize to paths in $Q$. 

The path algebra $kQ$ of a quiver $Q$ is the $k$-algebra with its basis consisting of all paths in $Q$ and multiplication defined by concatenation of paths extended by linearity. Paths are read from right to left as in the composition of functions.  

\subsection*{Acknowledgment} 

The first-named author thanks Xinyu Zhou for helpful discussions on non-Archimedean geometry.
The second-named author expresses his gratitude to Cheol-Hyun Cho, Hansol Hong and Yoosik Kim for numerous useful discussions and the collaborations on immersed Lagrangians and mirror functors, which form one of the main motivations for this paper.  The third-named author wants to express his gratitude to Ki Fung Chan, Kwokwai Chan, Bohan Fang, Mingyuan Hu, Yan-Lung Leon Li, Yu-Shen Lin, Chiu-Chu Melissa Liu, Kaoru Ono, Eriz Zaslow for valuable discussions.  We are grateful to Naichung Conan Leung for  enlightening discussions on 3D mirror symmetry and moduli theory during our visits to The Chinese University of Hong Kong. 

\section{Review of NC mirror functor and quiver stacks}

\subsection{Review of noncommutative mirror functor} 
\label{section:nc mirror}
Given a symplectic manifold $M$ and a compact spin oriented Lagrangian immersion $\bL \subset M$, we  associate a quiver $Q$ to $\bL$, whose path algebra encodes the boundary deformation of $\bL$. Let's recall the construction in \cite{CHL21}.

\begin{enumerate}
	\item 	
	Associate a quiver $Q$ to $\CF^1(\bL)$. Namely, each component of (the domain of) $\bL$ is associated with a vertex, and each generator in $\CF^1(\bL)$ is associated with an arrow.
	\item Let $\Lambda_0 Q$ be the free path algebra of $Q$ over Novikov ring. Each arrow $a$ in $Q$ is associated with a Novikov valuation such that $\text{val}(a) >0$. The valuation induces a filtration on $\Lambda_0 Q$. Take the completion of $\Lambda_0 Q$ with respect to this filtration. Abuse the notation, we will denote the completion by $\Lambda_0 Q$. 
	\item 
	Extend the Fukaya algebra $A$ of $\bL$ over the path algebra $\Lambda_0 Q$ and obtain a non-commutative $A_\infty$-algebra $$\tilde{A}^\mathbb{L} = \Lambda_0 Q\otimes_{\Lambda_0^{\oplus}} \CF(\mathbb{L}),$$ 
	whose unit is $\one_{\mathbb{L}} = \sum \one_{L_i}$.  $\Lambda_0^{\oplus} \subset \Lambda_0 Q$ denotes $\bigoplus_i \Lambda_0 \cdot e_i$ where $e_i$ are the trivial paths at vertices of $Q$.  The fibered tensor product means that an element $a \otimes X$ is non-zero only when tail of $a$ corresponds to the source of $X$.  The $A_\infty$-operations are defined by
	\begin{equation} \label{eq:mk}
		m_k (f_1 X_1,\ldots,f_k X_k) := f_k \ldots f_1 \, m_k (X_1,\ldots,X_k)
	\end{equation}
	where $X_l \in \CF(\bL)$  and $f_l \in \Lambda_0 Q$.
	\item Extend the formalism of bounding cochains of \cite{FOOO09} over $\Lambda_0 Q$, that is, we take 
	\begin{equation} \label{eq:b}
		\bb = \sum_l b_l B_l
	\end{equation}
	where $B_l$ are the generators of $\CF^1(\bL)$, and $b_l$ are the corresponding arrows in $Q$.  Then define the deformed $A_\infty$-structure $m_k^{\bb}$ as in \cite{FOOO09} and via Equation \eqref{eq:mk}.
	\item Quotient out the quiver algebra by the two-sided ideal $R$ generated by coefficients of the obstruction term $m_0^{\bb}$, so that $m_0^{\bb} = W\cdot \one_\bL$ over 
	$$\A := \Lambda_0 Q/R.$$  $\A$ is the space of noncommutative weakly unobstructed deformations of $\bL$.
	We call $(\A,W)$ to be the noncommutative local mirror of $X$ probed by $\bL$.
	\item Extend the Fukaya category over $\A$, and enlarge the Fukaya category by including the noncommutative family of objects $(\bL,\bb)$ where $\bb$ in \eqref{eq:b} is now defined over $\A$.  This means for $L_1,L_2$ in the original Fukaya category, the morphism space is now extended as $\A \otimes \CF(L_1,L_2)$.  The morphism spaces between $(\bL,\bb)$ and $L$ are enlarged to be $\CF((\bL,\bb),L) := \A\otimes_{\Lambda_0^{\oplus}} \CF(\mathbb{L},L)$ (and similarly for $\CF(L,(\bL,\bb))$).  We already have $\CF((\bL,\bb),(\bL,\bb))$ in Step 2 (except that $\Lambda_0 Q$ is replaced by $\A$).  The $m_k$ operations are extended in a similar way as \eqref{eq:mk}.
\end{enumerate}

\begin{thm}[\cite{CHL21}]
	Consider the boundary deformation $\bb$ of $\bL$. There exists a well-defined $A_\infty$-functor
	$$\cF^{(\bL,\bb)}: \Fuk(X) \to \MF(\A,W).$$
\end{thm}

\begin{rem}
	$m_0^{\bb}=W\cdot \one_\bL$ has degree $2$.  Thus in the $\Z$-graded situation, $W=0$, and the above $\MF(\A,W)$ reduces to the dg category of complexes of $\A$-modules.
\end{rem}

\begin{rem}
	The formal boundary deformation is part of the data of the original Fukaya algebra, which is a formalism to extract the stacky deformation space from the original Fukaya algebra. Here, stacky means it's natural to consider the isomorphism classes of Maurer-Cartan deformations, which is a quotient stack if we don't throw away unstable points. 
\end{rem}

\subsection{Quiver Algebroid Stack} \label{sec:qstack}

Philosophically, the noncommutative deformation space of $\bL$ forms a local neighborhood of the mirror. It's natural to ask how to glue the local mirrors. It is worth noting that the noncommutative deformation space of distinct Lagrangian immersions could be path algebras with varying numbers of vertices. Hence, the gluing is complicated in general and can not be achieved by the usual algebra isomorphisms.  This issue can be addressed by the notion of quiver algebroid stack introduced by the second, third authors and their collaborator \cite{LNT23}. Quiver algebroid stack is essential for the gluing of the quiver algebras.

The algebroid stack and twisted complexes were initially introduced by Bressler-Gorokhovsky-Nest-Tsygan \cite{BGNT07, BGNT08}. In the recent work \cite{LNT23}, the authors generalize these concepts for the purpose of mirror construction. More precisely, the authors introduce the localization of a path algebra and quiver algebroid stack, which provides a local-to-global approach to understand the path algebra of a quiver. Here we first recall the definitions of affine local chart and quiver algebroid stack. Following that, we will recall the mirror functors based on these constructions. Throughout this section, we will use $Q$ to denote a connected quiver and $I$ to denote the set of vertices. Let $k$ be a field.

\begin{defn} \label{def: loc1}
	Let $S \subset \cA=k Q/R$ be a finite subset of paths such that each summand of an element in $S$ has the same head and tail vertices. For each $\gamma \in S$, we add one arrow, denoted by $\gamma^{-1}$, with $s(\gamma^{-1})=t(\gamma)$ and $t(\gamma^{-1})=s(\gamma)$, to the quiver $Q$.  Moreover, we take the ideal $\hat{R}$ generated by $R$ and $\gamma \gamma^{-1} - e_{t(\gamma)}, \gamma^{-1} \gamma - e_{s(\gamma)}$ to be the new ideal of relations.  The new quiver with relations $k \hat{Q} / \hat{R}$ is called the localized algebra at $S$, and is denoted as $\cA (S^{-1})$.
\end{defn}

\begin{defn}
	Let $B$ be a topological space.  A quiver algebroid stack consists of the following data.
	\begin{enumerate}
		\item An open cover $\{U_i: i \in I\}$ of $B$.
		\item
		A sheaf of algebras $\mathcal{A}_i$ over each $U_i$, coming from localizations of a quiver algebra $\cA_i(U_i) = k Q^{(i)}/ R^{(i)}$.
		\item
		A sheaf of representations $G_{ij}$ of $Q^{(j)}_{V}$ over $\mathcal{A}_i(V)$ for every $i,j$ and $V \stackrel{\mathrm{open}}{\subset} U_{ij}$.
		\item
		An invertible element $c_{ijk}\left(v\right)\in \left(e_{G_{ij}(G_{jk}(v))}\cdot\mathcal{A}_{i}(U_{ijk})\cdot e_{G_{ik}(v)}\right)^{\times}$ for every  $i,j,k$ and $v \in Q^{(k)}_0$, that satisfies
		\begin{equation}
			G_{ij}\circ G_{jk}\left(a\right)=c_{ijk}\left(h(a)\right)\cdot G_{ik}\left(a\right)\cdot c_{ijk}^{-1}\left(t(a)\right)
			\label{eq:cocycle-gen}
		\end{equation}
		such that for any $i,j,k,l$ and $v$,
		\begin{equation}
			c_{ijk}(G_{kl}(v)) c_{ikl}(v) = G_{ij}(c_{jkl}(v))c_{ijl}(v).
			\label{eq:c-gen}
		\end{equation}
		In this paper, we always set $G_{ii}=\Id, c_{jjk}\equiv 1 \equiv c_{jkk}$.
	\end{enumerate}
	
\end{defn}

In particular, we can define an affine chart of a quiver algebra using quiver algebroid stack:
\begin{defn} \label{def:chart}
	An affine chart of a quiver algebra $\A$ is $$\left(\cA=k Q'/R',G_{01},G_{10}\right)$$
	where $Q'$ is a quiver with a single vertex and $R'$ is a two-sided ideal of relations; $$G_{01}: \cA \to \A_{\mathrm{loc}} \textrm{ and } G_{10}: \A_{\mathrm{loc}} \to \cA$$ 
	are representations that satisfy
	\begin{align*}
		G_{10} \circ G_{01} =& \mathrm{Id};\\
		G_{01} \circ G_{10}(a) =& c(h(a)) \,a \,c(t(a))^{-1} 
	\end{align*}
	for some function $c: I \to (\A_{\mathrm{loc}})^{\times}$
	that satisfies $c(v) \in e_{v_0} \cdot \A_{\mathrm{loc}} \cdot e_{v}$, where $v_0$ denotes the image vertex of $G_{01}$.  Here, $\A_{\mathrm{loc}}$ is a localization of $\A$ at certain arrows and $e_v$ denotes the trivial path at the vertex $v$.
\end{defn}

The quiver algebroid stack arises naturally when we want to glue the noncommutative deformation space of a collection of immersed Lagrangians using (pre)isomorphism pairs.

\begin{defn}
	Let $(L_1,\bb_1)$, $(L_2,\bb_2)$ be two Lagrangian immersions. An isomorphism pair is a pair of degree 0 intersection points $(\alpha,\beta)$, where $\alpha \in CF^0((L_1,\bb_1),(L_2,\bb_2)), \beta \in CF^0((L_2,\bb_2),(L_1,\bb_1)),$ satisfying 
	\begin{center}
		$m_1^{\bb_1,\bb_2}(\alpha)=m_1^{\bb_2,\bb_1}(\beta)=0$, \quad
		$m_2^{\bb_1,\bb_2,\bb_1}(\alpha,\beta)=\one_{L_1}$, \quad $m_2^{\bb_2,\bb_1,\bb_2}(\beta,\alpha)=\one_{L_2}.$
	\end{center} 
\end{defn}

Let $\cL_1,\cdots,\cL_n$ be compact spin oriented immersed Lagrangians. We denote their nc unobstructed deformations by $\cA_i$. Now we can transform an nc family of objects $(\cL_j,\bb_j)$.  Let's define
\begin{equation} 
	\mathscr{F}^{(\cL_i,\bb_i)}((\cL_j,\bb_j)) := \left(\cA_i \otimes_{(\Lambda^\oplus)_i} \CF^\bullet(\cL_i,\cL_j) \otimes_{(\Lambda^\oplus)_j} \cA_j^\op, d = (-1)^{|\cdot|}m_1^{\bb_i,\bb_j}(\cdot)\right).
\end{equation}
For an algebra $\A$, recall that $\A^\op$ is the opposite algebra which is the same as $\A$ as a set (and the corresponding elements are denoted as $a^\op$), with multiplication $a^\op b^\op := (ba)^\op$. The concatenation is read from right to left with $h(a^\op)=t(a).$
$\mathscr{F}^{(\cL_i,\bb_i)}((\cL_j,\bb_j)) $ is a (graded) $(\cA_i,\cA_j)$-bimodule, where the right $\cA_j$-module structure on $\cA_j^\op$ is by taking
$ a^\op \cdot b := (ab)^\op = b^\op a^\op$.  The tensor product over $(\Lambda^\oplus)_i$ and $(\Lambda^\oplus)_j$ means that an element $ a_i X a_j^\op$ is non-zero only if $t(X)=t(a_i)$ and $h(X)=t(a_j^\op)=h(a_j)$.


Indeed, as a generalization of Step (6) in localized mirror construction, we shall extend the whole Fukaya category over
$$T(\cA_1,\ldots,\cA_n):=\widehat{\bigoplus_{m\ge 0}} \bigoplus_{|I|=m} (\cA_{i_0}\otimes\cdots\otimes\cA_{i_m})
$$ 
that is understood as a product of the deformation spaces. $\widehat{\bigoplus}_{k\geq 0}$ means that we take the completion over $\Lambda_0$, meaning that we allow infinite series with valuation in $\Lambda$ increasing to infinity. One can extend the Fukaya category over $T(\cA_1,\ldots,\cA_n)$ in an analogue way. Namely, the space of Floer chains and $A_\infty$ operations have been extended over $T(\cA_1,\ldots,\cA_n)$.  For two Lagrangians $L_0,L_1$ that are not any of these $\cL_i$'s, the morphism space is $T(\cA_1,\ldots,\cA_n)\otimes\CF(L_0,L_1)$.  The morphism spaces involving $(\cL_i,\bb_i)$ are extended as $(\cA_j\otimes T(\cA_1,\ldots,\cA_n) \otimes \cA_i)  \otimes_{(\Lambda_0^\oplus)_i \otimes (\Lambda_0^\oplus)_j} \CF^\bullet(\cL_i,\cL_j),$ $T(\cA_1,\ldots,\cA_n)\otimes \cA_i \otimes_{(\Lambda_0^\oplus)_i} \CF^\bullet(\cL_i,L)$, and $\cA_i \otimes T(\cA_1,\ldots,\cA_n) \otimes_{(\Lambda_0^\oplus)_i} \CF^\bullet(L,\cL_i)$.  All coefficients can be pulled to the left. More details and discussions about the coefficients can be found in the Chapter three of \cite{LNT23}.

Given a collection of Lagrangian immersions, quiver algebroid stack naturally emerges as a geometric mirror as shown in \cite{LNT23}. 
\begin{thm}[\cite{LNT23}]
	Given symplectic manifold $M$ and a collection of immersed Lagrangians $\cL=\{\cL_0, \cdots, \cL_N \}$. Suppose there exist isomorphism pairs among these Lagrangians. Then
	there exists an $A_\infty$ functor $$\cF^\cL:\Fuk(M) \to \Tw(\cX),$$ where $\cX$ is a quiver stack constructed by gluing the noncommutative deformation spaces of $\cL$, $\Tw(\cX)$ is the dg category of twisted complexes over $\cX$.
\end{thm}

Furthermore, there exists a universal twisted complex $\mathbb{U}$, which induces a functor $\cF^{\U}:=\U \otimes -: \mathrm{dg \,\A- mod} \xrightarrow{} \Tw(\cX)$. In some interesting cases, $\U$ corresponds to the universal bundle over the moduli space of stable $\A$-module. And it enables us to compare these two $A_\infty$ functors.

\section{Framed Fukaya category}\label{Sec:Frcat}

\subsection{Framed symplectic manifold}
Let $(M,\omega)$ be a (non-compact) convex symplectic manifold of dimension $2n$ (see for instance \cite[(7b)]{Sei08} in the exact setting, and \cite[Section 3]{RS2017} in the monotone setup). In other words, $M$ is assumed to be conical at infinity: there exists a compact subset $M_c \subset M$ such that its complement can be identified with an infinite cone by a symplectomorphism 
\begin{align*}
	\phi: \left(M\setminus M_{c}^{\circ},\omega\right)  \cong  \left(\partial M_{c}\times [1,\infty), d(r\alpha)\right)
\end{align*}
where $M_{c}^{\circ}$ is the interior of $M_{c}$, $(\partial M_{c},\alpha)$ is a contact manifold and $r$ is the coordinate of $[1,\infty)$.  Moreover,
\begin{enumerate}
	\item The Liouville vector field $Z$ characterized by $\omega(Z,\cdot) = \theta$ points strictly outwards along $\partial M_{c}$. 
	
	\item $\phi$ is induced by the flow of $Z$ for time $\log r$.

	\item Outside the compact set $M_{c}$, the symplectic form is exact: $\omega = d\theta$ where $\theta = \phi^*(r\alpha)$.  

\end{enumerate}

A conical almost complex structure $J$ is taken, which is an $\omega$-compatible almost complex structure on $M$ satisfying $J^*\theta = dr$ for $r$ large enough. 

We are going to define framed Lagrangian immersions, which consist of compact Lagrangian immersions, together with a  non-compact Lagrangian submanifold which we call the framing, with fixed gauge and formal perturbations.

\begin{defn}[Framing Lagrangian] \label{def:framing} A framing Lagrangian of a convex symplectic manifold $M$  is a Lagrangian $F\subset M$ diffeomorphic to $\mathbb{R}^n$  equipped with a Morse function $f_F$ with exactly one maximal point and no other critical point.

Moreover, we require the framing is conical which means $F$ is of the form $F = F_{c} \cup_{\partial F_{c}}(\partial F_{c}\times [1,\infty))$ where 
\begin{enumerate}
    \item $F_{c}=F\cap M_c$ intersects $\partial M_{c}$ transversally, 
    \item $\theta|_{F-F_{c}}$ vanishes near the boundary $\partial F_{c} = \partial M_{c}\cap F_{c}.$  In particular, the complement $F-F_c$ is an exact Lagrangian with a constant primitive function.
\end{enumerate} 
\end{defn}

\begin{defn}
    A framed convex symplectic manifold is a tuple $(M,\bF)$ where $M$ is a convex symplectic manifold and $\bF$ is a collection of framing Lagrangians that are disjoint with each other.  We denote by $|\bF|$ the cardinal number of the framing collection. 
    \end{defn}



\begin{example} \label{exmp:plumbing}
	Let's recall the  plumbing  construction of an exact symplectic manifold (see for example \cite[(2.3)]{MR2786590}), which is automatically a Liouville domain, whose completion is a convex symplectic manifold. 
	It is obtained from gluing by a symplectomorphism between two disk bundles $D^*S_{1}$ and $D^*S_{2}$ over two spheres that identifies the fiber and base in opposite directions.

Let $M_1,M_2$ be Riemannian manifolds, and we consider the total spaces of their disk cotangent bundles $D^*M_1, D^*M_2$.  Consider open balls $B_i \subset D_i$ such that the disk bundles $D^*B_i \subset D^*M_i$ for $i=1,2$ are symplectomorphic under  $(x_{1},y_{1})  \mapsto (-y_{2},x_{2})$, where $x_i,y_i$ denote the base and fiber coordinates respectively.
The glued manifold 
can be completed to a Liouville manifold (after rounding off the corner).  Moreover, $M_{i}$ are are embedded exact Lagrangian submanifolds. 
In the resulting Liouville manifold, we can take a cotangent fiber of $M_i$ to be a framing Lagrangian. 
\end{example}

\begin{example} \label{exmp:conic}
Consider the self-plumbing of $\bS^2$, which is obtained by symplectic gluing of the total space $T^*\bS^2$ along the disc bundles $D^*B_i$ for open balls around the north and south poles similar to the above example.  The glued manifold is completed to a Liouville manifold.  The zero-section becomes an immersed Lagrangian, and we can take a cotangent fiber of a point in $\bS^2$ away from the north and south poles to be a framing $F$ for the space.
	
By the Weinstein neighborhood theorem, this can be identified as a neighborhood of the immersed sphere constructed by $\bS^1$-reduction as follows, which has an advantage of possessing a Lagrangian torus fibration.  
Let $\epsilon \neq 0$ in $\C$ and consider the conic fibration $\pi:X:=\{(a,b)\in \C^2: ab \neq \epsilon\} \to \C \setminus \{\epsilon\}$ defined by $(a,b) \mapsto ab$. This fibration admits a Hamiltonian $S^1$-action on $X$ given by $$S^1 \times X \to X, \quad (\theta, (a,b)) \mapsto (e^{i\theta}a, e^{-i\theta}b),$$ whose moment map is $\mu(a,b)=\frac{1}{2}(|a|^2-|b|^2).$  We have the Lagrangian torus fibration $(\mu,\log |ab-\epsilon|): X \to \R^2$.  

The fiber at $\mu=0, \log |ab-\epsilon|=\log |\epsilon|$ is an immersed sphere.  Moreover, 
under such an identification, the framing $F$ is given as a section of the Lagrangian fibration.  In particular, it intersects with the immersed sphere at a single point.



\end{example}


\subsection{Framed Fukaya category} 
\label{sec:framedfukayacategory}
\begin{defn}[Framed Lagrangian immersions]
	\label{defn:FramedLag}
	Let $(M,\bF)$ be a framed convex symplectic manifold.  Let $L \subset M$ be a compact Lagrangian immersion that satisfies the following conditions. 
	\begin{enumerate}
		\item $L \cap \bF \not= \emptyset$ and $L$ intersects $\bF$ transversely.
		\item $L \cap \bF$ does not contain the maximum point of the Morse function $f_\bF$ on $\bF$.
	\end{enumerate}
	Without loss of generality, we can just assume $L \subset M_c$.  The union $L \cup \bF$ is called a framed Lagrangian immersion.
\end{defn}

The above conditions on $L$ can be easily achieved by a Hamiltonian perturbation.  Next, we follow the arguments of \cite[Section 7]{abouzaid_open_2010} to make sure that the moduli space of pseudo-holomorphic discs bounded by $L\cup \bF$ is compact by using the convexity assumption.  
Without loss of generality, $L$ is assumed to be contained in $M_c$, and hence 
all the intersections $L \cap \bF$ are contained in $M_c$.

\begin{lem}[Convexity argument]

Let $(M,\bF)$ be a framed convex symplectic  manifold, and $L$ a compact Lagrangian immersion.  All the pseudo-holomorphic curves $u:(\mathbb{S},\partial \mathbb{S})\rightarrow (M,L\cup \bF)$ satisfying $u(\partial \mathbb{S})\subset L\cup \bF$ must be contained in $M_c$.
\end{lem}

\begin{proof}
Assume there is such a non-constant pseudo-holomorphic curve $u$ that intersects $\partial M_c\times(1,\infty)$. Choose a regular value $1+\epsilon$ of $r\circ u$ and let $\Sigma$ be the preimage of $[1+\epsilon,\infty)$ under $r\circ u$. 
The boundary $\partial \Sigma$ is stratified into  $\partial_{F}\Sigma$ and $\partial_{in}\Sigma$ which lie in $\bF$ and the interior of the Riemann surface $\bS$ respectively. For convenience, we can additionally assume the primitives over $F-F_c$ are $0$ for all framings $F$.  Then we can rewrite the  energy of $u|_{\Sigma}$ as
\begin{align*} \int_{\Sigma}u^*\omega & = \int_{\partial \Sigma }\theta\circ(du)\\
	& = \int_{\partial_{in}\Sigma} \theta\circ (du) = \int_{\partial_{in}\Sigma}-d(r\circ u)\circ j  
\end{align*}
where the last equality follows from the contact-type condition of $J$ (and $j$ is the complex structure on the domain $\bS$). Now $j(\xi)$ points inwards along the $\partial_{in} \Sigma$ where $\xi$ is tangent to $\partial_{in} \Sigma $ and respecting its orientation. Take $\epsilon\rightarrow 0$, the energy of $u|_{\Sigma}$ is non-positive which is impossible. Therefore we conclude the image of $u$ must be contained in $M_c$. 
\end{proof}

The above ensured the compactness of moduli spaces of pseudoholomorphic discs, and so we can define Fukaya category and Floer cohomologies.
There are different choices of models for chains, including the de Rham model, singular chain model and Morse model.  In this paper, we use the Morse model \cite{BC-pearl,BC12,OZ11} for clean intersections, which counts pearl trajectories, namely certain unions (fiber products) of gradient flow trajectories and pseudo-holomorphic discs.  \cite{FOOO-can} constructed a Morse model via homotopy from their singular chain model that incorporated with their virtual perturbation techniques.  See also \cite[Section 2]{KLZ} for a review and the construction of a homotopy unit.

Consider $(L_0,\ldots,L_k)$, where $L_i$ are either a component of $\bF$ or a compact Lagrangian immersion satisfying the conditions in Definition \ref{defn:FramedLag}.  When $L_i$ is not intersecting transversely with $L_{i+1}$, we need to choose a perturbation of $L_{i+1}$.  This can occur when $L_i,L_{i+1}$ are both compact Lagrangian immersions or both are the same framing Lagrangian $\bF$.  When the intersection is clean (in particular for the case that $L_i=L_{i+1}$), we can take a Morse perturbation: namely we choose a Morse function on the clean intersection and take its critical points as Floer generators.  For the case that $L_i,L_{i+1}$ are both compact Lagrangian immersions, we additionally require that the Morse function (on the clean intersection of $L_i$ and $L_{i+1}$) does not have a critical point in $L_i \cap L_{i+1} \cap \bF$ (which can be empty).  When the intersection is not clean, one can take a Hamiltonian perturbation on $L_{i+1}$.  We additionally require that the Hamiltonian perturbation preserves $\bF$ and $L_i \cap \phi(L_{i+1}) \cap \bF = \emptyset$.

For the case $L_i=L_{i+1} = F$, we always use the prefixed Morse function $f_F$ as the perturbation, which has a unique degree-zero critical point that represents the unit $\one_F$.  After fixing these perturbations between $L_i$ and $L_{i+1}$ (starting with $i=0$ and proceeding to $i=k-1$), we can then consider the moduli spaces of pearl trajectories bounded by $(L_0,\ldots,L_k)$ and their virtual perturbations to define the $m_k$ operations.  

As in \cite{Sei08} and \cite{FOOO09}, the perturbations have to be carried out inductively on $(L_0,\ldots,L_k)$ and the disc classes.  Namely, the codimension-one boundaries of disc moduli are fiber products of moduli spaces of discs that have lower energies or fewer numbers of Lagrangian boundary labels.  The Hamiltonian and virtual perturbations on disc moduli need to be chosen inductively on $(E,k)$ (where $E$ stands for energy and $k+1$ is the number of boundary labels) to be compatible with those chosen on boundaries.  This gives the $m_k$ operations for framed Lagrangian immersions that satisfy the $A_\infty$ equations.

To extend the use of formal perturbations and Morse theory, let's first enlarge the category by equipping a function on every compact Lagrangian immersion, whose differential gives a formal Hamiltonian perturbation to the Lagrangian.  For a pair $((L,f),(L,g))$ where $f$ and $g$ are functions equipped on a Lagrangian $L$ and $(g-f)$ is Morse, we take the critical points and gradient flow trajectories of $(g-f)$ in the definition of Floer theory (for every tuple of Lagrangians that have $((F,f),(F,g))$ as adjacent pairs).  If $(g-f)$ is not Morse, we further take a slight perturbation $g+\epsilon h$ such that $(g+\epsilon h-f)$ is Morse.  This is regarded as a formal perturbation on $(L,g)$.  Proceeding in this way, we have $A_\infty$ operations on Lagrangians equipped with (possibly zero) functions.

On a framing Lagrangian $F$, we consider either $(F,0)$ or $(F,f_F)$ where $f_F$ is the prefixed Morse function on $F$.  $(F,f_F)$ will be used as a part of the reference Lagrangian later.  In the consideration of Floer theory, we have the pairs $((F,0),(F,0))$, $((F,f_F),(F,f_F))$, $((F,0),(F,f_F))$, and $((F,f_F),(F,0))$.  For the first two cases, we have clean intersections and we proceed like before by taking a formal perturbation by $f_F$ on the second Lagrangian.  In the first three cases, we take $f_F$ as the governing Morse function which has a unique maximum as the degree zero critical point.  On the other hand, for the last case, we take $0-f_F = -f_F$ as the governing Morse function which has a unique minimum as the degree $n$ critical point.  Note that the minimum point can only appear as an output of a pearl trajectory.

\begin{rem}
  A rich and powerful method to study non-compact Lagrangians is provided by Wrapped or partially wrapped Floer theory \cite{NZ09,abouzaid_open_2010,GPS2,GPS1,GPS3}. However,  for the purpose of mirror symmetry for framed torsion-free sheaves over a compactification, we treat the corresponding non-compact Lagrangians in a similar way to the compact ones without wrapping.  In particular, they have a finite-dimensional Floer cohomology.  For consistency with our Morse theoretical treatment for immersed Lagrangians, we will fix a Morse function on the non-compact Lagrangian component, and use pearl trajectory models \cite{OZ11,BC12} instead of actual Hamiltonian perturbations.
\end{rem}

\subsection{Formal families of framed Lagrangian branes}
\label{sec:Lagrangianbranes}
Let's fix a reference framed Lagrangian immersion $\bL\cup (\bF,-f_\bF)$.  We use this to construct a mirror functor from the framed Fukaya category to the derived category of framed quiver representations.

We carry out the quiver construction of \cite{CHL21} for $\bL \cup \bF$ to give a framed quiver $Q^\fr$.  In this setup, the reference framed Lagrangian has more than one component.  The moduli space of its deformations is a stacky quotient by a non-compact Lie group.  Quiver path algebra provides a canonical way to encode the stacky moduli.

For convenience of descriptions, we shall assume $\bL$ and $\bF$ are graded and relatively spin, so that the associated Floer cochains have integer degrees and the $m_k$ operations have well-defined signs respectively.  Without these assumptions, Floer cochains have degrees in $\Z_2$, and $m_k$ are defined up to signs respectively.  We will also take canonical models, namely we take representatives of a basis of the Floer cohomology of $\bL \cup \bF$.

In other words, each framing Lagrangian in $\bF$ (which is disjoint to each other) corresponds to a framing vertex, and each component of $\bL$ corresponds to a vertex of the quiver.  Each degree one Floer generator between different components of $\bL \cup \bF$ corresponds to an arrow of $Q^\fr$ between corresponding vertices.  On each component of $\bL$, we fix a basis of its degree-one cohomology and associate self loops $z,w$ with relation $zw=1$ at the corresponding vertex to each basic element.  (This means these are invertible self loops.)  These self loops encode the homolomies around corresponding loops for families of flat connections.

Let $\Lambda Q^\fr$ denote the path algebra of $Q^\fr$ over $\Lambda$.  We take 
$$\CF(\bL \cup \bF,\bL \cup \bF)_{\Lambda Q^{fr}} := \Lambda Q^{fr} \otimes \CF(\bL \cup \bF,\bL \cup \bF)$$ 
and extend the $A_\infty$ operations $m_k$ on it.  In particular, let $\bb = \sum_a a \otimes X_a \in \CF(\bL \cup \bF,\bL \cup \bF)_{\Lambda_0 Q^{fr}}$, where the sum is over the arrows $a$ corresponding to jumps between corresponding Lagrangian branches, and consider $m_0^{\bb,\nabla} := \sum_{k \geq 0} m_k^\nabla(\bb,\ldots,\bb)$ for the family of flat connections $\nabla$ parametrized by the self loops $z$ introduced above, which is the obstruction of defining the Floer cohomology.

\begin{defn} \label{def:univ}
	Let $\A^\fr := \Lambda_0 Q^\fr / I$ where $I$ is the two-sided ideal generated by the coefficients of $m_0^{\bb,\nabla}$.  $(\bL \cup \bF,\A^\fr)$ is called to be a universal family of framed Lagrangian branes supported on $\bL \cup \bF$.
\end{defn}

\begin{rem}
	In above, we have assumed $\bL \cup \bF$ has a $\Z$-grading and $m_0^{\bb}$ has degree two.  In general, $I$ is the ideal generated by the coefficients of $m_0^{\bb,\nabla}$ mod out the span of the unit $1_{\bL \cup \bF}$ (which is the sum of the unit on each component of $\bL \cup \bF$). The coefficient $W$ of $1_{\bL \cup \bF}$ in $m_0^{\bb,\nabla}$ is called the disc potential.  $(\A^\fr,W)$ is called to be a Landau-Ginzburg model.
\end{rem}

Below, we will define general families of framed Lagrangian branes which involve a choice of dimension vectors.  In such a general formulation, a Lagrangian brane may not have commutative unobstructed deformations, and we have to consider noncommutative families.

First, we fix dimension vectors $\vec{w}$ and $\vec{v}$ for the components of the framing Lagrangian $\bF$ and 
 for the components of $\bL$ respectively.  We equip the framing Lagrangian $\bF$ with the trivial vector bundle in the rank specified by $\vec{w}$ on each component.  Moreover, the components of $\bL$ are decorated with local systems $\mathcal{E}$ of corresponding ranks specified by $\vec{v}$. 
We call $\bL\cup \bF$ equipped with these bundles to be  a Lagrangian brane.
Later, the local systems will be specified by the representing matrices of the self loops $z$ introduced above.  $\vec{v}$ has positive entries, while $\vec{w}$ has non-negative entries.  When an entry of $\vec{w}$ is zero, we delete the corresponding component of $\bF$ in the union $\bL\cup\bF$.  In particular, we can consider unframed Lagrangian branes by setting $\vec{w}=0$.

The morphisms between  $(\bF,\overrightarrow{v} ,\overrightarrow{w})$-framed Lagrangian brane $L$ and $(\bF,\overrightarrow{v}' ,\overrightarrow{w}')$-framed Lagrangian brane $L'$ is required to satisfy the following condition: suppose $p$ is an intersection point between $F_{i}$ and $F_{i}'$, which is represented by a maximal point of the Morse function,  then $\gamma \in Hom(\mathcal{F}_{i}|_{p},\mathcal{F}_{i}'|_{p})$ will not only preserve the basis but also preserve the order. This implies if $\gamma$ is nontrivial, then $\overrightarrow{w} = \overrightarrow{w}'$. 



Conceptually, a family of (framed) Lagrangian branes over $\cA$ (or $\cA^\fr$) comes from the pull-back of the universal family by a map between $\cA$ (or $\cA^\fr$) and $\A$ ( or $\A^\fr$ resp.). Such a map is given by a rank $\vec{v}$ (resp. $(\vec{v},\vec{w})$) representation of the quiver algebra $\A$ (resp. $ \A^\fr$) over another quiver algebra $\cA$ (or $\cA^\fr$).  

\begin{defn}[Representation by another quiver algebra]
A representation $G$ in rank $\vec{v}$ of a quiver algebra $\A$ by another quiver algebra $\cA$ consists of an assignation $f:I_{\A}\rightarrow I_{\cA}$, together with a family of maps $g_{h,t}:E_{h,t}\rightarrow Mat_{v_{h} \times v_{t}}(\cA)$ indexed by the ordered pairs $(h,t)\in V\times V$, where $E_{h,t}$ is set of all arrows with head $h$ and tail $t$. Moreover, for an arrow $a\in E_{h,t}$, the entries of $g_{h,t}(a)$ are required to be  arrows from $f(t(a))$ to $f(h(a))$.   The representation $G$ is required to preserve relations of $\A$ and $\cA$.
	
A framed representation $G$ in rank $(\vec{v},\vec{w})$ of a framed quiver algebra $\A^\fr$ by another framed quiver algebra $\cA^\fr$ is a representation in the above sense, with the further condition that $f: I_{\A^\fr} \to I_{\cA^\fr}$ maps a framing vertex to a framing vertex.
\end{defn}

In other words, each vertex $v_i$ of the quiver for $\A$ is mapped to a vertex of $\cA$, and each arrow of $\A$ is represented by a matrix of arrows in $\cA$. This representation is done in a manner that respects head and tail vertices, together with the relations for algebras mod idempotents at vertices; any expression in arrows of $\A$ that equals to zero is still zero in $\cA$ upon substitution according to the representation. 


For simplicity, we will denote the representation defined above by $G:\A \to Mat(\cA)$.  We also think of this representation as defining a family $\A$ over $\cA$.


Now we define a noncommutative family of Lagrangian brane $(\bL,\cE)$.

\begin{defn}
	Let $\cA$ be an algebra (possibly noncommutative). A family of Lagrangian branes $(\bL,\cE)$ over $\cA$ is a representation $G:\A \to Mat(\cA)$ in the above sense, where $\A$ is the universal family associated with $(\bL,\cE)$ (see Definition \ref{def:univ}).  This means there's a function that maps generators of the degree-one cohomology of the normalization of $\bL$ to  elements in $GL(\cA)$, and maps other generators of $CF^1(\bL)$ to elements in $Mat(\cA)$ of suitable ranks. 
\end{defn}

One can extend the Fukaya algebra of $(\bL, \cE)$ over $\cA$ and consider the formal boundary deformations of $(\bL, \cE)$. 
	An $A_\infty$-functor $\cF^{(\bL,\cE)}: \Fuk(X) \to \mathrm{dg} \,\A\lmod$ given by the family $(\bL,\cE)$ over $\cA$ 
(in the $\Z$-graded case) can be obtained by employing constructions similar to those in \cite{CHL21} (by replacing $\A$ with its representation in $Mat(\cA)$). 




One would like to compare the formal deformation spaces of two (framed) Lagrangian immersions, which motivates the notion of a (framed) isomorphism.  First, we define isomorphisms between framed quiver algebras.  Next, we define isomorphisms between formal deformations of framed Lagrangian branes.


\begin{defn}
	Two framed quiver path algebras $\cA_1^{\fr}$ and $\cA_2^{\fr}$ are framed isomorphic if there exist framed representations $G_{21}:\cA_1^{\fr} \to \cA_2^{\fr}$, $G_{12}:\cA_2^{\fr} \to \cA_1^{\fr}$ and gerbe terms $c_{121}$, $c_{212}$, satisfying the following conditions: \begin{enumerate}
		\item $G_{12}$ and $G_{21}$ are mutually inverse up to gerbe terms. More precisely, $G_{ij}\circ G_{ji}\left(a\right)=c_{iji}\left(h(a)\right)\cdot a \cdot c_{iji}^{-1}\left(t(a)\right)$ for $(i,j)=(1,2)$ or $(i,j)=(2,1)$. 
		\item The gerbe terms at the framing vertices are trivial, i.e. $c(v_f)=e_{v_f}$ for all gerbe terms $c$ and framing vertices $v_f$.
	\end{enumerate}
\end{defn}

A framed quiver algebroid stack is locally glued by framed quiver path algebras, which is motivated from taking compactification of the (unframed) quiver algebroid stack.
\begin{defn}
	Let $B$ be a topological space.  A framed quiver algebroid stack consists of the following data:
		\begin{enumerate}
			\item An open cover $\{U_i: i \in I\}$ of $B$.
			\item
			A sheaf of framed quiver algebras $\mathcal{A}^\fr_i$ over each $U_i$, coming from localizations of a quiver algebra $\cA^\fr_i(U_i) = k Q^{(i)}/ R^{(i)}$.
			\item
			A sheaf of framed representations $G_{ij}$ of $\mathcal{A}^\fr_j(V)$ over $\mathcal{A}^\fr_i(V)$ for every $i,j$ and $V \stackrel{\mathrm{open}}{\subset} U_{ij}$.
			\item
			An invertible element $c_{ijk}\left(v\right)\in \left(e_{G_{ij}(G_{jk}(v))}\cdot\mathcal{A}^\fr_{i}(U_{ijk})\cdot e_{G_{ik}(v)}\right)^{\times}$ for every  $i,j,k$ and $v \in Q^{(k)}_0$, that satisfies
			
			$$	G_{ij}\circ G_{jk}\left(a\right)=c_{ijk}\left(h(a)\right)\cdot G_{ik}\left(a\right)\cdot c_{ijk}^{-1}\left(t(a)\right)$$ such that for any $i,j,k,l$ and $v$,
			$$	c_{ijk}(G_{kl}(v)) c_{ikl}(v) = G_{ij}(c_{jkl}(v))c_{ijl}(v).$$
			Furthermore, the gerbe terms at the framing vertices are trivial. In this paper, we always set $G_{ii}=\Id, c_{jjk}\equiv 1 \equiv c_{jkk}$.
			
		\end{enumerate}
\end{defn}

\begin{rem}
	Given any framed quiver algebroid stack $\hat{\cY}^\fr$, there exists an induced unframed quiver stack $\hat{\cY}$, whose local charts and transition maps are the same as those in $\hat{\cY}^\fr$ after setting the arrows whose head or tail vertex is a framing vertex to be zero.
\end{rem}

\begin{defn}
	Let $(L_1^\fr,\bb^\fr_1)$, $(L_2^\fr,\bb^\fr_2)$ be two framed Lagrangian immersions. A pair of degree 0 intersection points $(\alpha^\fr,\beta^\fr)$, where $\alpha^\fr \in CF^0((L_1^\fr,\bb^\fr_1),(L^\fr_2,\bb^\fr_2)), \beta^\fr \in CF^0((L_2^\fr,\bb^\fr_2),(L_1^\fr,\bb^\fr_1)),$ is called a framed isomorphism pair if there exists a framed quiver algebroid stack $\hat{\cY}^\fr$ over which $(\alpha^\fr,\beta^\fr)$ satisfies
	\begin{center}
		$m_1^{\bb^\fr_1,\bb^\fr_2}(\alpha^\fr)=m_1^{\bb^\fr_2,\bb^\fr_1}(\beta^\fr)=0$, \quad
		$m_2^{\bb^\fr_1,\bb^\fr_2,\bb^\fr_1}(\alpha^\fr,\beta^\fr)=\one_{L_1^\fr}$, \quad $m_2^{\bb^\fr_2,\bb^\fr_1,\bb^\fr_2}(\beta^\fr,\alpha^\fr)=\one_{L_2^\fr}.$
	\end{center} 
\end{defn}

In general, the gluing of nc deformation spaces of framed Lagrangian immersions can be restricted to the gluing of compact Lagrangian immersions.

\begin{prop} \label{prop:framediso}
	Let $L_s^\fr$ be framed Lagrangian immersions with nontrivial topology and $\bb_s$ (resp. $\bb_s^\fr$) be the formal boundary deformation of $L_s$ (resp. $L_s^\fr$) for $s=1,2$. Suppose there exists a quiver algebroid stack $\hat{\cY}^\fr$ and an isomorphism pair $(\alpha^\fr,\beta^\fr)$ between $(L_1^\fr,\bb_{1}^\fr)$ and $(L_2^\fr,\bb_{2}^\fr)$ that solves the Fukaya isomorphism equations over $\hat{\cY}^\fr$. Then the restriction of $(\alpha^\fr, \beta^\fr)$ to the compact Lagrangian immersions solves the Fukaya isomorphism equations between $(L_1,\bb_1)$ and $(L_2,\bb_2)$ over the unframed quiver algebroid stack $\hat{\cY}$. Here, $\hat{\cY}$ is a quiver stack induced by $\hat{\cY}^\fr$. 
\end{prop} 

\begin{proof}
	Let $\bF:=\{F_1, \cdots, F_r\}$ be the set of framing Lagrangians and $I_k, J_k$ be the immersed sectors, where $I_k \in \CF^1(F_k, (L_1^\fr,\bb_{1}^\fr))$ and $J_k \in \CF^1((L_1^\fr,\bb_{1}^\fr),F_k)$ for $k=1, \cdots,r$.  By definition, $\bb_1^\fr=\bb_1+ \sum_k i_{k} I_{k} + \sum_k j_{k} J_{k}$, which is similar for $\bb_2^\fr$. Observe that the preisomorphism pair $(\alpha^\fr,\beta^\fr)$ between $(L_1^\fr,\bb_{1}^\fr)$ and $(L_2^\fr,\bb_{2}^\fr)$ is of the following form:
	$$\alpha^\fr=\alpha+ \sum_l \one_{F_l} + \sum_m M_m+ \sum_n N_n, \quad \beta^\fr= \beta + \sum_l \one_{F_l}+ \sum_p P_p+ \sum_q Q_q,$$
	where $\alpha \in \CF^0((L_1,\bb_{1}),(L_2,\bb_2))$, $\beta \in \CF^0((L_2,\bb_{1}),(L_1,\bb_2))$,  $\one_{F_l}$ is the fundamental class of the framing Lagrangian $F_l$, $M_m \in \CF^0((L_1,\bb_{1}),\bF)$, $N_n \in \CF^0(\bF,(L_2,\bb_{2}))$, $P_p\in \CF^0((L_2,\bb_{2}),\bF)$ and $Q_q \in \CF^0(\bF,(L_1,\bb_{1}))$. We want to show $(\alpha,\beta)$ forms an isomorphism pair between $(L_1,\bb_1)$ and $(L_2,\bb_2)$ over $\hat{\cY}^\fr$. 
	
	First, we show $m_1^{\bb_1,\bb_2}(\alpha)=0$. Since $N_n \in \CF^0(\bF,L_2)$, the outputs of $m_1^{\bb_1^\fr,\bb_2^\fr}(N_n)$ at $\CF(L_1,L_2)$ are contributed by the polygons, which have $J_k \in \CF(L_1,\bF)$ as corners. Thus, the coefficients of the outputs are multiples of $j_k$ in $m_1^{\bb_1^\fr,\bb_2^\fr}(N_n)$. The situation is similar for the outputs of $m_1^{\bb_1^\fr,\bb_2^\fr}(\sum_l \one_{F_l} + \sum_m M_m)$ at $\CF(L_1,L_2)$. Therefore, the outputs of $m_1^{\bb_1,\bb_2}(\alpha)$ are the same as those of $m_1^{\bb_1^\fr,\bb_2^\fr}(\alpha^\fr)$ after setting $i_k=j_k=0$ for all $k$. Hence, $m_1^{\bb_1,\bb_2}(\alpha)=0$ over $\hat{\cY}.$ Using the same idea, one can show $m_1^{\bb_{2},\bb_{1}}(\beta)=0$. Furthermore, the outputs at other components have similar properties. One can also obtain $m_1^{\bb_{1},\bb_{2}}(M_m)=0$ and $m_1^{\bb_{2},\bb_{1}}(Q_q)=0$ over $\hat{\cY}$ for all $m$ and $q$.
	
	Next, we consider $m_2^{\bb_1,\bb_2,\bb_1}(\alpha,\beta)$. Since $(\alpha^\fr,\beta^\fr)$ is an isomorphism pair over $\hat{\cY}^\fr$, we have $m_2^{\bb_1^\fr,\bb_2^\fr,\bb_1^\fr}(\alpha^\fr,\beta^\fr)=\sum_j \one_{L_{1j}}+ \sum_l \one_{F_l}.$ Similar to the above argument, it suffices to show $m_2^{\bb_1^\fr,\bb_2^\fr,\bb_1^\fr}(M_m, Q_q)$ is not a multiple of $\one_{L_{1j}}$ for all $m,q,j$.
	
	Suppose $m_2^{\bb_1^\fr,\bb_2^\fr,\bb_1^\fr}(M_m, Q_q)= \one_{L_{1j}}.$ We have $M_m \in \CF^0((L_{1j},\bb_1),F_k)$ and $Q_q \in \CF^0(F_k,(L_{1j},\bb_1))$ for some $k$. Otherwise if the intersections happen in the different components, the coefficient of the output will be a multiple of $i$ and $j$. Notice that $M_m$ and $Q_q$ are closed. The existence of $M_m$ and $Q_q$ induces an injective homomorphism from $HF(L_{1j})$ into $HF(F_q),$ but $F_q$ is diffeomorphic to $\R^n$. This contradicts to the fact that $L_{1j}$ has nontrivial topology. Hence, $m_2^{\bb_1,\bb_2,\bb_1}(\alpha,\beta)=\one_{L_1}$. Similarly, we obtain  $m_2^{\bb_2,\bb_1,\bb_2}(\beta,\alpha)=\one_{L_2}.$ Thus the prposition holds.
\end{proof}

\begin{example}
	Let $L_0$ be the union of two spheres, and $L_i$ be the i-th immersed spheres in affine $A_1$ case for $i=1,2$. Let $\bF=\{F_1, F_2\}$ be the collection of framing Lagrangians and $L_j^\fr$ be the corresponding framed Lagrangian immersions for $j=0,1,2$ as shown in the Figure \ref{fig:aff A1}.  Denote the deformation parameter of $L_j^\fr$ by $\bb_j^\fr$ with formal deformation space $\cA_j^\fr$ for $j=0,1,2$.
	
	\begin{figure}
		[htb!]
		\includegraphics[scale=0.85]{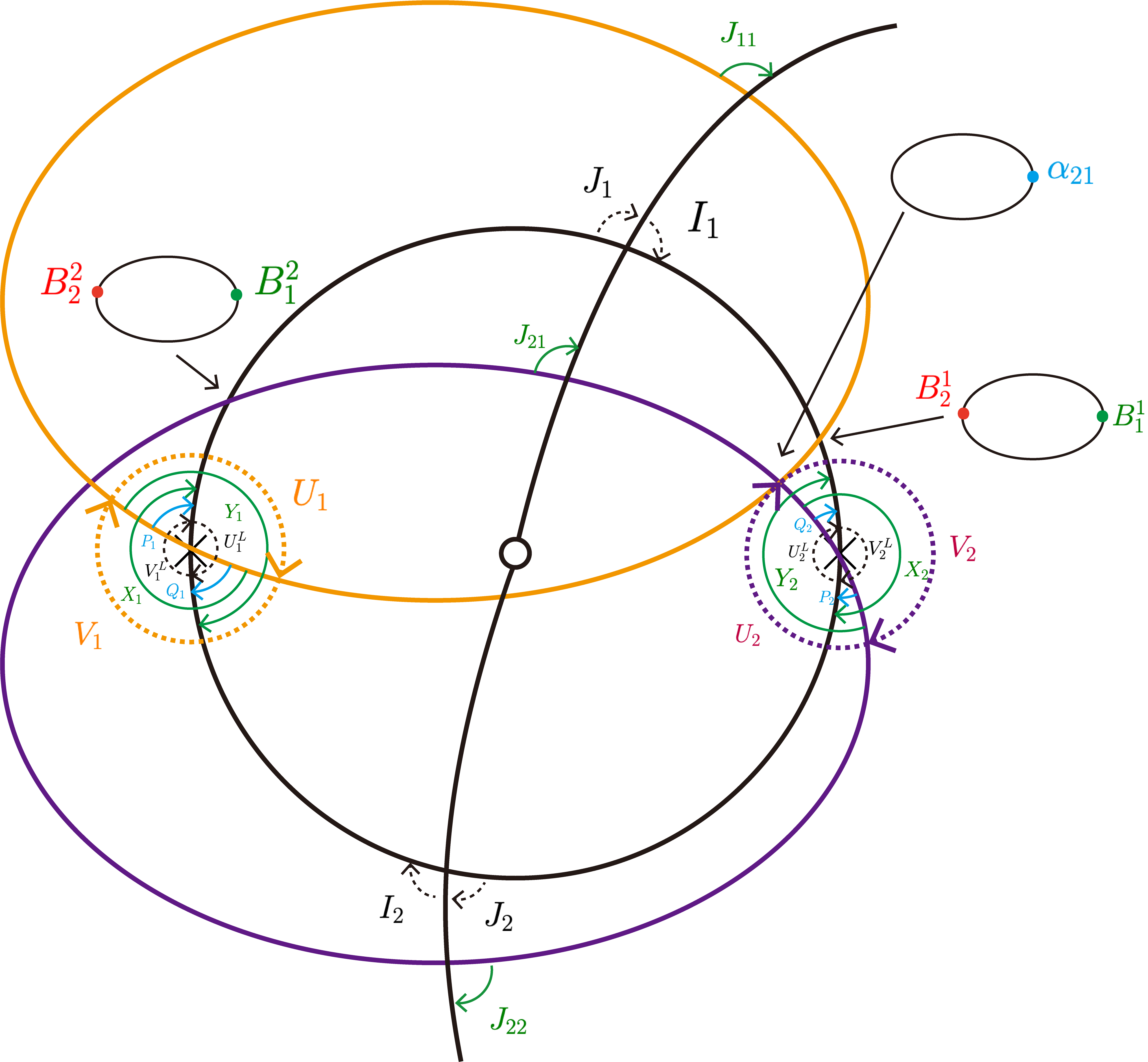}
		\caption[\small Framed Lagrangian immersions in affine $A_1$]{Framed Lagrangian immersions in affine $A_1$. The framed Lagrangian immersions are labeled by $L_1^\fr, L_0^\fr, L_2^\fr$ from top to bottom. The generators of $\CF^1(L_0^\fr,L_0^\fr)$ are $U_i^L \in \CF^1(\bS^2_i, \bS^2_{i+1}), V_i^L \in \CF^1(\bS^2_{i+1}, \bS^2_{i})$, $I_k \in \CF^1(F_k, \bS^2_{k})$ and  $J_k \in \CF^1( \bS^2_{k},F_k)$ for $k=1,2$; the generators of $\CF^1(L_m^\fr,L_m^\fr)$ are the self-immersed sectors $U_m, V_m$ of $L_m$ and immersed sectors $I_{mk}\in \CF^1(F_k,L_m),  J_{mk} \in \CF^1(L_m,F_k)$ for $m,k=1,2$. We denote the maximal point of the clean intersect circle between $L_0$ and $L_k$ for $k=1,2$ (resp. $L_m$ and $L_n$ for $(m,n)=(1,2)$ or $(m,n)=(2,1)$) by $\beta_k^1 \in \CF^0(L_0,L_k)$ (resp. $B_{mn}^1 \in \CF^0(L_m,L_n)$.}
		\label{fig:aff A1}
	\end{figure}

	There exists a framed quiver algebroid stack $\hat{\cY}^\fr$ and framed isomorphism pairs $(\alpha_{ij}^\fr, \alpha_{ji}^\fr)$ that solve Fukaya isomorphism equations over $\hat{\cY}^\fr$. More precisely, we have the following preisomorphism pairs: $$\alpha_{10}^\fr=P_1+ (u_2^L)^{-1}Q_1+\one_{F_1}+\one_{F_2}, \quad \alpha_{01}^\fr=\beta_1^1+\one_{F_1}+\one_{F_2}$$
	$$\alpha_{20}^\fr=(v_1^L)^{-1}P_2+ Q_2+\one_{F_1}+\one_{F_2}, \quad \alpha_{02}^\fr=\beta_2^1+\one_{F_1}+\one_{F_2}$$
	$$\alpha_{12}^\fr= B_{12}^1+\one_{F_1}+\one_{F_2}, \quad \alpha_{21}^\fr=B_{21}^1+\one_{F_1}+\one_{F_2}.$$
	
	We construct the quiver stack $\hat{\cY}^\fr$ as follows:
	\begin{enumerate}
		\item The underlying topological space of $\hat{\cY}^\fr$ is the polyhedral set  $P$ of $K_{\bP^1}$. We endow $P$ with a topology $\left\{\emptyset,\{U_{i}\}_{0\leq i \leq 2}\right\}$  where $U_{i}$ is the complements of the $i$-th facet. 
		\item We associate each open set $U_{i}$ a sheaf which will be obtained by localizing the corresponding global sections $\cA_{0}$ for $i=0$ or $\cA_{i}$ for $1\leq i\leq n$. If there is no ambiguity we will use $\cA_{0}$ or $\cA_{i}$ to denote the sheaves. Then $\cA_0(U_{1})$ (resp. $\cA_0(U_{2})$) is defined by localizing $\cA_0$ at $\{u_2^L\}$ (resp. $\{ v_1^L \}$); on the intersection, it's the algebra localized at the union of corresponding elements.  
		\item The transition maps  $G_{i0}:\cA_{0}|_{U_{0i}}\rightarrow\cA_{i}|_{U_{0i}}$ and $G_{0i}$ for $i=1,2$ are defined by $$G_{10}:= \begin{cases} 
			u_{1}^L\rightarrow  u_1 \\
			v_{1}^L\rightarrow  v_1 \\ 
			u_{2}^L\rightarrow 1  \\
			v_{2}^L\rightarrow v_1u_1+i_{11}j_{11}  \\
			i_k \rightarrow i_{1k} & k=1,2  \\
			j_k \rightarrow  j_{1k} & k=1,2
		\end{cases},\quad G_{01}:=\begin{cases}
			u_{1}\rightarrow u_{2}^L u_{1}^L   \\
			v_{1}\rightarrow v_{1}^L (u_{2}^L)^{-1} \\
			i_{11} \mapsto i_1 \\
			j_{11} \mapsto j_1 \\
			i_{12} \mapsto u_2^L i_2 \\
			j_{12} \mapsto j_2 (u_2^L)^{-1}
		\end{cases}.$$
		
		$$G_{20}:= \begin{cases} 
			u_{2}^L\rightarrow  u_2 \\
			v_{2}^L\rightarrow  v_2 \\ 
			u_{1}^L\rightarrow  u_2v_2-i_{21}j_{21} \\
			v_{1}^L\rightarrow 1  \\
			i_k \rightarrow i_{2k} & k=1,2  \\
			j_k \rightarrow  j_{2k} & k=1,2
		\end{cases},\quad G_{02}:=\begin{cases}
			u_{2}\rightarrow u_{2}^L (v_{1}^L)^{-1}   \\
			v_{2}\rightarrow v_{1}^L v_{2}^L \\
			i_{21} \mapsto i_1 \\
			j_{21} \mapsto j_1 \\
			i_{22} \mapsto v_1^L i_2 \\
			j_{22} \mapsto j_2 (v_1^L)^{-1}
		\end{cases}.$$
		\item The gerbe terms are defined by  $$c_{010}(v_1)= e_1   \quad c_{010}(v_2)=u_2^L \quad c_{020}(v_1)= e_1   \quad c_{020}(v_2)=v_1^L. $$
		
		\item We define the transition maps between the first chart and second chart to be $G_{ik}(x):=G_{i0}\circ G_{0k}(x)$, and the corresponding gerbe terms are trivial. More precisely, \begin{align*}
			G_{21}=G_{20}\circ G_{01}=\begin{cases}
				u_{1}\rightarrow u_2(u_2v_2-i_{21}j_{21}) &  \\
				v_{1}\rightarrow u_2^{-1} \\
				i_{11} \mapsto i_{21} \\
				j_{11} \mapsto j_{21} \\
				i_{12} \mapsto u_2 i_{22} \\
				j_{12} \mapsto j_{22} u_2^{-1} 
			\end{cases}, \quad G_{12}=G_{10}\circ G_{02}=\begin{cases}
				u_{2}\rightarrow v_1^{-1} &  \\
				v_{2}\rightarrow v_1(v_1u_1+i_{11}j_{11}) \\
				i_{21} \mapsto i_{11} \\
				j_{21} \mapsto j_{11} \\
				i_{22} \mapsto v_1 i_{12} \\
				j_{22} \mapsto j_{12} v_1^{-1} 
			\end{cases}.
		\end{align*}
	\end{enumerate}

	The above data indeed defines a quiver algebroid stack $\hat{\cY}^\fr$. We will take $G_{01} \circ G_{10}(u_1^L)$ as an example. The rest is similar: 
	$$G_{01} \circ G_{10}(u_1^L)= G_{01}(u_1)= u_2^Lu_1^L= c_{010}(v_2) \cdot u_1^L \cdot c_{010}(v_1)^{-1}.$$
	
	The framed quiver algebroid stack is constructed by analyzing the coefficients in the Fukaya isomorphism equations.  For example, by direct calculation, we obtain $$m_1^{\bb^\fr}(P_1+ Q_1+\one_{F_1}+\one_{F_2})= (j_1-j_{11})J_{11}+ (u_1^L-u_1)X_1+ (v_1^L-v_1)Y_1+ (u_2^L-1)B_1^1+ (j_2-j_{12})J_{12}+(i_1-i_{11})I_{11}+(i_2-i_{12})I_{12}.$$ Therefore, we construct a framed quiver algebroid stack over which the coefficients vanish. This provides the transition maps $G_{10}$ and $G_{01}$ in $\hat{\cY}^\fr$. And $m_1^{\bb^\fr} (\alpha_{10}^\fr)$ vanishes over $\hat{\cY}^\fr$. Similarly, one can check that $\alpha_{ij}$ satisfies the Fukaya isomorphism equations over $\hat{\cY}^\fr$. Furthermore, we also have $m_2^{\bb_i,\bb_j,\bb_k}(\alpha_{ij},\alpha_{jk})=\alpha_{ik}$ for $i,j,k=0,1,2$.
	
	By Proposition \ref{prop:framediso}, the framed isomorphism restricts to an isomorphism pair between compact Lagrangian immersions. The resulted unframed quiver algebroid stack $\hat{\cY}$ is the quiver algebroid stack corresponding to minimal resolution of $A_1$ singularity. Readers can find more details about the quiver stack corresponds to $A_n$ resolution in Theorem \ref{thm:qstack}.
	
\end{example}

\subsection{Localization and charts in general ranks}\label{sec:h-rank}

Given a Lagrangian brane $(\bL,\cE)$. The noncommutative deformation space is in general stacky. In order to analyze the local structure of the nc deformation space, we generalize the localized path algebra. We define the localization of a path algebra at a matrix of arrows. We denote a path algebra with relations by $\A=k Q/R$. In this case, even though a single path may not be invertible, a matrix of arrows can possess both left and right inverse. 

\begin{defn}[Higher rank localization of a quiver algebra] \label{def: loc2}
	Let $S=(s_{ij})_{1 \leq i \leq m, 1 \leq j \leq n}$ be a matrix of paths satisfying the following condition:
	\begin{enumerate}
		\item The paths in the same row (column) have the same heads (tails). 
		\item Let $v$ be a column of paths such that the head of $i$-th entry is the  tail of $i$-th column of $S$, then $Sv= 0$ implies $v=0$. The similar condition holds for rows of paths $w$ and $wS=0$. 
	\end{enumerate}
	Then we define the localized algebra at $S$ in the following way:
	
	For each entry $s_{ij}$, we add one path, denoted by $\gamma_{ji}$, with $h(\gamma_{ji})=t(s_{ij})$ and $t(\gamma_{ji})=h(s_{ij})$, to the quiver $Q$. This forms a matrix of paths $\gamma$. Moreover, we take the ideal $\hat{R}$ generated by $R$ and $\left(\gamma S - diag(e_{t(s_{11})}, \cdots, e_{t(s_{1n})}), S \gamma - diag(e_{h(s_{11})}, \cdots, e_{h(s_{m1})})\right)$ to be the new ideal of relations.  The new quiver algebra with relations $k \hat{Q} / \hat{R}$ is called the localized algebra at $S$, and is denoted as $S^{-1}\A$.
\end{defn}
 
\begin{rem}
	Definition \ref{def: loc1} is a special case of Definition \ref{def: loc2} by taking $S$ to be a diagonal matrix of arrows.
\end{rem}

Similarly to definition \ref{def:chart}, we can define an affine chart of rank $\vec{w}$ for a path algebra.

\begin{defn} \label{def: affine}
	Let $\A$ be a quiver algebra with a reference vertex $v_1$ and $\vec{w}$ be a dimension vector which equals $1$ at $v_1$.  An affine chart of rank $\vec{w}$ centered at $v_{1}$ is a triple $$\left(\cA_i=k Q_i/R_i,G_{0i},G_{i0}\right)$$
	where $Q_{i}$ is a quiver with a single vertex and $R_{i}$ is a two-sided ideal of relations; $$G_{0i}: \cA_i \to S^{-1}\A \textrm{ and } G_{i0}: S^{-1}\A \to Mat(\cA_i)$$ 
	are representations with $v_1$ being the image vertex of $G_{0i}$ that satisfy 
	\begin{align*}
		G_{i0} \circ G_{0i} =& \mathrm{Id},\\
		G_{0i} \circ G_{i0}(a) =& c(h_a) \,a \,c(t_a)^{-1}.
	\end{align*}
	The map $c$ assigns each vertex $v$ an invertible matrix of paths (which we call Gerbe terms) in $S^{-1}\A$  with the properties:
	\begin{enumerate}
		\item each entry is a path from $v$ to $v_1$,
		\item the number of rows equals the rank of $v$.
	\end{enumerate}
	Here, $S^{-1}\A$ is the localization of $\A$ at some matrix of paths $S$ as described in Definition \ref{def: loc2}.
\end{defn}

Let's give an example of affine chart of rank $\vec{w}$ for a quiver algebra $\A$.

\begin{example} 
\label{eg: D4}
Let $Q$ be the quiver of the  doubling of affine $D_4$ as shown in the Figure \ref{fig:affD_4} (a). We consider the path algebra $\cA_{0}$ of $Q$ restricted to the complex moment map level zero, in other words, the elements in $\cA_{0}$ satisfy $\sum_{i=1}^{4} a_i b^i=0$ and $b^ia_i=0$ for $i=1,\cdots, 4$. In the following, we will write down an affine local chart for $\cA_{0}$ explicitly, whose dimension vector $\vec{w}$ is $(2, 1,1,1,1)$ where the order of components agree with subscripts of vertices.

\begin{center}
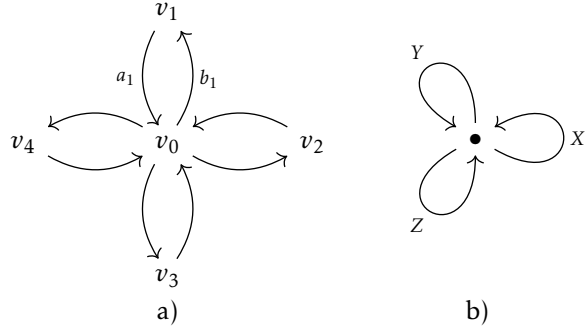

	\begin{tabular}{cc}
		\begin{tikzcd}[row sep=huge,column sep=large]
			& v_{1} \arrow[d, "a_{1}"', bend right] &                    \\
	v_{4} \arrow[r,  bend right] &  v_{0} \arrow[u, "b_{1}"', bend right] \arrow[r, bend right] \arrow[d, bend right] \arrow[l, bend right] & v_{2} \arrow[l, bend right] \\& v_{3} \arrow[u, bend right]&                                      
	\end{tikzcd} &
	\begin{tikzcd}[row sep=huge]
		\bullet \arrow["X"', loop, distance=4em, in=30, out=330] \arrow["Y"', loop, distance=4em, in=150, out=90] \arrow["Z"', loop, distance=4em, in=270, out=210]
		\end{tikzcd} \\
	 a) & b)\\
	 \end{tabular}
	\captionof{figure}{Double quiver of affine $D_{4}$ and the underlying quiver of $\cA_{2}$}
	\label{fig:affD_4}
	\end{center}

	We localize $\cA_{0}$ at the matrix of arrows $\begin{pmatrix}
		a_1 & a_2 
	\end{pmatrix}$, that is, we insert the reverse arrows $\alpha^1$ and $\alpha^2$ so that \begin{equation} \label{eq:inv1}
		\begin{pmatrix}
			a_1 & a_2
		\end{pmatrix} \cdot \begin{pmatrix}
			\alpha^1 \\
			\alpha^2
		\end{pmatrix} =e_0,
	\end{equation}
	\begin{equation} \label{eq:inv2}
		\begin{pmatrix}
			\alpha^1 \\
			\alpha^2
		\end{pmatrix} \cdot \begin{pmatrix}
			a_1 & a_2
		\end{pmatrix}= \begin{pmatrix}
			e_1 & 0\\
			0 & e_2\\
		\end{pmatrix}.
	\end{equation}
	We then localize it at $diag(b^2a_1,b^3a_1,b^4a_1$), the above construction is equivalent to localize $\cA_{0}$ at a block matrix $S$ and we will denote it by $S^{-1}\cA_0$. 
	
	Let $\cA_2$ be the quiver algebra with relations $\C \langle X,Y,Z \rangle/([X,Y], [X,Z],[Y,Z], Z(X+1)-XY)$ whose  underlying quiver is depicted in  Figure \ref{fig:affD_4} (b). Thus, $\cA_2 \cong \C[X,Y,Z]/(XY-Z(X+1)).$ We define the following representations: 
	$$G_{02}: \cA_2 \to S^{-1}\cA_{0}, \quad  G_{02}:=
	\begin{cases}
		X \mapsto (b^2a_1)^{-1}(\alpha^2a_3)(b^3a_1) \\
		Y \mapsto (b^3a_1)^{-1}(b^3a_2)(b^2a_1) \\
		Z \mapsto (b^4a_1)^{-1}(b^4a_2)(b^2a_1) 
	\end{cases}.$$
    This is a well-defined representation. In other words, it represents the relation by zero. Let's show $G_{02}(XY)=G_{02}(YX)$:
    \begin{align*}
    	G_{02}(X)G_{02}(Y)&= (b^2a_1)^{-1}(\alpha^2a_3)(b^3a_2)(b^2a_1)\\
    	&=- (b^2a_1)^{-1}\alpha^2(a_1b^1+ a_2b^2+ a_4b^4)a_2(b^2a_1)\\
    	&=- (b^2a_1)^{-1}(\alpha^2a_4b^4a_2b^2a_1) \\
    	&=  (b^2a_1)^{-1}(\alpha^2a_4b^4(a_1b^1+ a_3b^3+ a_4b^4)a_1)= (b^2a_1)^{-1}(\alpha^2a_4b^4a_3b^3a_1)\\
    	&=-(b^2a_1)^{-1}(\alpha^2 a_2b^2a_3b^3a_1)\\ &= -(b^2a_1)^{-1}(b^2a_3)(b^3a_1).
    \end{align*}
    In the second line, we use $\sum_{i=1}^{4} a_ib^i=0$. In the third line, we use $b^2a_2=0$ and the Equation \eqref{eq:inv2}: $\alpha^2 a_1=0$. Similarly, we get the fourth and fifth line. Finally, we apply  $\alpha^2 a_2 = e_2$ to obtain the last line.
    \begin{align*}
    	G_{02}(Y)G_{02}(X)&= (b^3a_1)^{-1}(b^3a_2)(\alpha^2a_3)(b^3a_1) \\
    	&= (b^3a_1)^{-1}b^3(e_5-a_1 \alpha^1)a_3(b^3a_1) \\
    	&= - (b^3a_1)^{-1}(b^3a_1\alpha^1 a_3)(b^3a_1)= -(\alpha^1 a_3)(b^3a_1).
    \end{align*}
    In the second line, we use the Equation \eqref{eq:inv1}: $a_1 \alpha^1+ a_2 \alpha^2 = e_5$. The last line comes from $b^3 a_3=0$. Notice that using the same relation \begin{align*}
    	&-(\alpha^1 a_3)(b^3a_1)=- (b^2a_1)^{-1}(b^2a_1)(\alpha^1 a_3)(b^3a_1) \\
    	=& -(b^2a_1)^{-1}b^2(e_5- a_2 \alpha^2)a_3b^3a_1 = -(b^2a_1)^{-1}b^2a_3b^3a_1.
    \end{align*} This shows $G_{02}(X)G_{02}(Y)= G_{02}(Y)G_{02}(X)$. One can check $G_{02}$ respects other relations by similar calculations. Hence, $G_{02}$ is a well-defined representation.

    Conversely, we define $G_{20}: S^{-1}\cA_{0} \to Mat(\cA_2)$ by first assigning $G_{20}(a_1):=\begin{psmallmatrix}
    	1 \\
    	0
    \end{psmallmatrix}$ and $G_{20}(a_2):= \begin{psmallmatrix}
    	0 \\
    	1
    \end{psmallmatrix}$. Using the complex moment map equation, we can solve the other terms:  $$G_{20}: S^{-1}\cA_{0} \to Mat(\cA_2), \quad G_{20}:=
    \begin{cases}
    	a_1 \mapsto \begin{psmallmatrix}
    		1 \\
    		0
    	\end{psmallmatrix} \\
    	a_2 \mapsto \begin{psmallmatrix}
    		0 \\
    		1
    	\end{psmallmatrix} \\
    	a_3 \mapsto \begin{psmallmatrix}
    		-YX \\
    		X
    	\end{psmallmatrix} \\
    	a_4 \mapsto \begin{psmallmatrix}
    		-YX \\
    		-X-1
    	\end{psmallmatrix} \\
    	b^1 \mapsto \begin{psmallmatrix}
    		0 & YXY-ZXY
    	\end{psmallmatrix} \\
    	b^2 \mapsto \begin{psmallmatrix}
    		1 & 0
    	\end{psmallmatrix} \\
    	b^3 \mapsto \begin{psmallmatrix}
    		1 & Y
    	\end{psmallmatrix} \\
    	b^4 \mapsto \begin{psmallmatrix}
    		1 & Z
    	\end{psmallmatrix}
    \end{cases}.$$
    
    Additionally, we have the following unavoidable gerbe terms: $c(v_1)=e_{1}, c(v_2)=(b^2a_1)^{-1}, c(v_3)= (b^3a_1)^{-1}, c(v_4)= (b^4a_1)^{-1}, c(v_0)=\begin{psmallmatrix}
    	\alpha^1 \\
    	(b^2a_1)^{-1} \alpha^2
    \end{psmallmatrix}.$ 
    
    One can check that the representations $G_{20}$ and $G_{02}$ are mutually inverse up to gerbe terms. We will take $a_3$ as an example. The other terms can be verified similarly. 
    $$G_{02} \circ G_{20}(a_3)=G_{02}(\begin{psmallmatrix}
    	-YX \\
    	X
    \end{psmallmatrix})= \begin{psmallmatrix}
    	-(b^3a_1)^{-1}(b^3a_2\alpha_2 a_3b^3a_1) \\ (b^2a_1)^{-1}(\alpha_2a_3b^3a_1)
    \end{psmallmatrix}= \begin{psmallmatrix}
    \alpha_1a_3b^3a_1 \\ (b^2a_1)^{-1}(\alpha_2a_3b^3a_1)
\end{psmallmatrix}, $$ where we use $\alpha_2a_2=e_5- \alpha_1 a_1$ and $b^3a_3=0$.
    Meanwhile, $$c(v_0) \cdot a_3 \cdot (c(v_3))^{-1}= \begin{psmallmatrix}
    	 \alpha_1 \\
    	  (b^2 a_1)^{-1} \alpha_2
    \end{psmallmatrix} \cdot a_3 \cdot b^3a_1. $$
    
    

Therefore, $\cA_2$ is a commutative affine chart of rank $\vec{w}$. In fact, $\cA_2$ is the affine 2-space $\A^2_\C$, since $\cA_2= \C[X,Y,Z]/(XY-(X+1)Z) \cong \C[X,T],$ where $T:=Y-Z.$ 
	
	In the Proposition \ref{Prop: D4}, we will explain how to construct the affine chart in rank $\vec{w}$ of $\cA_0$ using stability condition. 
	
\end{example}

\subsection{$\zeta_\R$-stable Lagrangian} \label{sec:stable}
In this subsection, we first recall the stability condition of quiver representations, and then define $\zeta_\R$-stable Lagrangians. This definition is motivated by the relations between (formal) deformation spaces of a Lagrangian immersion and quiver representations.
 
Let's fix notations for quiver representations. Recall that we denote a quiver by $Q=(I,E)$ where $I$ (resp. $E$) denotes the set of vertices (resp. arrows). Let $V=\sum_{i\in I} V_i$ be an $I$-graded vector space, we define its dimension vector by $\dim V:= (\dim V_i)_{i \in I} \in \Z^I_{\geq 0}$. 
Fix a dimension vector $\dim V$ and a weight $\zeta_{\R}$  with respect to $\dim V$, quiver representations possess explicit stability conditions (Definition \ref{def:quiverstability}), see for instance \cite{Kin94} and \cite{Nak07}.

 We start this by recalling the stability condition of quiver representations.

 If $V^1, V^2$ are $I$-graded vector spaces, we introduce vector spaces $$L(V^1, V^2):=\oplus_{i \in I}\Hom(V^1_i, V^2_i)$$ which is the space of linear maps between the vector spaces over the same vertices. We also introduce $$E(V^1, V^2):=\oplus_{a \in E} \Hom(V^1_{t(a)},V^2_{h(a)}).$$

Let $V, W$ be $I$-graded vector spaces. We define $$M(V,W):=E(V,V) \oplus L(V,W) \oplus L(W,V),$$ 
which is the representation space of the framed quiver $Q^{\textrm{fr}}$.

We now define the stability conditions. For $\zeta_\R= (\zeta_{\R,i})_{i \in I} \in R^I,$ let $\zeta_\R \cdot \dim V:= \sum_i \zeta_{\R,i} \dim V_i.$
\begin{defn}
	\label{def:quiverstability}
	A point $(B,a,b) \in M$ is $\zeta_\R$-semistable if the following two conditions are satisfied:
	\begin{enumerate}
		\item If an $I$-graded subspace $S$ of $V$ is contained in $Ker\, b$ and $B$-invariant, then $\zeta_\R \cdot \dim S \leq 0.$
		\item If an $I$-graded subspace $T$ of $V$ contains in $Im\, a$  is $B$-invariant, then $\zeta_\R \cdot \dim T \leq \zeta_\R \cdot \dim V.$
	\end{enumerate}
	We say $(B,a,b)$ is $\zeta_\R$-stable if the strict inequalities hold in $1,2$ unless $S=0$, $T=V$ respectively.
\end{defn}

When $W=0$ i.e. $B\in E(V,V),$ the stability conditions are slightly different.

\begin{defn}
	Suppose that $\zeta_\R \cdot \dim V=0.$ A point $B \in E(V,V)$ is $\zeta_\R$-semistable if it satisfies the following condition: if an $I$-graded subspace $S$ of V  is $B-$invariant, then $\zeta_\R \cdot \dim S \leq 0.$ 
	
	A point $B$ is $\zeta_\R$-stable if the strict equality holds unless $S=0$ or $S=V.$
\end{defn}

This definition coincides with the case $W\neq 0$ for $\zeta_\R$-semistability, but not for $\zeta_\R$-stability.

Let $(L^\fr,\tilde{\cE_1})$ and $(L^\fr,\tilde{\cE_2})$ be two families of Lagrangian branes over a $\C$-algebra $\cA$. Denote the compact components of $L^\fr$ by $\cL_j$ and framing components by $F_j$, for $j=1, \cdots, n$. We also denote $$L:=\cup \cL_j \quad F:=\cup F_j \quad \cE_i:= \tilde{\cE_i}|_L \quad w_i:= \tilde{\cE_i}|_F,$$ for $i=1,2$. The Floer complex $CF^\bullet((L^\fr,\tilde{\cE_1}),(L^\fr,\tilde{\cE_2}))$ has the following interesting components. In particular, these provide a nice decomposition of the Floer complex of the framed Lagrangian branes obtained by  the plumbing.

First, we introduce the following notion to represent certain degree $i$-th Floer complex of each corresponding compact component:

$$L^i((L,\cE_1),(L,\cE_2)):= \bigoplus_j CF^i((\cL_{j},\cE_{1}), (\cL_{j},\cE_{2}))=\bigoplus_j \bigoplus_{p \in \mathcal{X}^i(\cL_{j},\cL_{j})} \cA \otimes \Hom_\C(\cE_1|_p, \cE_2|_p),$$where $ \mathcal{X}^i(\cL_j,\cL_j)$ denotes the degree $i$ generators of the Floer complex. Since we use the Morse model for clean intersections, the generators are critical points of the Morse functions.



Besides, the framing Lagrangian $F$ intersects transversally with $L$. These transversally intersection points are also of interest. Similarly, we denote 
$$L^i((L,\cE_1),(F,w_2))= \bigoplus_{j \in I} CF^i((\cL_{j},\cE_{1}), (F_{j},w_{2}))$$
$$L^i((F,w_1),(L,\cE_2))= \bigoplus_{j \in I} CF^i( (F_{j},w_{1}),(\cL_{j},\cE_{2})).$$

Degree $1$ intersection points are special, which are used for localized mirror constructions. We introduce
$$E((L^\fr,\tilde{\cE}_1),(L^\fr,\tilde{\cE}_2)):= \bigoplus_{p \in E}  \cA \otimes \Hom_\C(\tilde{\cE}_1|_p, \tilde{\cE}_2|_p),$$ where $E$ is the set of arrows in the associated quiver $Q$.

In particular, $E((L^\fr,\tilde{\cE}),(L^\fr,\tilde{\cE}))$ is the representation space of the associated quiver $Q$, which is of great interest since it encodes the Maurer-Cartan space of $(L^\fr,\tilde{\cE})$. In the plumbing construction, we show that $CF^1((L^\fr,\tilde{\cE}_1),(L^\fr,\tilde{\cE}_2))=E((L^\fr,\tilde{\cE}),(L^\fr,\tilde{\cE}))= E((L,\cE_1),(L,\cE_2)) \oplus L^1((L,\cE_1),(F,w_2)) \oplus L^1((F,w_1),(L,\cE_2))$ in Section \ref{sec:preproj}.

Now we can define stability conditions. Given an unframed Lagrangian immersion $\cL$ with a rank $\Vec{v}$ trivial vector bundle $\cE$. Denote its associated quiver by $Q$. For $\zeta_\R=(\zeta_{\R,i})_{i \in I} \in \R^{|I|},$ let $\zeta_\R \cdot rank \, \cE:= \sum_{i \in I} \zeta_{\R,i} v_i.$

\begin{defn}
	Given a Lagrangian brane $(\cL,\cE)$ of rank $\Vec{v}$. Suppose that $\zeta_\R \cdot \Vec{v}=0.$ A deformation $B \in MC((\cL,\cE)) \subset E((\cL,\cE),(\cL,\cE))$ is $\zeta_\R-$semistable if the following is satisfied:
	\begin{itemize}
		\item If an $I$-graded trivial subbundle $S$ of $\cE$ is $B$-invariant, then $\zeta_\R \cdot rank \, S \leq 0.$ 
	\end{itemize}
	A deformation $B$ is $\zeta_\R$-stable if the strict inequality holds unless $S=0$ or $S=\cE.$
\end{defn}

The stability condition for the framed Lagrangian immersions is slightly different. Let $(L^\fr,\tilde{\cE})$ be a framed Lagrangian brane of rank $(\vec{v},\vec{w})$.

\begin{defn}
	Given $\zeta_\R= (\zeta_{\R,i})_{i \in I} \in \R^I$, where $I$ is the index of compact components. A deformation $(B,a,b) \in MC((L^\fr,\tilde{\cE})) \subset E((L^\fr,\tilde{\cE}),(L^\fr,\tilde{\cE}))$ is $\zeta_\R-$semistable if the following two conditions are satisfied:
	\begin{enumerate}
		\item If an $I$-graded trivial subbundle $S$ of $\cE$ is contained in $Ker(b)$ and $B$-invariant, then $\zeta_\R \cdot rank \, S \leq 0.$
		\item If an $I$-graded trivial subbundle $T$ of $\cE$  contains $Im(a)$ and is $B$-invariant, then $\zeta_\R \cdot rank \, T \leq \zeta_\R \cdot rank \, \cE.$
	\end{enumerate}
	We say $(B,a,b)$ is $\zeta_\R$-stable deformation if the strict inequalities hold in $(1), (2)$ unless $S=0$, $T=V$ respectively.
\end{defn}

\begin{defn}
	In both cases, a Lagrangian $\cL$ with a $\zeta_\R$-stable deformation is called $\zeta_\R$-stable Lagrangian.
\end{defn}

\begin{rem}
	\begin{enumerate}
		\item  In most of the cases we are interested in, we take $\cA$ to be the complex number $\C$, Novikov ring $\Lambda_0$ or Novikov field $\Lambda$.
		\item We will omit the trivial vector bundles if the trivial bundles are clear in the context.
	\end{enumerate}
\end{rem}

When performing localized mirror construction, the following stability condition is of great interest, which can be achieved by taking certain character $\zeta_{\R}$.

By definition of noncommutative family, if the quiver of $\cA$  has a single vertex, an $\A$-family over $\cA$  rank $\vec{w}$ can be identified with an $(\A, \cA)$-bimodule $\cA^{\oplus (\sum w_i)}$, where the module structure is given by the representation $G: \A \to Mat(\cA)$. If there is a vertex whose rank is one (WLOG, we assume it's $v_{1}$), then $e_{1}(\cA^{\oplus (\sum w_i)})\cong \cA$. Under the above assumptions, we now introduce a special class of representations called stable families, which can be seen as a noncommutative analogue of the stable locus. 

\begin{defn}[Stable family] \label{def:h-rk}
	Let $\mathbb{A}$ be a family over $\cA$ of rank $\Vec{w}$ where the vertex $v_1$ of $\A$ is of rank one and the quiver of $\cA$ has a single vertex. We say a family $\A$ over $\cA$ of rank $\vec{w}$ is stable if there's no proper $\A$-submodule $M$ of $\cA^{\oplus (\sum w_i)}$ such that $e_{1}M = \cA$. 
\end{defn}

\begin{example}[Example of a stable family]
	Let's show the family $S^{-1}\cA_{0}$ over $\cA_{2}$ of rank $\vec{w}$ constructed in example \ref{eg: D4} is indeed stable. 
	
	Let $M$ be a $S^{-1}\cA_{0}$-submodule of $\cA_2^{\oplus 6}$, which satisfies $e_1 M = \cA_2$. Because $b^2a_1$ is invertible in $S^{-1}\cA_{0}$, $e_2 M \cong \cA_2$. Similarly, vertices $v_3$ and $v_4$ also admit an invertible path from the vertex $v_1$, which implies $e_3 M \cong e_4 M \cong \cA_2$. Furthermore, since $M$ is a $S^{-1}\cA_0$-module, it's closed under $a_1$ and $a_2$ action. Thus, $e_0M \cong \cA_2^{\oplus 2}$. Hence, $M \cong \cA_2^{\oplus 6},$ which is stable.
	
\end{example}

\begin{defn}
	Let $(\bL,\cE)$ be a Lagrangian brane of rank $\vec{v}$ where   $\cE|_{\bL_{v_0}}$ has rank 1. A nc family of Lagrangian brane $(\bL,\cE)$ over $\cA$ is called stable if the deformation space $M$ is a stable family over $\cA$. 
\end{defn}

Given a dimension vector $\vec{v}$, we can define a canonical nc family of $\A$-representations of rank $\vec{v}$ as follows.  Let $k$ be a unital algebra. First, we define $\cA'$ to be the free $k$-algebra generated by the variables corresponding to the matrix entries of $\Hom_k(k^{ v_{t(a)}}, k^{ v_{h(a)}})$ for all arrows $a$.  The number of variables is $\sum_{a \in E} v_{t(a)}v_{h(a)}$. $\cA'$ can be understood as a quiver algebra with a single vertex.  If we take quotient of $\cA'$ by commuting relations, we would obtain the coordinate ring of the representation space $\mathrm{Rep}(Q,\vec{v})$. 

For a quiver algebra $\A:= kQ/I= k Q/ \langle r_1, \cdots, r_s \rangle$, 
we define the associative $k$-algebra $\cA:=\cA'/R$, where $R$ is the ideal generated by the induced relations on the matrix entries according to $r_1,\ldots,r_s$.
$\cA$ will be called the noncommutative parameter space corresponding to the subvariety $M(Q,\vec{v}) \subset \mathrm{Rep}(Q,\vec{v})$ of 
$\A$-representations.

Given a stable family of $\A$-representations over $\cA$ in rank $\vec{v}$, we need to suitably localize $\cA$ according to the stability condition (Definition \ref{def:h-rk}) to get stable subfamilies.  We will take the collection of all localizations $\cA_{\mathrm{loc}}$ of $\cA$ (in the sense of Definition \ref{def: loc2}) such that the sub-families of $\A$-representations over $\cA_{\mathrm{loc}}$ are stable.  A sub-collection of them will form stable affine charts by setting the invertible matrices to be identity.

We will carry out this procedure for the affine $D_4$ quiver in Section \ref{sec:D4-stable}.

\section{ADHM quiver as a mirror to the framed immersed two-sphere} \label{sec:ADHM}

In this section, we aim to interpret the ADHM construction as a mirror symmetry phenomenon. Specifically, we will accomplish this by constructing the ADHM quiver representations and corresponding framed sheaves using the framed mirror construction, which is applied to the framed nodal sphere $\bL$ in $\mathbb{C}^2 \setminus \{ab=\epsilon\}.$

Our study of this local model fits well in the Strominger-Yau-Zaslow program \cite{SYZ96}. More broadly, the SYZ program proposes that mirror manifolds can be constructed as dual special Lagrangian torus fibrations.  
A crucial ingredient in the SYZ program is wall-crossing of open Gromov-Witten invariants bounded by Lagrangian torus fibers.  Wall-crossing was extensively studied, see for instance \cite{KS01, GS11, GHK, GHKK, GS21} from the algebraic aspects and \cite{Aur07, CLL12,AAK16} from the symplectic aspects. Additionally, Family Floer theory \cite{Fuk02,Tu,Abouzaid17} applied to smooth SYZ fibers established a more direct connection between wall-crossing and the Fukaya category of Lagrangian submanifolds.

The nodal sphere $\bL$ is a crucial object in the study of mirror symmetry since it provides an important class of singular SYZ fibers. Singular fibers of a SYZ fibration are the sources of walls of holomorphic discs of Maslov-index zero.  Among them, $\bL \times T^{n-2}$ for the nodal sphere $\bL$ (where $n$ is the dimension and $T$ denotes a torus) is the most generic one.

Now let's introduce the notations. As described in Example \ref{exmp:conic}, the zero section gives rise to an immersed sphere $\mathbb{L}$ with two immersed generators $X$ and $Y$ following the Floer theory for immersed Lagrangians developed by \cite{AJ10};  we can take a cotangent fiber of a point away from the north and south poles to be the framing $F$. 

An alternative description is to consider the $S^1$-reduced space of $\mathbb{C}^2 \setminus \{ab=\epsilon\}$ which is a conic fibration over $\mathbb{C}$; $\mathbb{L}$ is the fiber of  $(\mu,\log |ab-\epsilon|): M \to \R^2$ at $(0, \log |\epsilon|)$, and the framing $F$ is given as a section of this fibration. Throughout this section we will primarily use this perspective, only resorting to the plumbing construction for topological considerations.  

We fix a reference framed Lagrangian immersion $\bL\cup (F,-f_F)$ and denote it by $\mathbb{L}^{\mathrm{fr}}$. It's evident that the framing $F$  intersects $\mathbb{L}$ only once.  Any degree-one generators in the reference framed Lagrangian along with their corresponding arrows in the quiver will be labeled with a subscript $0$. See Figure \ref{fig:symplecticreduction}.

\begin{figure}[htb!]
	\includegraphics[scale=0.7]{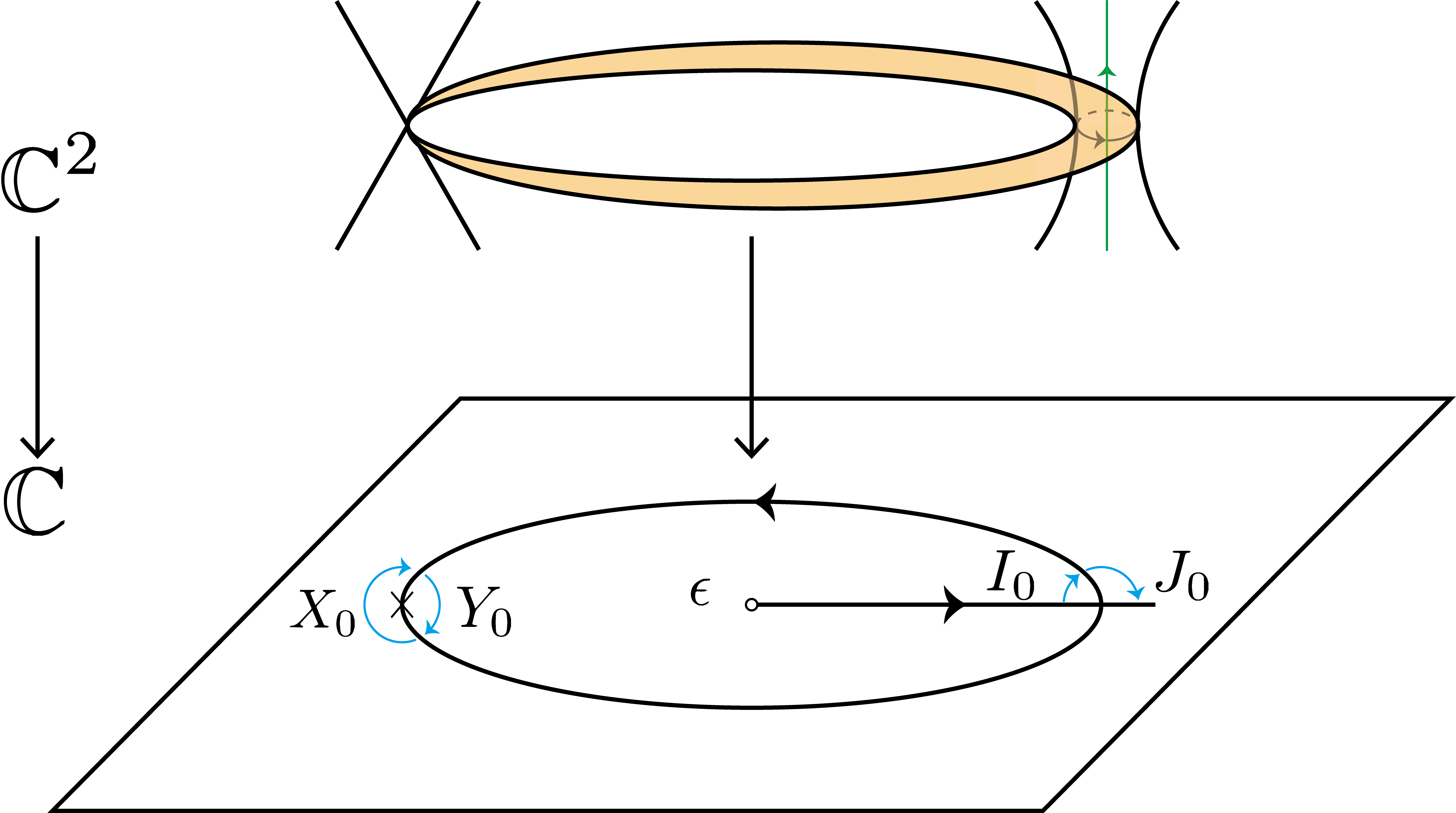}
	\caption[\small Lagrangian torus with orientations]{Lagrangian torus with orientations.}
	\label{fig:symplecticreduction}
\end{figure}

 The associated quiver $Q^{\mathrm{ADHM}}$(resp. $Q$) of $\mathbb{L}^{\mathrm{fr}}$ (resp. $\mathbb{L}$) is shown in Figure \ref{fig:ADHMquiver}.

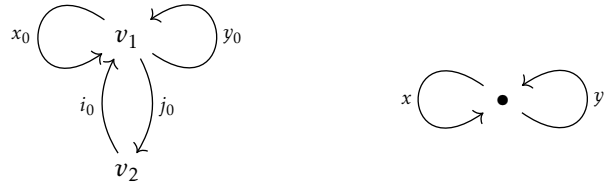
\begin{figure}[hpt!]
	\begin{tikzcd}[row sep=huge]
		v_{1} \arrow[d, "j_{0}", bend left,] \arrow["y_{0}"', loop, distance=4em, in=35, out=325] \arrow["x_{0}"', loop, distance=4em, in=215, out=145] \\ 
		v_{2} \arrow[u, "i_{0}", bend left]  
		\end{tikzcd}
		\hspace{50pt} 
		\begin{tikzcd}
			\bullet \arrow["y"', loop, distance=4em, in=35, out=325] \arrow["x"', loop, distance=4em, in=215, out=145]
			\end{tikzcd}
		\caption[ADHMquiver]{\small $Q^{\mathrm{ADHM}}$ and $Q$ are the quivers related to $\mathbb{L}^{\mathrm{fr}}$ and  $\mathbb{L}$, where $x,y$ correspond to the immersed sectors $X$ and $Y$. } 
		\label{fig:ADHMquiver}
\end{figure}

	The construction of the framed Lagrangian immersion $\bL^{\fr}$ can be intuitively justified as follows.  First, by Kontsevich's Homological mirror symmetry conjecture \cite{Kon95}, coherent sheaves are mirror to Lagrangian submanifolds.   In this case, let's consider rank-one framed torsion-free sheaves over $\C\bP^2$, which correspond to 
ideal sheaves of points on $\C^2$.



On the other hand, by the conjecture of Strominger-Yau-Zaslow  \cite{SYZ96}, the mirror space admits a dual torus fibration. A Lagrangian torus should correspond to a skyscraper sheaf, and a Lagrangian section should correspond to a line bundle over the mirror.  
Combining the pictures, the connected sum of a Lagrangian torus and a Lagrangian section should be mirror to an ideal sheaf of points up to degree. 
Hence, stable deformations of $\bL^{\fr}$ ought to correspond to framed torsion-free sheaves.

More generally, we observe that the quiver associated with $\bL^{\fr}$ coincides precisely with the ADHM quiver.

\subsection{Completion of path algebras}
Because of the existence of constant polygons, in general we have series like $1+a_1xy +a_2(xy)^2+\cdots$, in order to make sense of this, we shall consider the completion of the path algebra.  Recall the valuation of the Novikov ring is defined by: $$\text{val:} \sum\limits_{i=1}^{\infty}a_{i}T^{A_{i}} \mapsto A_{1} \text{ and }\text{val}(0)=+\infty.$$ The valuation induces a non-Archimedean norm on $\Lambda_{0}$ given by $$|c|_{\mathfrak{v}}:= e^{-\text{val}(c)}.$$  

Let $\Lambda_0 Q$  be the free path algebra generated by arrows $x_a$.  Let's fix a function $\mathfrak{v}$ from the arrow set to $[0,1]$, $x_a\mapsto |x_a|_{\mathfrak{v}}\in [0,1]$.  The Gauss norm on $\Lambda_0 Q$ is defined as $$\left\lVert\sum\limits_{I}a_{I}\zeta^I\right\rVert_\mathfrak{v}=\max_{I}\{|a_{I}\zeta^I|_{\mathfrak{v}}\}$$ where $\zeta^I$ is an ordered monomial of $x_a$ and $|\zeta^I|$ is the corresponding product of $|x_a|_{\mathfrak{v}}$. 
One can associate a corresponding valuation on $\Lambda_0 Q$ by setting $\text{val}(z)=-\log \lVert z \rVert$. 

The valuation induces a filtration $$F^{p}\Lambda_0 Q=\{z \in \Lambda_0 Q: |z|_{\mathfrak{v}} \leq e^{-p}\}.$$
In the following, we will always take the completion of $\Lambda_0 Q$ with respect to this filtration and denote it by $\widehat{\Lambda_0 Q}_\mathfrak{v}.$ The above construction makes $\widehat{\Lambda_0 Q}_\mathfrak{v}$ a complete topological space with respect to the Gauss norm. 

When we take the tensor product of two complete path algebras, it's always assumed to be the complete tensor product. 

\subsection{Unobstructedness and deformation space}\label{sec:unobs}
We will utilize linear combinations of immersed generators $X,Y$ to deform the immersed sphere $\bL$ (respectively, $X_{0}$, $Y_{0}$, $I_{0}$, $J_{0}$ for $\mathbb{L}^{\mathrm{fr}}$).

Let $\widehat{\Lambda_{0}Q}_{(a,b)}$ be the completed path algebra of $\Lambda_{0} Q$, where $|x|_{\mathfrak{v}} =a, |y|_{\mathfrak{v}}=b$ for $a,b \in [0,1]$. We take $$\A_{\textrm{free}} := \bigcap_{0\leq a,b \leq 1, 0\leq ab<1}\widehat{\Lambda_{0}Q}_{(a,b)}$$
which contains the series that converge with respect to all the norms satisfying $|x|_{\mathfrak{v}},|y|_{\mathfrak{v}}\leq 1$ and $|xy|_{\mathfrak{v}}<1$.

We rephrase Lemma 3.3 of \cite{HKL23} as the following proposition. 
We find a coordinate transformation that renders $m_0^{\bb}$ into the commutator. 

\begin{prop}\label{lem:HKL20}
	Let $\bb= xX+ yY$.  The obstruction term of $\bL$ equals
	$$m_0^{\bb}=\left(\left(1+ \sum_{i=1}^{\infty}  a_i(xy)^i\right)xy-\left(1+ \sum_{1}^{\infty}  a_i(yx)^i\right) y x\right)P$$
	for some $a_i \in \Q$, where $P$ is the minimum point of the Morse function. There exists a change of coordinates on $(x,y)$ such that the unobstructed deformation space of $\bL$ can be written as
	\begin{equation}
		\A:= \frac{\A_{\text{free}}}{(xy-yx)} \cong \Lambda_0 [x,y][[xy]].
	\end{equation} 
\end{prop} 
\begin{proof}
	The above expression of  $m_{0}^{\bb}$ follows from Lemma 3.3 of \cite{HKL23}.  The coefficients of $xy$ and $yx$ are equal due to the anti-symmetric involution on $\bL$.
	The coefficients $a_{i}$ depend on the choice of Kuranishi structures on the moduli spaces of constant polygons.  This can be rewritten as $$m_0^{\bb}=\left(\left(1+ \sum_{i=1}^{\infty}  a_i(xy)^i\right)x y-y\left(1+ \sum_{1}^{\infty}  a_i(xy)^i\right) x\right)P.$$ We take $(1+ \sum_{i=1}^{\infty} a_i(xy)^i)x$ as a new coordinate, which is still denoted by $x$ by abuse of notation.  $y$ is kept to be the same.  Note that $(1+ \sum_{i=1}^{\infty} a_i(xy)^i)$ is invertible and has norm $1$ in $\A_{\text{free}}$. This is a norm-preserving isomorphism on $\A_{\text{free}}$.
	Then the unobstructed relation becomes $(xy-yx)$, and hence the unobstructed deformation space is $\A:= \frac{\A_{\text{free}}}{(xy-yx)}$.  This is isomorphic to $\Lambda_0 [x,y][[xy]]$ since $|x|_{\mathfrak{v}},|y|_{\mathfrak{v}}\leq 1$ and $|xy|_{\mathfrak{v}}<1$.
	



\end{proof}

\begin{cor}
The immersed $2$-sphere $(\bL, b=xX+yY, \A)$ is unobstructed and its deformation space is commutative. 
\end{cor}

\begin{rem}
	The mirror construction for $\mathbb{C}^2 \setminus \{ab=\epsilon\}$ via wall-crossing has been studied by Auroux \cite{Aur07} using Lagrangian torus fibers. The mirror manifold glued from the deformation spaces of tori is $\{(u,v,z) \in (\C^2 - (0,0)) \times \C^\times \mid uv = 1+z\}.$ The missing point $(u,v)=(0,0)$ that needs to be filled in is dual to the singular fiber $\bL$. In \cite{HKL23}, Hong-Kim-Lau computed Maurer-Cartan deformations of $\bL$, accompanied by a localized mirror functor constructed by \cite{CHL17}. Following the formalism developed in \cite{CHL-glue},  \cite{HKL23} also derived the gluing between the deformation space of the immersed sphere and that of the tori using Floer-theoretical isomorphisms. In this subsection, we rephrase the noncommutative local mirror of $\bL$ following \cite{CHL21}, with a focus on the construction of framed torsion-free sheaves. 
\end{rem}


Now we consider the framed immersed sphere $\bL^\fr$.  Below, we show that the unobstructed relation coincides with the ADHM relation up to a coordinate change.

We will take a similar completion $\A_{\text{free}}^\fr$ of $\Lambda_0 Q^\mathrm{ADHM}$ by taking the intersection over all the valuations $\mathfrak{v}$ that have
$|x_0|_{\mathfrak{v}},|y_0|_{\mathfrak{v}}, |i_0|_{\mathfrak{v}}, |j_0|_{\mathfrak{v}} \leq 1$ and $|x_0y_0|_{\mathfrak{v}}, |i_0j_0|_{\mathfrak{v}}< 1$.

\begin{thm}
	 \label{thm:nc-framed}
	 Let $\bb^\fr:= x_{0}X_{0}+y_{0}Y_{0}+i_{0}I_{0}+j_{0}J_{0}$.  There exists a coordinate change on $\A_{\text{free}}^\fr$ such that 
	 the formal deformation space of the framed Lagrangian immersion $\mathbb{L}^{\mathrm{fr}}$ equals
	\begin{equation}
		\A^\fr:=\frac{\A_{\text{free}}^\fr}{(x_0y_0-y_0x_0+i_0j_0)}.
	\end{equation}
\end{thm}
\begin{proof}
	
	The unobstructed relations are coefficients of $m_0^{\bb^\fr}$ which have degree two. In this case,  the only degree two generator is the minimal point $P$ of the immersed sphere component. Thus, we only have a single relation, which is contributed from pearl trajectories to $P$.  Such a trajectory is a union of a Morse flow line to $P$ and either constant polygons $(x_0y_0)^{k+1}, (y_0x_0)^{k+1}$ or $(i_0j_0)^{k+1}$.  Thus the coefficient of  $m_0^{\bb^\fr}$ equals 
	\begin{equation}
		\left(1+ \sum_{k=1}^{\infty}  a_k(x_0y_0)^k\right)x_{0} y_{0}-\left(1+ \sum_{k=1}^{\infty}  a_k(y_0x_0)^k\right) y_{0} x_{0}+ \left(1+\sum_{k=1}^{\infty}  b_k(i_0j_0)^k \right)i_{0} j_{0},
	\end{equation} 
	for some constants $a_k$ and $b_k$ depending on the perturbation of the Kuranishi structure. This is the same as\begin{align*}
		& \left(1+ \sum_{k=1}^{\infty}  a_k(x_0y_0)^k\right)x_{0} y_{0} - y_0\left(1+ \sum_{k=1}^{\infty}  a_k(x_0y_0)^k\right)x_{0} + \left(1+\sum_{k=1}^{\infty}  b_k(i_0j_0)^k \right)i_{0} j_{0}.
	\end{align*}     
We take $(1+ \sum_{k=1}^{\infty} a_k(x_0y_0)^k)x_0$ (resp. $(1+\sum_{k=1}^{\infty}  b_k(i_0j_0)^k )i_{0}$) as a new coordinate, which is still denoted by $x_0$ (resp. $i_0$) by abuse of notation.  $y_0$ and $j_0$ are kept to be the same.
Then the unobstructed equation becomes the ADHM equation: \begin{align*}
	&x_0y_0- y_0x_0+ i_0j_0=0.
\end{align*}  
Hence, the noncommutative deformation space of $\mathbb{L}^{\fr}$ is $\A^\fr:=\frac{\A_{\text{free}}^\fr}{(x_0y_0-y_0x_0+i_0j_0)}$.
\end{proof}

\begin{cor}
	The framed immersed $2$-sphere $(\bL^{\fr}, \bb^\fr=x_{0}X_{0}+y_{0}Y_{0}+i_{0}I_{0}+j_{0}J_{0}, \A^\fr)$ is unobstructed. 
\end{cor}

\begin{cor} 
	\label{cor:rep}
	Let $\cE_1$ be a trivial vector bundle of rank $k$ over the immersed sphere, and $\cE_2$ a trivial vector bundle of rank $r$ over the framing. We restrict to the subcategory generated by $\bL^\fr$.  The Maurer-Cartan space of the framed Lagrangian brane $(\bL^{\fr},\cE)$ is 
	\begin{equation}
		\mathcal{M}(V,W):=\left\{(B_1,B_2, i, j)\in Rep(Q^\mathrm{ADHM},V,W) \otimes \Lambda_0 \mid \, [B_1,B_2]+ij=0\right\}/GL(k),
	\end{equation}
	where $Q^\mathrm{ADHM}$ is the ADHM quiver, $rank \, V= k$ and $rank \, W=r.$
\end{cor}
\begin{proof}
	By definition, an element $\bb:=B_1X_0+B_2Y_0+iI_0+jJ_0 \in CF^{1}((\bL^{\fr},\cE),(\bL^{\fr},\cE))$ is Maurer-Cartan if $$m_0^{\bb}=\sum m_k(\bb, \cdots,\bb)=0,$$ which counts the same pearl trajectories considered in Theorem \ref{thm:nc-framed}. Notice that for matrices, the map $B_1 \mapsto (1+\sum_{k=1}^{\infty}a_k(B_1B_2)^k )B_1$ is still invertible. Besides, this coordinate change is also gauge equivariant, since $g\cdot (1+\sum_{k=1}^{\infty}a_k(B_1B_2)^k )B_1 \cdot g^{-1}= (1+\sum_{k=1}^{\infty}a_k(g\cdot B_1 \cdot g^{-1} g\cdot B_2 \cdot g^{-1})^k )g \cdot B_1 \cdot g^{-1}.$
	
	Hence, up to coordinate change, the Maurer-Cartan space of $\bL^{\fr}$ is the isomorphic class of representations of $Q^\mathrm{ADHM}$, satisfying  $[B_1,B_2]+ij=0$. Restricting to the subcategory generated by $\bL^\fr$, $\bb$ can take $\C$-valued deformation, since there  is no convergence issue. Thus, the result holds. 
\end{proof}

\subsection{Framed torsion-free sheaves} \label{section:framed sheaves}

By Lemma \ref{lem:HKL20}, the deformation space of $\bL$ forms a formal neighborhood of the zero locus of $xy$ in the affine space $\A^2$, which is the localized mirror. Performing the localized mirror construction as in Section \ref{section:nc mirror}, we obtain an $A_\infty$-functor $\cF^{(\bL,\bb)}: \Fuk(X) \to \mathrm{dg} \, \A\lmod.$ We will restrict $\cF^{(\bL,\bb)}$ to the Fukaya subcategory generated by the framed Lagrangian immersion $\bL^{\fr}$. As a consequence, the mirror sheaves of the Lagrangian submanifolds can be extended to the affine space $\A^2_{\C}$, as $\C[x,y] \subset \A$ and there is no convergence issue after changing variables.  


In the following, we use the functor $\cF^{(\bL,\bb)}$ to transform the framed Lagrangian brane itself into a monad.  We denote the family of objects to be transformed by $(L^{\fr},\cE)$ to distinguish from the reference Lagrangian $\bL$. Here, $L$ is equipped with a trivial rank $n$ vector bundle $\cE_1$ and the framing $F$ is equipped with a Morse function $f_F$ and a trivial rank $r$ vector bundle $\cE_2$.

As described in Section \ref{sec:framedfukayacategory},  the Floer group is defined by counting pearl trajectories. An illustration of perturbed strip is depicted in Figure \ref{fig:noframedcase}, which we will primarily consider in the rest of this section for a better demonstration of strip counting. 


\begin{figure}[htb!]
	\includegraphics[scale=1]{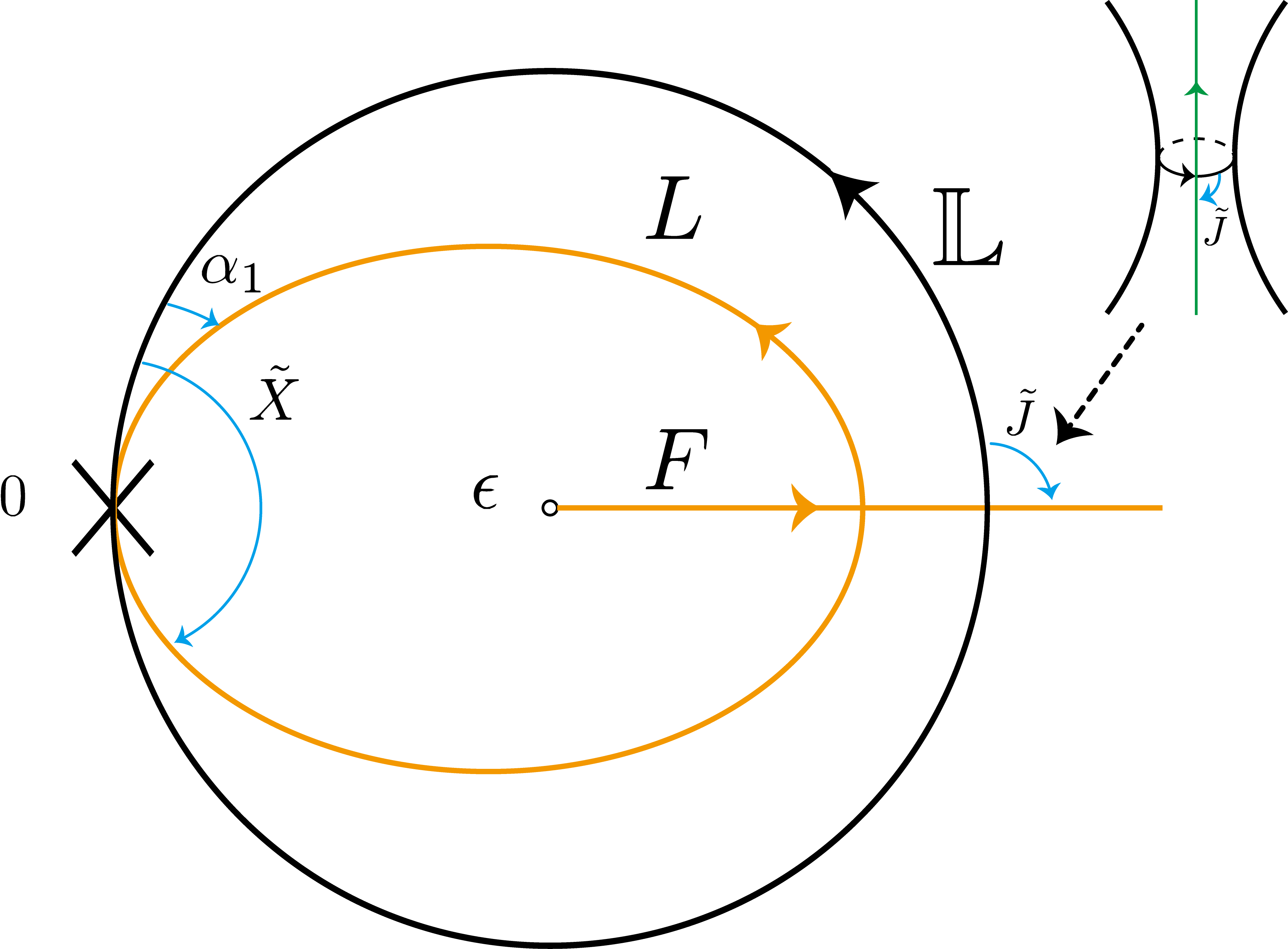}
	\caption[noframecase]{\small The self-immersed generators of $L^\fr$ are labeled in the same pattern as $\bL^\fr$ with subscript $1$. The degree $1$ generator between $\bL$ and framing $F$ is named $\tilde{J}$. The generator $\alpha_{1}$ (resp. $\tilde{X}$ ) corresponds to the branch jump from the upper half branch of $\bL$ to the upper half branch of $L^\fr$ (resp. the lower half branch of $L^\fr$). Similarly, $\alpha_{2}$ and $\tilde{Y}$ are defined as the jumps from the lower half of $\bL$ to the lower half and upper half branches of $L^\fr$ respectively, which are omitted in the picture.}
	\label{fig:noframedcase}
\end{figure}

\begin{thm} \label{thm:frash}
	$\cF^{(\bL,\bb)}$ transforms the framed Lagrangian brane $(L^{\fr},\cE)$ into the monad of the torsion-free sheaves over $\mathbb{C}^2$.  
\end{thm}
\begin{proof}
	The labeling is as shown in Figure \ref{fig:ADHMquiver} and Figure \ref{fig:noframedcase}. The associated quiver of $L^{\fr}$ is exactly the ADHM quiver $Q^{\mathrm{ADHM}}$. The unobstructed equation of $L^{\fr}$ is $[x_{1},y_{1}]+i_{1}j_{1}=0$, as computed in Corollary \ref{cor:rep}. Using the Morse model, $\alpha_1$ (reps. $\alpha_2$) can be identified with the maximal (reps. minimal) point of the Morse function. Consequently, $L^{\fr}$ is transformed by $\bL$ to the following complex:
	\begin{equation} \label{eq:monad}
		(\cF^{(\bL,\bb)}((L^{\fr},\cE)), m_1^{\bb}): \xymatrix{
			  \langle \alpha_1 \rangle \ar[rrr]^-{\begin{psmallmatrix}
					x_1 \cdot - \cdot x & y_1 \cdot - \cdot y &  j_1 \cdot
			\end{psmallmatrix}} &&& 
			\langle \tilde{X}, \tilde{Y}, \tilde{J} \rangle \ar[rr]^-{\begin{psmallmatrix}
					-(y_1 \cdot - \cdot y) \\
					x_1 \cdot - \cdot x \\
					i_1 \cdot
			\end{psmallmatrix}} && \langle \alpha_2 \rangle 
 		},
	\end{equation}
	where the condition that the differential square equals zero can be verified directly using the unobstructedness of the Lagrangians. 
	
	Since $L$ is equipped with the rank $n$ trivial bundle $\cE_1$ and $F$ with a rank $r$ trivial bundle $\cE_2$, the coefficients of the Floer complex are of the following form:
	\begin{equation} \label{eq:monad1}
		\xymatrix{ \Lambda_0^n \otimes \A \langle \alpha_1 \rangle \ar[rrr]^-{\begin{psmallmatrix}
					x_1 \cdot - \cdot x & y_1 \cdot - \cdot y &  j_1 \cdot
			\end{psmallmatrix}} &&& 
			\Lambda_0^n \otimes \A \langle \tilde{X} \rangle \oplus \Lambda_0^n \otimes \A \langle \tilde{Y} \rangle \oplus \Lambda_0^r \otimes \A \langle \tilde{J} \rangle \ar[rr]^-{\begin{psmallmatrix}
					-(y_1 \cdot - \cdot y) \\
					x_1 \cdot - \cdot x \\
					i_1 \cdot
			\end{psmallmatrix}} && \Lambda_0^n \otimes \A \langle \alpha_2 \rangle}.
	\end{equation}
	Notice that the differentials are polynomials and $L^\fr$ is an object in the subcategory that generated by $\bL^\fr$. One can restrict the Complex \ref{eq:monad1} to $\C[x,y]$ and work over the complex coefficients. The resulting complex is the monad of torsion-free sheaves over $\C^2$.
\end{proof}

\begin{cor}
	Given a $\zeta_\R$-stable framed Lagrangian brane $(L^{\fr},\cE)$ over $\C$. The cohomology of $\cF^{(\bL,\bb)}((L^{\fr},\cE))$ is a torsion-free sheaf of rank $r$. Furthermore, the monadic complex $\cF^{(\bL,\bb)}((L^{\fr},\cE))$ can be extended trivially to $\C \mathbb{P}^2$ such that the cohomology  gives a framed torsion-free sheaf. 
\end{cor}
\begin{proof}
	By Corollary \ref{cor:rep}, the Maurer-Cartan (deformation) space of $(L^{\fr},\cE)$ is equivalent to the representation of the ADHM quiver over $\Lambda_0$. Therefore, a $\zeta_\R$-stable framed Lagrangian brane $(L^{\fr},\cE)$ over $\C$ is the Lagrangian $L^{\fr}$ equipped with a stable framed quiver representation $b=(B_1,B_2,i,j) \in Rep(Q^{\mathrm{ADHM}})$.
	
	As demonstrated in the Chapter 2 of \cite{Nak99}, the first differential of the monad is injective at all but finitely many points, while the second differential is surjective for stable framed quiver representations. Consequently, upon restricting to complex coefficients, the cohomology yields a torsion-free sheaf $\cE$ over $\mathbb{A}^2_{\mathbb{C}}$. The second statement follows from homogenizing the monad, which trivially extends $\cE$ to $\mathbb{C}\mathbb{P}^2$.
\end{proof}

\subsection{Framed quiver representation as a mirror} \label{sec:ADHM-fr}

In this section, we consider the reference framed Lagrangian immersion $\mathbb{L}^{\fr}:=\bL\cup (F,-f_F)$ and restrict the noncommutative mirror functor $\cF^{(\bL^\fr,\bb^\fr)}$ to the subcategory generated by $\bL^\fr$. We prove that $\cF^{(\bL^\fr,\bb^\fr)}$ gives a framed representation of ADHM quiver $Q^{\mathrm{ADHM}}$. 


Recall the category of $kQ$-modules and the category of representation of $Q$ are equivalent. Hence, given a framed Lagrangian brane $(L^\fr,\cE)$ over $\C$, one will expect the cohomology of $\cF^{(\bL^\fr,\bb^\fr)}(L^\fr,\cE)$ provides a representation of $Q^{\mathrm{ADHM}}$. Furthermore, this representation should lie in the zero locus of the complex moment map. In other words, it has to satisfy $x_{0} y_{0}-y_{0} x_{0}+i_{0} j_{0}=0$. An explicit calculation shows this is indeed true. Let's denote the Maurer-Cartan space of $(L^\fr,\cE)$ by $\cA$, which is the representation space of $Q^{\mathrm{ADHM}}$ as computed above. 


\begin{figure}[htb!]
	\includegraphics[scale=1]{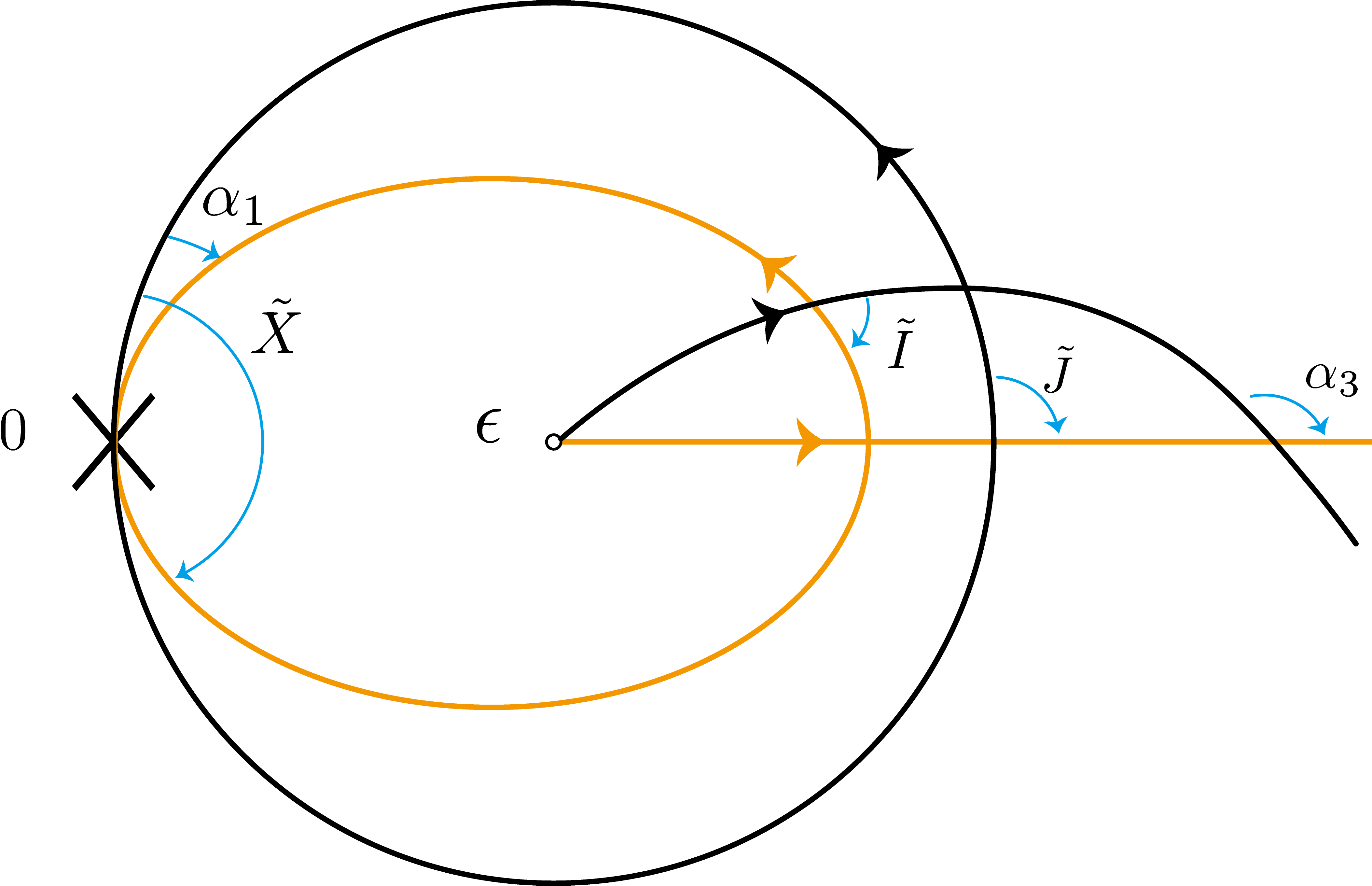}
	\caption[noframecase]{\small The labeling is same as in Figure \ref{fig:noframedcase}, besides we have two additional generators, $\tilde{I}$ and $\alpha_{3}$, arising from the intersection between the framing of $\bL^{\fr}$ and $L^{\fr}$. }
	\label{fig:framedcase}
\end{figure}

\begin{thm} \label{thm:frep}
	The cohomology of $\cF^{(\bL^\fr,\bb^\fr)}((L^\fr,\cE))$ gives a framed quiver representation. 
\end{thm} 
\begin{proof}
	$L^\fr$ is transformed by $\bL^\fr$ into the following complex:
	\begin{equation} 
		(\cF^{(\bL^\fr,\bb^\fr)}((L^\fr,\cE)),m_{1}^{\bb^\fr}): \xymatrix{
			\langle \alpha_1 \rangle \ar[rrr]^-{\begin{psmallmatrix}
			x_1 \cdot - \cdot x_0 & y_1 \cdot - \cdot y_0 & \cdot i_0 & j_1 \cdot
			\end{psmallmatrix}} &&& 
		\langle \tilde{X}, \tilde{Y}, \tilde{I}, \tilde{J} \rangle \ar[rr]^-{\begin{psmallmatrix}
					-(y_1 \cdot - \cdot y_0) & 0\\
					x_1 \cdot - \cdot x_0 & 0\\
					 -\cdot j_0 & j_1 \cdot \\
					i_1 \cdot & -\cdot i_0
		\end{psmallmatrix}} && \langle \alpha_2, \alpha_3 \rangle 
		},
	\end{equation} 
where $(m_1^{\bb^\fr})^2=0$ can be verified directly using unobstructedness of Lagrangians and $A_\infty$-equations. We claim the cohomology concentrates on the highest degree and it provides a framed quiver representation.

In view of the path concatenation relation, we find the only zero divisors are $x_{0}j_{0}=y_{0}j_{0}=0$ and $i_{0}x_{0}= i_{0}y_{0}=0$. Therefore, $ki_{0}=0$  implies that $k=0$, which means the first order differential is injective.  

We now prove the complex is exact at degree $1$. Let $a=(a_{1},a_{2},a_{3},a_{4})\in \ker m_{1}^{\bb^\fr} $  where $ a_{i} \in \cA \otimes \A^\fr$; more precisely, it means $m^{\bb^\fr}_{1}(a_1 \tilde{X}+ a_2 \tilde{Y} + a_3 \tilde{I} + a_4 \tilde{J})= 0$, thus we have the following equations:  
\begin{numcases}{}
	 -(y_{1}a_{1}-a_{1}y_{0})+ x_1a_2 -a_2x_0 - a_3 j_0 + i_1a_4=0 &  \label{equ:part1}
	\\
	j_1a_3-a_4i_0=0 & \label{equ:part2}.
 \end{numcases}
 Notice that by our convention of coefficients (see Section \ref{section:nc mirror}), $i_{0}$ has no left zero divisor in the coefficients of $\tilde{J}$. From the Equation \ref{equ:part2}, we can write $a_3=v+ k i_0, a_4= j_1k$ for some $v, k\in \cA \otimes \A^{\fr}$ such that $v$ doesn't contain a multiple of $i_0$ and $j_1v=0$. Let $a_{1}' = a_{1}-(x_{1}k-kx_{0})$ and $a_{2}'=a_{2}-(y_{1}k-ky_{0})$. Substituting $a_{1}$ and $a_{2}$ in Equation \ref{equ:part1} by $a_{1}'$ and $a_{2}'$ gives rise to the following reduced equation:
 $$\begin{psmallmatrix}
 	-y_1 & x_1
 \end{psmallmatrix} \cdot  \begin{psmallmatrix}
 	a_1' \\
 	a_2'
 \end{psmallmatrix}=\begin{psmallmatrix}
 	-a_1' & a_2'
 \end{psmallmatrix} \cdot \begin{psmallmatrix}
 	y_0 \\
 	x_0 
 \end{psmallmatrix}+vj_0.$$

 It suffices to show $a_1'=a_2'=v=0$. Notice that the relations of $\A^\fr$ preserve the length of the characters. Thus, the length of characters on the right is always one greater than the left. Hence, $a_1'=a_2'=0$, which further implies $v=0$. In conclusion,  $a$ lies in the image class of $k$; the complex is exact at degree $1$ and the cohomology concentrates on the highest degree. In conclusion, $a$ lies in the image class of $k$; the complex is exact at degree $1$ and the cohomology concentrates on the highest degree.

At the highest degree, the cohomology is the cokernel of the differential. Taking the cokernel of the first row $-(y_1 \cdot - \cdot y_0)$ in the differential matrix means we represent the arrow $y_0$ by the complex-valued matrix $y_1$, and similarly for $x_{0}$ and $x_{1}$. Hence, the cohomology is the framed quiver representation $ e_1 \cA \langle \alpha_1 \rangle  \oplus e_2 \cA \langle \alpha_2 \rangle$, where the arrows $y_0,x_0,j_0, i_0$ are represented by the matrices $y_1,x_1,j_1,i_1$ respectively.
\end{proof}

\begin{rem}
	We will focus on the family of framed Lagrangian brane $(L^\fr,\cE)$ over $\C$, but the above result also holds for the family over $\Lambda_0$.
\end{rem}

Since the morphisms between the trivial vector bundles over the Lagrangians are invariant under gauge change (see Proposition \ref{prop:definite}), the stable quiver representation constructed above descends to Nakajima quiver varieties.

\begin{cor}
	The $\zeta_\R$-stable framed Lagrangian branes $(L^\fr,\cE)$ over $\C$ are transformed into Nakajima quiver varieties. Namely, the cohomology of $\cF^{(\mathbb{L}^\fr,\bb^\fr)}((L^\fr,\cE))$ lands in $\mathcal{M}(V,W).$
\end{cor}

Moreover, using the universal twisted complex introduced in \cite{LNT23}, we can compare these constructions systematically. On the objects level, it gives a correspondence between the framed quiver representation and the framed torsion-free sheaves, which is a version of ADHM construction.

Let $\U$ be the universal twisted complex $\U:=(\cF^{(\bL,\bb)}(\bL^\fr,\bb^\fr),m_1^{\bb,\bb^\fr}(-))$. $\U$ induces a dg-functor $\cF^{\U}:= \U \otimes - : \A^\fr\lmod \to \A\lmod.$

\begin{cor} \label{cor:compare}
	$\cF^{\U} \circ \cF^{(\bL^\fr,\bb^\fr)}((L^\fr,\cE))$ is quasi-isomorphic to $\cF^{(\bL,\bb)}((L^\fr,\cE)).$
\end{cor}
\begin{proof}
	For simplicity, we will omit the deformation parameter in the proof. First, using the same method as in theorem \ref{thm:frash}, we get the following complex:	$$(\cF^{\bL}(\bL^\fr), m_1^{\bb,\bb^\fr}): \xymatrix{
		 \langle \alpha_1' \rangle \ar[rrr]^-{\begin{psmallmatrix}
				x_0 \cdot - \cdot x & y_0 \cdot - \cdot y &  j_0 \cdot
		\end{psmallmatrix}} &&& 
		\langle \tilde{X}', \tilde{Y}', \tilde{J}' \rangle \ar[rr]^-{\begin{psmallmatrix}
				-(y_0 \cdot - \cdot y) \\
				x_0 \cdot - \cdot x \\
				i_0 \cdot
		\end{psmallmatrix}} && \langle \alpha_2' \rangle 
	},$$
	
	By definition, $\cF^{\U} \circ \cF^{\mathbb{L}^\fr}(L^\fr)= \cF^{\bL}(\bL^\fr) \otimes \cF^{\bL^\fr}(L^\fr)$, which can be made into a double complex in the classical way. We will use the spectral sequence to compute the total cohomology.
	
	Let's denote the double complex by $C^{\bullet,\bullet}= \oplus_{p,q}C^{p,q}:= \oplus_{p,q} \A \otimes \CF^p(\bL,\bL^\fr) \otimes (\A^\fr)^{op} \otimes_{\A^\fr} \A^\fr \otimes CF^q(\bL^\fr,(L^\fr,\cE)),$ and denote the total complex by $Tot(C)$. Hence, there exists a spectral sequence with the first page $E_1^{p,q}=H^q_{ver}(C^{\bullet,\bullet})$ converging to the total cohomology $H^{p+q}(Tot(C))$.
	
	By Theorem \ref{thm:frep}, the cohomology of $\cF^{\bL^\fr}(L^\fr)$ concentrates on the top degree. Hence, the first page $E_1^{p,q}$ consists of a single row. Since the cohomology of $\cF^{\bL^\fr}(L^\fr)$ gives a representation of the framed quiver, which represents $\A^\fr$ by the $\C$ valued representation, the single row is exactly the monad we obtained in theorem \ref{thm:frash}. Hence, the second page becomes stable, and $\cF^{\U} \circ \cF^{(\mathbb{L}^\fr,\bb^\fr)}((L^\fr,\cE))$ is quasi-isomorphic to $\cF^{(\bL,\bb)}((L^\fr,\cE))$. Moreover, if one extends the complex over $\A^2_\C$, the stable page has only one nontrivial term, which is the framed torsion-free sheaf.
\end{proof}

\section{Nakajima quivers as mirrors to two-dimensional Lagrangian immersions}

A Lagrangian immersion and its deformation theory is governed by a quiver algebra with relations (reviewed in Section \ref{section:nc mirror}).  In this section, we construct the preprojective algebra relations that define the Nakajima quiver varieties by Lagrangian deformation theory of union of (immersed) spheres in dimension two.  Moreover, we transform framed Lagrangian immersions to monadic complexes which are crucial objects in Nakajima's formulation of framed torsion-free sheaves.

\subsection{Warm-up: quiver algebras without relations}
Let us start with complex dimension one.  For each quiver algebra without relation, we can construct a symplectic two-fold together with a Lagrangian immersion whose formal deformation space gives back the quiver algebra.



\begin{prop}
	\label{prop:freequiveralgebras}
	Given a quiver $Q$, there exists a symplectic 2-fold $M$ with a Lagrangian immersion $\mathbb{L}$, such that the associated quiver to $\mathbb{L}$ is $Q$. Furthermore, the formal deformation space of $\bL$ is the free path algebra.
\end{prop}
\begin{proof}
	


	The construction is similar to Example \ref{exmp:plumbing}. To each vertex $i$ we associate the graded Lagrangian $\R \subset \C$ and denote it by $L_i$.  They are plumbed together according to the directed graph $Q$ to produce $\bL$.  Then the degree-one immersed generators are one-to-one corresponding to the arrows of $Q$.  Since there is no degree-two generator, $m_0^{\bb}=0$ and hence the formal deformation space of $\bL$ equals the free path algebra.
\end{proof}



In the next subsection, we carry out the construction in dimension-two and produce a double quiver with relations.

\subsection{Preprojective quiver algebra coming from Lagrangian Floer theory} 
\label{sec:preproj}
	 
In this section, we construct framed quiver algebras with relations by using framed Lagrangian immersions. Given a graph $D$, we show that the (framed) preprojective algebra emerges as the noncommutative deformation space of the (framed) Lagrangian immersion obtained by plumbing of 2-spheres according to $D$.   In particular, we will show that the unobstructed equations for the (framed) Lagrangian immersions in Theorem \ref{thm:preproj} and \ref{prop:nc-framed} give the complex moment map equations for Nakajima quiver varieties.


Let $D:=(I,E)$ be a graph, where $I$ is the set of vertices and $E$ the set of edges. The corresponding double quiver $D^{\#}$ is the directed quiver with the same vertices and the set of arrows (oriented edges) 
\begin{align*}
	\mathfrak{A}:=\left\{(e,o(e))\mid e\in E, o(e)  \text{ is an orientation of } e\right\}.
\end{align*}
Thus, each edge $e$ connecting vertices $v_{i}$ and $v_{j}$ in $D$ gives rises to two arrows $a:v_{i}\rightarrow v_{j}$ and $\bar{a}:v_{j}\rightarrow v_{i}$ in $D^{\#}$. An orientation $\Omega$ of the double quiver is a subset of $\mathfrak{A}$ such that $\mathfrak{A}=\Omega \cup \bar{\Omega}$ and $\Omega \cap \bar{\Omega}=\emptyset$. The orientation defines a function $\epsilon:\mathfrak{A} \to \{\pm 1\}$ given by $\epsilon(a)=1$ if $a \in \Omega$ and $\epsilon(a)=-1$ if $a \in \bar{\Omega}.$

 Similar to Example \ref{exmp:plumbing} and Proposition \ref{prop:freequiveralgebras}, we obtain a Liouville manifold by plumbing 2-spheres according to $D$. Let $\bL$ be the  resulting Lagrangian immersion.  Then the associated quiver of $\bL$ is the double quiver $Q=D^{\#}$ of the graph $D$. Here we choose a perfect Morse function for each sphere, so the only degree-one generators are given by intersections between adjacent spheres, which appear twice for each transverse intersection of Lagrangian two-spheres. 
 
 For each $e\in E$, we label two corresponding degree-one immersed generators  by $X_a$ and $X_{\bar{a}}$, where $X_a \in CF^{1}(S^{2}_{t(a)},S^{2}_{h(a)})$ and $X_{\bar{a}} \in CF^{1}(S^{2}_{h(a)},S^{2}_{t(a)})$. 
 


As in Section \ref{sec:unobs}, we will take the completion of $\Lambda_0 Q$  by taking the intersection over all the valuations $\mathfrak{v}$ that have
$|x_a|_{\mathfrak{v}},|x_{\bar{a}}|_{\mathfrak{v}} \leq 1$ and $|x_{\bar{a}}x_{a}|_{\mathfrak{v}}< 1$, where $x_a$ is an arrow in $Q$ and $x_{\bar{a}}$ is the corresponding arrow with the reverse orientation.  We denote the completed algebra by $\A_{\text{free}}$.

\begin{thm} \label{thm:preproj}
	 There exists an automorphism on $\A_{\text{free}}$ such that the noncommutative deformation space $\A$ of $\bL$ is the preprojective algebra of $Q$, i.e. $$\A= \A_{\text{free}}/J$$ where $J$ is the two-sided ideal generated by $\sum_{t(a)=v} \epsilon(a)x_{\bar{a}}x_{a}$ for $v \in I$ for some assignment of signs $\epsilon$.
\end{thm}
\begin{proof}	
	

	Let $\bb=\sum_{a\in \mathfrak{A}} x_a X_a$. 
	We compute the obstruction term $m_0^{\bb}$ of $\bL$. The orientation and spin structure of $\bL$ correspond to an orientation $\Omega$ of the double quiver $Q$, which defines the function $\epsilon:\mathfrak{A}  \to \{\pm 1\}$. 
	

	
	The outputs of the obstruction term $m_0^{\bb}$ have degree 2. In this case, the degree-two generators are the minimum points $P_v$ of each sphere $S^2_v$. Thus, the relations are contributed from pearl trajectories to $P_v$. Such trajectories are unions of Morse flow lines to $P_v$ and constant polygons $(x_{\bar{a}}x_a)^{k}$ for all $x_a$ with $t(x_a)=v$.  
	Similarly to Theorem \ref{thm:nc-framed}, we obtain  
	\begin{align*}
		m_0^b & =\sum_k m_{2k}\left(\sum_{a} x_a X_a, \sum_{\bar{a}} x_{\bar{a}} X_{\bar{a}}, \cdots, \sum_{a} x_a X_a, \sum_{\bar{a}} x_{\bar{a}} X_{\bar{a}}\right) \\
		& = \sum_v \left(\sum_{t(a)=v} \epsilon(a)x_{\bar{a}}x_{a}\left(1+ \sum_j a_j(x_{\bar{a}}x_{a})^j \right)\right)P_v,
	\end{align*}
  where $a_j$ depends on the Kuranishi perturbation, and $(x_{\bar{a}}x_{a})^j$ is contributed by the constant polygon with corners being the immersed sectors $x_a$ and $x_{\bar{a}}$, ordered counterclockwise.
  
Because of the anti-symmetric involution in a neighborhood of the transverse intersection point, a pair of arrows contributes to the coefficients $x_{\bar{a}}x_{a}(1+ \sum_j a_j(x_{\bar{a}}x_{a})^j)$ at $P_{t(a)}$ and $x_{a}x_{\bar{a}}(1+ \sum_j a_j(x_{a}x_{\bar{a}})^j)= x_a (1+ \sum_j a_j(x_{\bar{a}}x_{a})^j)x_{\bar{a}}$ at $P_{h(a)}$. We define: 
	\begin{equation}
		\label{eq:coord}
		\tilde{x}_a= x_{a}\left(1+ \sum_j a_j(x_{\bar{a}}x_{a})^j\right), \quad \tilde{x}_{\bar{a}}= x_{\bar{a}},
	\end{equation} for all $a$ such that $\epsilon(a)=1$.
	
	We will abuse the notation and replace $\tilde{x}$ by $x$. Thus the noncommutative deformation space of $\bL$ is $\A:= \A_{\text{free}}/(\sum_{t(a)=v} \epsilon(a)x_{\bar{a}}x_{a})$.
\end{proof}


Similarly, we compute the weakly unobstructed equation of a framed Lagrangian immersion $\bL^{\textrm{fr}}:= \bL \bigcup_v F_v$, which is obtained by extending the construction in Example \ref{exmp:plumbing} to $D$. The associated quiver will be the framed quiver $Q^{\fr}$ of $Q$. Let $i_v$ (resp. $j_v$) be the arrow corresponding to the  branch jump $I_v$ from $F_v$ to $S^2_v$ (resp. $J_v$ from $S^2_v$ to $F_v$). We  take the completion $\A_{\text{free}}^\fr$ of $\Lambda_0 Q$ by taking the intersection over all the valuations $\mathfrak{v}$ that satisfy $|x_a|_{\mathfrak{v}},|x_{\bar{a}}|_{\mathfrak{v}}, |i_a|_{\mathfrak{v}}, |j_a|_{\mathfrak{v}} \leq 1$ and $|x_{\bar{a}}x_{a}|_{\mathfrak{v}}, |i_0j_0|_{\mathfrak{v}}< 1$. 

\begin{prop} 
	\label{prop:nc-framed}
	Let $\bb^{\fr}:=\sum_{a\in \mathfrak{A}} x_a X_a + \sum_{v\in I} i_vI_v + \sum_{v\in I} j_vJ_v$. There exists a coordinate change on $\A_{\text{free}}^\fr$ such that 
	the formal deformation space of the framed Lagrangian immersion $\mathbb{L}^{\mathrm{fr}}$ equals  $$\A^{\textrm{fr}}= \A_{\text{free}}^\fr/ J^{\textrm{fr}}, $$  where $J^{\textrm{fr}}$ is the two-sided ideal generated by $(\sum_{t(a)=v} \epsilon(a) x_{\bar{a}} x_{a})_v + (ij)_v =\vec{0}$ for $v \in I$.
\end{prop}
\begin{proof}
	The framing Lagrangians are equipped with a Morse function with exactly one maximal point, which has degree zero. Hence, the obstruction term $m_0^{\bb^\fr}$ only has outputs at the minimum points on each sphere $S_v^2$, for $v \in I$. Once we fix the perturbation of Kuranishi structure, we can change the coordinate of $i_v$ at each vertex consistently. The rest of the proof is similar to Theorem \ref{thm:nc-framed} and Theorem \ref{thm:preproj}.
\end{proof}

Each component $\bL_i$ is equipped with a trivial vector bundle of rank $r_i$.  We take the vector space $M(V,W) \times \prod_i \Hom(\pi_1(\bL_i), \GL(r_i,\C))$ and define a family of Fukaya algebras over it.  Since Floer theory, namely the vector space of chains together with the $m_k$-operators, is equivariant under gauge transformations, the family indeed descends to its quotient by the gauge group $\prod_i \GL(r_i,\C)$, which in this case (with the above preprojective algebra relations) gives a Nakajima quiver variety by choosing a suitable GIT stability condition.  As a consequence, we have a family of stable Lagrangian deformations of $\bL^\fr$ parametrized by a Nakajima quiver variety.

\begin{cor} 
	Let $\bL^\fr$ be a framed Lagrangian immersion and $\cE$ be a trivial vector bundle of rank $\vec{v}$ (resp. rank $\vec{w}$) over the compact components (resp. framing components) of $\bL^\fr$. Then restricting to the subcategory generated by $\bL^\fr$, the Maurer-Cartan space of the framed Lagrangian brane $(\bL^{\fr},\cE)$ is a quotient stack
	\begin{equation}
		\mathcal{M}(V,W):=\left\{(B, i, j)\in Rep(Q^\fr,V,W) \otimes \Lambda_0 \mid \, (\sum\epsilon(a) B_{\bar{a}} B_{a} + ij)_v =\vec{0}  \right\}/GL(\vec{v}).
	\end{equation}
\end{cor}
\begin{proof}
	Notice that the coordinate change is gauge equivariant as in the Corollary \ref{cor:rep}. The remaining parts are similar.
\end{proof}

\subsection{ALE spaces: when the Lagrangian immersions are homological tori}

From the point of view of Lagrangian deformations, an affine ADE quiver is special because the Euler characteristic of the Floer cohomology of its associated Lagrangian brane is semi-positive definite as a  quadratic form on rank.

First, let us deduce a formula for Euler characteristic of the Floer cohomology of a Lagrangian immersion in general dimensions.

\begin{lem}
	The Euler characteristic of the Floer cohomology of a Lagrangian immersion $\bL$ equals
	$$ \chi^{\textrm{LF}}(\bL) = \sum_i \chi(\bL_i) + 2 \sum_{e:\, \textrm{codim}(\bL_e) \textrm{ is even}} (-1)^{\deg u_e}\chi(\bL_e) $$where
	$\bL_i$ denotes the normalization of the $i$-th component of $\bL$, $e$ are the labeling of the clean intersections $\bL_e$, $u_e$ is one of the two branch jumps associated with $\bL_e$ (and the parity of $\deg u_e$ is independent of the choice),
	and $\chi(\bL_i), \chi(\bL_e)$ are the usual Euler characteristics.
\end{lem}

\begin{proof}
First, note that the Euler characteristic of the total cohomology equals that of the chain complex.  It suffices to compute the Euler characteristic of the Floer complex formed by the Morse generators of clean intersections of $\bL$. 

The Floer complex is generated by the following.  Codimension zero clean intersections are the connected components $\bL_i$ in the normalization of $\bL$, which contribute $\sum_i\chi(\bL_i)$ in $\chi^{\textrm{LF}}(\bL).$  Additionally, the clean intersections of positive codimension are smooth manifolds $\bL_e$. Each degree $k$ critical point of $\bL_e$ contributes two generators, which have degree $ u_e +k$ and $v_e+k$ respectively, in the Floer complex, where $u_e$ and $v_e$ are the two branch jumps associated with the transverse intersection point. Denote the number of degree $k$ critical points of $\bL_e$ by $a_k$. The Euler character of $\bL$ is
\begin{align*}
	\chi^{\textrm{LF}}(\bL)=& \sum_i \chi(\bL_i)+ \sum_e \sum_k (-1)^k a_k\left((-1)^{\deg u_e}+(-1)^{\deg v_e}\right) \\
	= & \sum_i \chi(\bL_i)+ \sum_e \left(\sum_k (-1)^k a_k (-1)^{\deg u_e} \left(1+(-1)^{\deg v_e + \deg u_e}\right)\right).
\end{align*}
Notice that $\deg v_e +\deg u_e  =\textrm{dim} \bL - \textrm{dim} \bL_e$. Hence, \begin{align*}
	\chi^{\textrm{LF}}(\bL)=& \sum_i \chi(\bL_i)+2 \sum_{e:\, \textrm{codim}(\bL_e) \textrm{ is even}} \left(\sum_k (-1)^k a_k (-1)^{\deg u_e} \right) \\ 
	=& \sum_i \chi(\bL_i) + 2 \sum_{e:\, \textrm{codim}(\bL_e) \textrm{ is even}} (-1)^{\deg u_e}\chi(\bL_e).
\end{align*}
\end{proof}

Note that $\bL$ is non-displaceable if $\chi(\bL)\not=0$.

For a Lagrangian brane $(\bL,\cE)$ of rank $\vec{r}$, where $\vec{r}=(r_i)_{i \in I}$ records the rank of the bundle over each component, the Euler characteristic of its Floer cohomology equals
$$ \chi_\bL^{\textrm{LF}}(\vec{r}) = \sum_i \chi(\bL_i) \cdot r_i^2 + 2 \sum_{e:\, \textrm{codim}(\bL_e) \textrm{ is even}} (-1)^{\deg u_e}\chi(\bL_e) \cdot r_{h(e)}r_{t(e)}$$
where $h(e),t(e)$ label the two components involved in the intersection $\bL_e$.  $\chi_\bL^{\textrm{LF}}$ is a quadratic form on integer vectors $\vec{r}$.

\begin{defn}
	A Lagrangian immersion $\bL$ is said to be (semi-)positive definite if $\chi_\bL^{\textrm{LF}}$ is (semi-)positive definite.
\end{defn}

\begin{prop} \label{prop:definite}
	Suppose $\bL$ is of dimension two and every component of the normalization of $\bL$ is a sphere.  Then $\bL$ is positive definite if and only if its intersection graph is of ADE type.  $\bL$ is strictly semi-positive definite if and only if it is of affine ADE type.  Moreover, $\chi_\bL^{\textrm{LF}}$ is always even.
\end{prop}

\begin{proof}
	In this case,
	$$ \chi_\bL^{\textrm{LF}}(\vec{r}) = 2 \left(\sum_i r_i^2 - \sum_{e}  r_{h(e)}r_{t(e)} \right) $$
	and $\sum_i r_i^2 - \sum_{e}  r_{h(e)}r_{t(e)}$ is the Tits quadratic form on $\vec{r}$ associated to the intersection graph with vertices $i$ and edges $e$.  by the well-known theorems of Gabriel \cite{Gab72} and Nazarova \cite{Naz73}, the Tits quadratic form is positive definite (resp. semi-positive definite) if and only if the graph is of ADE type (resp. affine ADE type) .
\end{proof}

Now consider an affine ADE diagram. Let $\bL$ be the corresponding Lagrangian immersion.  Then the associated matrix $\textbf{C}$ of $\chi_\bL^{\textrm{LF}}(\vec{r})$ is a Cartan matrix of an affine type, which has an one-dimensional kernel. Let $\vec{\delta} \in \mathbb{Z}^I$ be a vector in the kernel of $\textbf{C}$ such that $\delta_i >0$ and $\min \, \delta_i=1$. This uniquely determines $\vec{\delta}.$ Choose a generic parameter $\zeta^0 \in \R^I$ such that $\zeta^0 \cdot \vec{\delta}=0$. As a consequence of Proposition \ref{prop:nc-framed}, the localized mirror of $(\bL,\cE)$ gives an asymptotically locally Euclidean (ALE) space found by \cite{Kro89}.

\begin{cor}\label{prop:moduli}
	Let $\bL$ be the Lagrangian immersion obtained by plumbing according to the affine ADE diagram. Then the moduli space of $\zeta^0$-stable Lagrangian branes $(\bL,\cE)$ of rank $\vec{\delta}$ over $\C$ admits an asymptotically locally Euclidean metric.
\end{cor}
\begin{proof}
    By Theorem \ref{thm:preproj}, a Maurer-Cartan element $b=\sum B_a X_a \in \CF^1((\bL,\cE),(\bL,\cE))$ is a rank $\vec{\delta}$ representation of $Q$ over $\C$ satisfying the complex moment map equation: $$\left(\sum_{t(a)=v} \epsilon(a) B_{\bar{a}} B_{a}\right)_{v}=\vec{0},$$ where $v \in I$ and $B_a$ is the linear map corresponding to the arrow $a$. 
    
    Let $X_{\zeta^0}:= \{(B) \in Rep(Q,\vec{\delta}) \mid (\sum_{t(a)=v} \epsilon(a) B_{\bar{a}} B_{a})_v=\vec{0}\}\sslash_{\zeta^0} (GL(\vec{\delta})/\C^*)$, where $‘{\sslash_{\zeta^0} }’$ means the GIT quotient with respect to the parameter $\zeta^0$. As we assume $\zeta^0$ is generic, $X_{\zeta^0}$ is the quotient of the set of $\zeta^0$-stable representations by $GL(\vec{\delta})/\C^*$, see \cite{Nak07}, which is 1-1 corresponding to the isomorphism class of $\zeta^0$-stable Lagrangian branes $(\bL,\cE)$. Hence, $X_{\zeta^0}$ is the moduli space of $\zeta^0$-stable Lagrangian branes $(\bL,\cE)$. By the work of Kronheimer \cite{Kro89}, $X_{\zeta^0}$ has a hyperK\"ahler ALE metric.
\end{proof}

 \subsection{Quiver relation implies commutativity of charts for $D_4$}\label{sec:D4-stable}


Before the general construction of framed torsion-free sheaves, let's give an example of how to find an affine chart of a stable family. Using quiver algebroid stack, one can obtain a local-to-global description of the noncommutative deformation spaces.

Let $\bL$ be the Lagrangian immersion produced by plumbing according to the affine Dynkin diagram $D_4$. By Theorem \ref{thm:preproj}, the quiver $Q$ associated to $\bL$ is the double quiver of affine $D_4$, see Figure \ref{fig:affD_4}. The nc deformation space is the path algebra $\cA_0$ introduced in Example \ref{eg: D4}.

Let $(\bL,\cE)$ be a family of the Lagrangian branes of rank $\vec{\delta}=(2,1,1,1,1)$ over the noncommutative parameter space $\cA$, see the end of Section \ref{sec:stable}. In the following, we will show there exists some stable charts and prove that these stable charts are indeed commutative.

\begin{prop} \label{Prop: D4}
	Let $\cA_0$ be the noncommutative deformation space of $(\bL,\bb)$. Take $v_1$ as the reference vertex. Then there exists some stable affine charts of the stable family $\cA_0$ over $\cA$ of rank $\vec{\delta}$, which can be glued into the minimal resolution of $D_4$ singularity. In particular, these stable affine charts are commutative.
\end{prop}

\begin{proof}
	By Theorem \ref{thm:preproj}, $\cA_0$ satisfies
	\begin{center}
		$\sum_{i=1}^{4} a_ib^i=0$ and $b^ia_i=0$ for $i=1, \cdots,4$.
	\end{center} 
	
	In the following, we will write down some affine charts of the stable $(\cA_{0}, \delta)$-family over $\cA$ explicitly.

	Since the family $\cA_0$ over $\cA$ of rank $\vec{\delta}$ is stable, there's no proper $\cA_0$-submodule $M$ of $\cA^{\oplus (\sum \delta_i)}$ such that $e_{1}M = \cA$. Hence, the representation of $a_1$ is nonzero.  For the same reason, one of the images of $a_2, a_3, a_4$ is linearly independent with $a_1$. Otherwise, there exists a proper $\cA_0$-submodule violating the definition of stable family. Without loss of generality, we can assume the image of $a_2$ is linearly independent with $a_1$. In other words, $\begin{psmallmatrix}
	G_{20}(a_1) & G_{20}(a_2)
	\end{psmallmatrix}$ is invertible. Hence, using the automorphism $GL_{\delta_2}$, one can set $\begin{psmallmatrix}
	G_{20}(a_1) & G_{20}(a_2)
\end{psmallmatrix}= \begin{psmallmatrix}
1 & 0 \\
0 & 1
\end{psmallmatrix}$. Besides, $b^2$ is also nonzero with $b^2a_2=0$. Thus, the image of $b^2$ can be normalized as $\begin{psmallmatrix}
		1 & 0
	\end{psmallmatrix}$. Similarly, the restriction on vertices $v_3$ and $v_4$ implies that $b^3$ and $b^4$ are nonzero. Assume $b^3a_1$ and $b^4a_1$ are not zero. Using the weakly unobstructed equation, the representation of $\cA_{0}$ over $\cA$ is determined by three parameters $X$, $Y$ and $Z$. Let $\cA_2$ be the $\C$-algebra generated by $X$, $Y$ and $Z$. Notice that $\cA_2$ is a quiver algebra, whose quiver has a single vertex, since $\cA$ only has a single vertex. One obtain the representation $G_{20}$ exactly as shown in Example \ref{eg: D4}.  Conversely, the inverse representation $G_{02}$ can be found by localizing a matrix of paths. Keeping track of the automorphisms at each vertex gives rise to the inverse representation $G_{02}$ up to gerbe terms. 
	
	Observe that even though $\cA_2$ is not assumed to be commutative, the existence of such representations $G_{02}$ and $G_{20}$ forces $\cA_2$ to be commutative. This is because $G_{02}(XY)=G_{02}(YX)$ as shown in Example \ref{eg: D4}. Applying $G_{20}$, one has $XY=YX$. Similarly, one can show $[Y,Z]=[X,Z]=0$. Thus, $\cA_{2}$ has to be a $\C$-subalgebra of $\C[X,Y,Z]/(XY-(X+1)Z).$ In particular, one can take $\cA_2$ to be $\C[X,Y,Z]/(XY-(X+1)Z) \cong \C[X,T],$ where $T:=Y-Z$. This provides a commutative affine chart of the stable $(\cA_{0},\delta)$-family over $\cA$, which is the affine space $\A_{\C}^2$. In the following, we will denote the generators of $\cA_2$ by $X_2, Y_2, Z_2$, indicating $\cA_2$ is constructed by localizing $\cA_0$ at the matrix $\begin{psmallmatrix}
		a_1 & a_2
	\end{psmallmatrix}$.  
	
	In general, different choices of localization correspond to different local affine charts. If instead of assuming $b^3a_1$ is nonzero, we assume $b^3a_2b^2a_1$ is invertible. Denote the corresponding localization of $\cA_0$ by $\cA_{0,loc}$. We have 
	$$G_{02'}: \cA_{2'}:= \C[X_2',Y_2',Z_2']/(X_2'Y_2'-(X_2'Y_2'^2-1)Z_2') \to \cA_{0,loc}, \quad  G_{02'}:=
	\begin{cases}
		X_2' \mapsto (\alpha_1a_3)(b^3a_2b^2a_1) \\
		Y_2' \mapsto (b^3a_2 b^2a_1)^{-1}(b^3a_1)   \\
		Z_2' \mapsto (b^4a_1)^{-1}b^4a_2b^2a_1
	\end{cases}$$
	$$G_{2'0}: \cA_{0,loc} \to Mat(\cA_{2'}), \quad G_{2'0}:=
	\begin{cases}
		a_1 \mapsto \begin{psmallmatrix}
			1 \\
			0
		\end{psmallmatrix} \\
		a_2 \mapsto \begin{psmallmatrix}
			0 \\
			1
		\end{psmallmatrix} \\
		a_3 \mapsto \begin{psmallmatrix}
			X_2' \\
			-Y_2'X_2'
		\end{psmallmatrix} \\
		a_4 \mapsto \begin{psmallmatrix}
			-Z_2'(Y_2'X_2'Y_2'-1) \\
			Y_2'X_2'Y_2'-1
		\end{psmallmatrix} \\
		b^1 \mapsto \begin{psmallmatrix}
			0 & Z_2'X_2'Y_2'-X_2'
		\end{psmallmatrix} \\
		b^2 \mapsto \begin{psmallmatrix}
			1 & 0
		\end{psmallmatrix} \\
		b^3 \mapsto \begin{psmallmatrix}
			Y_2' & 1
		\end{psmallmatrix} \\
		b^4 \mapsto \begin{psmallmatrix}
			1 & Z_2'
		\end{psmallmatrix} \\
	\end{cases}.$$
	$G_{02'}$ and $G_{2'0}$ are inverse to each others up to gerbe terms, which can be found in a similar way. Composing $G_{20}$ and $G_{02'}$, one has the transition map $G_{22'}:= G_{20} \circ G_{02'}: \cA_{2'} \to \cA_{2}:$
	$$G_{22'}:=
	\begin{cases}
		X_2' \mapsto (\alpha_1a_3)(b^3a_2b^2a_1) \mapsto -X_2Y_2^2\\
		Y_2' \mapsto (b^3a_2 b^2a_1)^{-1}(b^3a_1) \mapsto Y_2^{-1} \\
		Z_2' \mapsto (b^4a_1)^{-1}b^4a_2b^2a_1 \mapsto Z_2
	\end{cases}.$$
	This provides the local gluing of one $(-2)$-curve and the representation of the arrow $b^3$ serves as the homogeneous coordinate of this exceptional $\mathbb{P}^1$.
	
	By making alternative choices for localization, one can construct other $(-2)$-curves. If the image of $a_3$ is linearly independent with $a_1$, we define  $\begin{pmatrix}
		\alpha^1 \\
		\alpha^3
	\end{pmatrix}$ to be the inverse of $\begin{pmatrix}
		a_1 &
		a_3
	\end{pmatrix}$. Besides, localizing at $\mathrm{diag}(b^2a_1, b^3 a_1, b^4 a_1)$, we obtain another commutative chart $\cA_3:=\C[X_3,Y_3,Z_3]/(X_3Z_3-(X_3+1)Y_3).$ The representations can be found similarly:
	
	$$G_{03}: \cA_{3}:= \C[X_3,Y_3,Z_3]/(X_3Y_3-(X_3+1)Z_3) \to \cA_{0,loc}, \quad  G_{03}:=
	\begin{cases}
		X_3 \mapsto (b^3a_1)^{-1}(\alpha_3a_4)(b^4a_1) \\
		Y_3 \mapsto (b^2 a_1)^{-1}(b^2a_3)(b^3a_1) \\
		Z_3 \mapsto (b^4a_1)^{-1}b^4a_3b^3a_1
	\end{cases}$$
	$$G_{30}: \cA_{0,loc} \to Mat(\cA_{3}), \quad G_{30}:=
	\begin{cases}
		a_1 \mapsto \begin{psmallmatrix}
			1 \\
			0
		\end{psmallmatrix} \\
		a_2 \mapsto \begin{psmallmatrix}
			X_3Z_3 \\
			-X_3-1
		\end{psmallmatrix} \\
		a_3 \mapsto \begin{psmallmatrix}
			0 \\
			1
		\end{psmallmatrix} \\
		a_4 \mapsto \begin{psmallmatrix}
			-Z_3X_3 \\
			X_3
		\end{psmallmatrix} \\
		b^1 \mapsto \begin{psmallmatrix}
			0 & X_3Y_3^2-Z_3X_3Y_3
		\end{psmallmatrix} \\
		b^2 \mapsto \begin{psmallmatrix}
			1 & Y_3
		\end{psmallmatrix} \\
		b^3 \mapsto \begin{psmallmatrix}
			1 & 0
		\end{psmallmatrix} \\
		b^4 \mapsto \begin{psmallmatrix}
			1 & Z_3
		\end{psmallmatrix} \\
	\end{cases}.$$
	If we localize at $b^2a_1, b^3a_1, b^4a_3b^3a_1$ instead, we obtain the commutative chart $\cA_{3'}:=\C[X_3',Y_3',Z_3']/(X_3'Z_3'-(X_3'(Z_3')^2-1)Y_3')$. The transition map between $\cA_3$ and $\cA_{3'}$ can be found similarly:
	$$G_{33'}:=
	\begin{cases}
		X_3' \mapsto (\alpha_1a_4)(b^4a_3b^3a_1) \mapsto -X_3Z_3^2 \\
		Y_3' \mapsto (b^2a_1)^{-1}b^2a_3b^3a_1 \mapsto Y_3 \\
		Z_3' \mapsto (b^4a_3b^3a_1)^{-1}b^4a_1 \mapsto Z_3^{-1}
	\end{cases}$$
	which provides the second $(-2)$-curve and the representation of the arrow $b^4$ serves as the homogeneous coordinate.
	
	In addition, we can also obtain the transition maps between $\cA_4:= \C[X_4,Y_4,Z_4]/(X_4Y_4-(X_4+1)Z_4)$ and $\cA_4':=\C[X_4',Y_4',Z_4']/(X_4'Y_4'-(X_4'(Y_4')^2-1)Z_4'):$ $$G_{44'}:=
	\begin{cases}
		X_4' \mapsto (\alpha_1a_2)(b^2a_4b^4a_1) \mapsto -X_4(Y_4)^2 \\
		Y_4' \mapsto (b^2a_4b^4a_1)^{-1}b^2a_1 \mapsto Y_4^{-1}\\
		Z_4' \mapsto (b^3a_1)^{-1}b^3a_4b^4a_1 \mapsto Z_4
	\end{cases}$$
	where $\cA_4$ (resp. $\cA_4'$) comes from the normalization of $\cA_{0}$ after localizing at $\begin{psmallmatrix}
		a_1 & a_4
	\end{psmallmatrix}$ and $\mathrm{diag}\begin{psmallmatrix}
		b^2a_1& b^3a_1& b^4a_1
	\end{psmallmatrix}$ (resp. $\begin{psmallmatrix}
		a_1 & a_4
	\end{psmallmatrix}$ and $\mathrm{diag}\begin{psmallmatrix}
		b^2a_4b^4a_1& b^3a_1& b^4a_1
	\end{psmallmatrix}$). This gives the third $(-2)-$curve. 
	
	The fourth exceptional curve, which intersects the previous three, comes from the transition maps $G_{32}: \cA_2 \to \cA_3$:
	$$G_{32}:=
	\begin{cases}
		X_2 \mapsto (b^2a_1)^{-1}(\alpha_2 a_3)(b^3a_1) \mapsto -(X_3+1)^{-1}\\
		Y_2 \mapsto (b^3a_1)^{-1}(b^3a_2)(b^2a_1) \mapsto -Y_3(X_3+1) \\
		Z_2 \mapsto (b^4a_1)^{-1}(b^4a_2)(b^2a_1) \mapsto Z_3
	\end{cases}.$$ Notice that gluing between $\cA_2$ and $\cA_3$ happens at the open subset $X_3 \neq -1.$ Thus, the closure of $Y_3=0$ indeed provides the fourth $(-2)$-curve. The dual graph of these four $(-2)$-curves form the $D_4$ Dynkin  diagram. The stable family $\cA_0$ over $\cA$ of rank $\vec{\delta}$ provides a quiver algebroid stack corresponding to the minimal resolution of $D_4$-singularity.	
\end{proof}

\begin{rem}
	As in Example \ref{eg: D4}, $\cA_k \cong \C[X,T]$ for $k=2,3,4$. Besides, $\cA_{k'} \cong \C[X',Y']_{(X'Y'^2-1)}$, since $X_k'Y_k'^2-1=0$ would imply $X_k'Y_k'=0$ in $\cA_{k'}$, which is a contradiction. All these charts are open subsets of $\A_\C^2.$
\end{rem}

\subsection{Mirror functor and monadic complexes} \label{sec: MF}

In this section, we compute the images of framed Lagrangian branes under the mirror functor and show that they coincide with the monadic complexes algebraically constructed by Nakajima \cite{Nak07}.  Restricting to the affine ADE cases, we show that they produce the framed torsion-free sheaves over the mirror asymptotically locally Euclidean (ALE) spaces.

Let $\bL$ (resp. $\bL^{\fr}$) be a (framed) Lagrangian immersion obtained by applying 2-sphere plumbing construction as in Section \ref{sec:preproj}. In order to obtain the monad of framed torsion-free sheaves, we consider the family of framed Lagrangian branes $(\bL^{\textrm{fr}},\cE)$ of rank $\vec{r}$.  By replacing the formal universal deformations $(\bL^{\textrm{fr}},\bb^{\fr})$ by the family of branes $(\bL^{\textrm{fr}},\cE)$ in Proposition \ref{prop:nc-framed} and using the same method, we obtain the Maurer-Cartan space of $(\bL^\fr,\cE)$ as below.

\begin{lem} \label{lem:wk 2}
	Let $b=\sum_a B_a X_a + \sum_v i_v I_v + \sum_v j_v J_v$ where $a$ denotes an arrow and $v$ denotes a vertex; $B_a, i_v, j_v$ are matrices of corresponding ranks.  The unobstructed equation for the family $(\bL^{\textrm{fr}},\cE)$ is 
	$$\sum_{t(a)=v} \epsilon(a) B_{\bar{a}} B_{a} + i_vj_v =0$$
	for all $v \in I$. In particular, in the subcategory generated of $\bL^\fr$, the Maurer-Cartan space of the family  $(\bL^{\textrm{fr}},\cE)$ over $\Lambda_0$ is the stack $[M(Q^\fr,\vec{r})/GL(\vec{r})]$, where $M(Q^\fr,\vec{r}) \subset Rep(Q^\fr,\vec{r})$ is the subvariety of $\A$-representations.
\end{lem}

\begin{rem}
	The above unobstructed equation exactly coincides with the complex moment map equation.  Thus, below we consider the Nakajima quiver variety in zero complex moment map level.  The real parameters of Nakajima quiver variety correspond to choices of stability conditions.
\end{rem}


Next, we restrict the functor $\cF^{(\bL,\bb)}$ to the branes $(\bL^\fr,\cE)$ and produce the monadic complex of the corresponding framed double quiver. We write $\bb=\sum_a x_a X_a$ for $(\bL,\bb)$ where $a$ are arrows of $Q$.  The deformation parameter for $(\bL^\fr,\cE)$ is denoted as $b$ given in Lemma \ref{lem:wk 2}.


\begin{thm} \label{thm:univb}
	$\cF^{(\bL,\bb)}$ transforms the framed Lagrangian branes $(\bL^{\textrm{fr}},\cE)$ into the following complexes:
	\begin{equation} \label{eq:univb}
		(\cF^{(\bL,\bb)}(\bL^{\textrm{fr}},\cE),m_1^{\bb}):  L^0((\bL,\bb),(\bL,\cE)) \to E((\bL,\bb),(\bL,\cE)) \oplus L^1((\bL,\bb), (F,\cE)) \to L^2((\bL,\bb),(\bL,\cE)),
	\end{equation}
which coincides with the monad constructed algebraically by Nakajima \cite{Nak07}. 
The complex is spanned by Floer generators $M_v \in L^0((\bL,\bb),(\bL,\cE)), P_v \in L^2((\bL,\bb),(\bL,\cE))$, and the immersed sectors $J_v \in L^1((\bL,\bb), (F,\cE))$, $X_a \in E((\bL,\bb),(\bL,\cE))$ for all vertices $v$ and arrows $a$ of $Q$. Here, $M_v$, $P_v$ are the maximal points and minimum points of the Morse function $f_v$ on $S^2_v$ respectively, while $X_a \in \CF^1(S^2_{t(a)},S^2_{h(a)})$ and $J_v \in \CF^1(\bL_v, F_v)$ are degree $1$ immersed sectors. 
Furthermore, the first differential takes the form $$m_{1}^{\bb}(\eta M_v)= \sum_{t(\bar{a})=v} (B_{\bar{a}} \eta) X_{\bar{a}} + \sum_{h(a)=v} ( \eta x_a )X_a + j_v\eta J_v,$$ and the second one is $$m_{1}^{\bb}(\eta' X_a)=  \eta' x_{\bar{a}} P_{h(a)}+ B_{\bar{a}} \eta' P_{t(a)}; \,\,  m_{1}^{\bb}(\eta'' J_v) = i_v \eta'' P_v.$$
\end{thm}

\begin{proof}
	
	
	
	Recall that $\bL= \bigcup_{v \in I} S^2_v$ and $\bL^{\textrm{fr}}= \bigcup_{v \in I} (S^2_v \cup T^*_{p_v}S^2_v)$. By the construction of the localized mirror functor, $$\cF^{(\bL,\bb)}(\bL^{\textrm{fr}},\cE)= \A \otimes_{\Lambda^{\oplus}_0} \CF^{\bullet}(\bL, (\bL,\cE)) \oplus \A \otimes_{\Lambda^{\oplus}_0} \CF^{\bullet}(\bL, (\bF,\cE)),$$ where $\CF^{\bullet}(\bL, (\bF,\cE))=\bigoplus_{v\in I} \Lambda_0 \otimes \cE|_{J_v} \langle J_v \rangle$ and $\CF^{\bullet}(\bL, (\bL,\cE))= \bigoplus_{a \in \mathfrak{A}} \Lambda_0 \otimes \cE|_{X_a}\langle X_a \rangle \bigoplus_{v \in I}\oplus_{Y \in Crit(f_v)} \Lambda_0 \otimes \cE|_Y \langle Y \rangle$ is generated by $X_a$, $M_v$ and $P_v$ using the Morse model. Hence, $\cF^{(\bL,\bb)}(\bL^{\textrm{fr}},\cE)$  has the form as the complex \ref{eq:univb}. More conventions about the notations can be found in Section \ref{sec:stable}.
	
	
	It remains to compute the differential. For simplicity, we will omit the sign, which depends on the choices of orientations and spin structure.  
	
	Observe that $M_v$ is the fundamental class of the sphere $S^2_v$. We have $m_2(\sum_a X_a, M_v)=\sum_{ h(a)=v} m_2(X_a,M_v)=\sum_{h(a)=v} X_a$, $m_2( M_v,\sum_a X_a)=m_2(M_v,\sum_{t(a)=v} X_a)=\sum_{t(a)=v}X_a$ and $m_2( M_v, \sum_k J_k)=m_2(M_v,J_v)=J_v$. Besides, $m_k(\cdots, M, \cdots)$ vanishes for $k\geq 3.$ Hence, for the first differential, we have $$m_{1}^{\bb}(\eta M_v)= m_2(\sum_{h(a)=v} x_a X_a ,\eta M_v) + m_2(\eta M_v, \sum_{t(a)=v} B_aX_a)+ m_2(\eta M_v, j_vJ_v)= \sum_{h(\bar{a})=v} (B_{\bar{a}} \eta) X_{\bar{a}} + \sum_{h(a)=v} ( \eta x_a )X_a + j_v\eta J_v, $$which counts the bigons that pre- and post-compose formal boundary deformations of compact components. 
	
	For the second differential, we should use the fact that $m_1^{\bb} (X)=\sum m_k(\bb, \bb, \cdots, X, \bb,\cdots,\bb)$ is the same as $\partial_x m_0^{\bb}$. 
	Using the coordinate change in Theorem \ref{thm:preproj}, the second differential is the same as differential of the weakly unobstructed relations. In other words, we have $$m_{1}^{\bb}(\eta' X_a)=   m_2(\eta' X_a, B_{\bar{a}}X_{\bar{a}}) + m_2(x_{\bar{a}} X_{\bar{a}},\eta' X_a)= \eta' x_{\bar{a}} P_{h(a)}+ B_{\bar{a}} \eta' P_{t(a)}$$
	$$m_{1}^{\bb}(\eta'' J_v)= m_2(\eta'' J_v , i_vI_v)=i_v \eta'' P_v.$$
	Due to the unobstructed equations, $(m_1^{\bb})^2$ is zero, yielding the Floer complex \ref{eq:univb}. Notice that the projective $\A$-module $\A e_v$ in the coefficient corresponds to the tautological bundle geometrically, see Section 4.4 of \cite{AH99} for such correspondence. Therefore, $\cF^{(\bL,\bb)}(\bL^{\textrm{fr}},\cE)$ is the monad demonstrated in \cite{Nak07}. 
\end{proof}
Now, let's restrict to the affine ADE case so that $\chi_\bL^{\textrm{LF}}(\vec{r})$ is strictly semi-positive definite.
To construct framed torsion-free sheaves, let's first recall the chamber structure of the stability parameters.  Readers are refered to \cite{Nak94} for more details. Let $D=(I,E)$ be an affine ADE Dynkin diagram. Let $\textbf{A}$ be the adjacency matrix of the graph. Then $\textbf{C=2I-A}$ is a Cartan matrix of an affine type (which is also the matrix associated to $\chi_\bL^{\textrm{LF}}(\vec{r})$). Fix a dimension vector $\textbf{v}$. Let 
\begin{align*}
	R_+ & :=\{\theta=(\theta_i) \in \Z_{\geq 0}^I \mid  \theta^t\, \textbf{C} \,\theta \leq 2\} \setminus\{0\},\\
	R_+(\textbf{v}) & :=\{\theta=(\theta_i) \in R_+ \mid \theta_i \leq \dim_\C V_i \text{ for all }i\},\\
	D_\theta & :=\{x=(x_i)\in R^I \mid x\cdot \theta=0\}\text{ for  }\theta\in R_+.
\end{align*} For the graph of affine type, $R_+$ is the set of positive roots of the corresponding Dynkin diagram, and $D_\theta$ is the wall defined by the root $\theta.$ 

As in Corollary \ref{prop:moduli}, we pick the primitive vector $\vec{\delta}$ in the kernel of $\mathbf{C}$ and a generic parameter $\zeta^0 \in \R^I$. Namely, $\zeta^0$ is not contained in any $D_\theta$, where $D_\theta$ is the hyperplane defined by a real root. We take a parameter $\zeta_\R$ from the chamber containing $-\zeta^0$ in its closure with $\zeta_\R \cdot \vec{\delta} < 0$.


\begin{cor} \label{cor:ALE}
	Let $D$ be the affine ADE Dynkin diagram.  We take the stable family of Lagrangian branes $(\bL,\cE')$ of rank $\vec{\delta}$ over the coordinate ring $R$ of the subvariety $M(Q,\vec{\delta})\subset \mathrm{Rep}(Q,\vec{\delta})$ of $\A$-representations.  Then $\cF^{(\bL,\cE')}$ transforms the $\zeta_\R$-stable framed Lagrangian branes $(\bL^\fr,\cE)$ into the monad of the torsion-free sheaves over the mirror.
\end{cor}

\begin{proof}
	Since $m_k$-operators are $GL(\vec{\delta})$-equivariant, the Floer complex $\cF^{(\bL,\cE')}((\bL^\fr,\cE))$ descends to the ALE space if $(\bL,\cE')$ is a stable family of rank $\vec{\delta}$ over $R$. The coefficients in the resulted Floer complex are the tensor product of the tautological bundles and the trivial representation bundles. Hence, this Floer complex coincides with the monad of torsion-free sheaves constructed algebraically by Nakajima \cite{Nak07}. Furthermore, one can homogenizing the monad and get the framed torsion-free sheaves by trivial extension.
\end{proof}

\begin{thm} \label{thm:frmonad}
	Let $(\bL^\fr,\cE_{1})$ and $(\bL^\fr,\cE_{2})$ be $\zeta_\R$-stable framed Lagrangian branes in rank $(\vec{v}_1,\vec{w})$ and $(\vec{v}_2,\vec{w})$ respectively. Denote the deformation parameters by $b_k= \sum_a B^k_a X_a + \sum_v i^k_v I_v + \sum_v j^k_v J_v$ for $k=1,2$. Then using the notations as in Theorem \ref{thm:univb}, $\cF^{(\bL^\fr,\cE_{1})}(\bL^\fr,\cE_{2})$ is a monadic complex over the Nakajima quiver variety $\bM_{\zeta_\R}(\vec{v}_1,\vec{w})$, which takes the following form:
	\begin{equation} \label{eq:frmonad}
		L^0((\bL,\cE_1),(\bL,\cE_2)) \to E((\bL,\cE_1),(\bL,\cE_2)) \oplus L^1((\bL,\cE_1), (F,\cE_2)) \oplus L^1((F,\cE_1),(\bL,\cE_2)) \to L^2((\bL^\fr,\cE_1),(\bL^\fr,\cE_2)).
	\end{equation}
    This complex has more generators in degrees 1 and 2, which are the immersed sectors $I_v \in L^1((F,\cE_1),(\bL,\cE_2))$ and the minimum points of the framing Lagrangians $P_{F_v}\in L^2((\bL^\fr,\cE_1),(\bL^\fr,\cE_2))$.Furthermore, the first differential takes the form $$m_{1}^{\bb}(\eta M_v)= \sum_{t(\bar{a})=v} (B^2_{\bar{a}} \eta) X_{\bar{a}} + \sum_{h(a)=v} ( \eta B^1_a )X_a + j^2_v\eta J_v+  \eta i^1_v I_v,$$ and the second one is $$m_{1}^{\bb}(\eta' X_a)=  \eta' B^1_{\bar{a}} P_{h(a)}+ B^2_{\bar{a}} \eta' P_{t(a)}; \,\,  m_{1}^{\bb}(\eta'' J_v) = \eta'' i^1_v P_{F_v}+ i^2_v \eta'' P_{v}  ; \,\, m_{1}^{\bb}(\eta''' I_v)= \eta''' j_v^1 P_v+ j_v^2 \eta''' P_{F_v}.$$
\end{thm}

	Using the anti-symmetric involution in Theorem \ref{thm:preproj}, the proof of this theorem is similar to Theorem \ref{thm:univb}.

In the affine $A_n$ case, we have the following nice property of localized mirror functor $\cF^{\bL^\fr}$.
\begin{thm}\label{thm:fr rep}
	Let $(\bL^{\textrm{fr}},\bb^\fr)$ be the framed Lagrangian immersion in affine $A_n$. Then the cohomology of $\cF^{(\bL^{\textrm{fr}},\bb^\fr)}(\bL^{\textrm{fr}},\cE)$ gives a framed quiver representation.
\end{thm}

\begin{proof}
	In this proof, the minimum points of the framing components will be denoted by $F_k$. We write the formal boundary deformation of $(\bL^{\textrm{fr}},\bb^\fr)$ by $\bb^\fr=\sum x_k X_k + y_k Y_k +i_k I_k +j_k J_k$, while the deformation of $(\bL^\fr,\cE)$ by $b^\fr=\sum x_k^c X_k + y_k^c Y_k+ i_k^c I_k+ j_k^c J_k$. Here $X_k \in \CF^1(\bS^2_{k}, \bS^2_{k+1})$, $Y_k \in \CF^1(\bS^2_{k+1}, \bS^2_{k})$, $I_k \in \CF^1(F_{k}, \bS^2_{k})$ and $J_k \in \CF^1(\bS^2_{k}, F_{k})$ are the degree 1 immersed sectors for all $k$ mod $n+1$. Besides, we will disregard the sign of the complex, which depends on the choice of orientation and spin structure. As before, $x_k, y_k, i_k, j_k$ are arrows in $Q^\fr$, while $x_k^c, y_k^c,i_k^c,j_k^c$ are matrices of corresponding ranks. The complex $(\cF^{(\bL^{\textrm{fr}},\bb^\fr)}(\bL^{\textrm{fr}},\cE), m_1^{\bb^\fr})$ is in the form of Complex \ref{eq:frmonad}, which can be computed as in Theorem \ref{thm:univb}.

    We will prove that the cohomology concentrates on the highest degree and it provides a framed quiver representation. The idea is to make use of the symmetry of the affine $A_n$ diagram and reduce to Theorem \ref{thm:frep}.  
    
    First, we observe that the first differential is injective. The output of $m_1^{\bb}(\sum \eta_k M_k)$ at $I_v$ is given by $m_2(i_v I_v, \sum \eta_k M_k)= \eta_v i_v I_v$ for all vertices $v$. Since $\eta_v \in \cE|_{M_v} \otimes \A e_v$ and $i_v$ have no left zero divisors in $\cE|_{M_v} \otimes \A e_v$,we have  $\eta_v i_v=0$ if and only if $\eta_v=0$ for all $v$.
    
    Let's show the complex is exact at degree $1$. Let $U:= \sum_{k\in I} a_k X_k + b_k Y_k+ c_k J_k + d_k I_k$ be an element in the kernel. The outputs at $P_k$ and $F_k$ are zero for $k=0, 1, \cdots n$. In other words, 
    \begin{numcases}{}
    	(y_k^c a_k + b_k x_k + a_{k-1}y_{k-1}+ x_{k-1}^c b_{k-1}+i_k^c c_k+ d_k j_k)P_k=0 &  \label{equ:part3}
    	\\
    	(c_k i_k + j_k^c d_k)F_k=0. & \label{equ:part4}
    \end{numcases}
    
    Notice that $i_k$ has no left zero divisors. Thus $c_k=j_k^c d_k'$, $d_k= v_k+ d_k' i_k$ for some $d_k' \in \cE|_{M_k} \otimes \A e_k$ and $v_k \in \cE|_{M_k} \otimes \A e_{f_k}$ such that $v_k$ doesn't contain a multiple of $i_k$ and $j_k^cv_k=0$, $k=0, \cdots, n$. Here $e_{f_k}$ is the idempotent at the $k$-th framing vertex. Hence the Equation \eqref{equ:part3} becomes 
    \begin{equation} \label{eq: ker 1}
    	\begin{aligned}
    		&y_k^c a_k + b_k x_k + a_{k-1}y_{k-1}+ x_{k-1}^c b_{k-1}+i_k^c c_k+ d_k j_k\\
    		=& y_k^c a_k + b_k x_k + a_{k-1}y_{k-1}+ x_{k-1}^c b_{k-1}+ (y_k^c x_k^c+ x^c_{k-1}y^c_{k-1})d_k'+ d_k'(y_kx_k+ x_{k-1}y_{k-1})+v_k j_k=0,
    	\end{aligned}
    \end{equation} where we used the unobstructed equations.

     In order to prove the exactness, it suffices to show  $a_{k}=x_k^c d_k'+ d_{k+1}' x_k$ and $b_k=d_k' y_k+ y_k^c d_{k+1}'$ for $k=1, \cdots, n$. Let $a_{k} =  a_{k}'+(x_k^c d_k'+ d_{k+1}' x_k)$ and $b_{k}=b_{k}'+(d_k' y_k+ y_k^c d_{k+1}')$. 
     Substituting $a_{k}$ and $b_{k}$ in Equation \eqref{equ:part3} by the above expressions, we have the following reduced equation:
     \begin{equation}\label{eq:part6}
     	y_k^c a_k'+ b_k'x_k+ a_{k-1}'y_{k-1} + x_{k-1}^c b_{k-1}'+v_kj_k=0
     \end{equation} for $k=0, 1, \cdots, n.$ 
 
     By the convention of coefficients, $h(a_k')=v_{k+1}$ (resp. $h(b_k')=v_k$), while $t(a_k')=v_k$ (resp. $t(b_k')=v_{k+1}$). Besides, since $v_k \in \cE|_{M_k} \otimes \A e_{f_k}$, $v_k j_m= 0$ for $k \neq m$. Therefore, the sum of these $n+1$ Equations (\ref{eq:part6}) is equal to 
     \begin{equation} \label{eq:part5}
     	\left(\sum_l y_l^c\right)\left(\sum_k a_k'\right)+ \left(\sum_l x_l^c\right)\left(\sum_k b_k'\right)= -\left(\sum_k b_k'\right)\left(\sum_l x_l\right) - \left(\sum_k a_k'\right)\left(\sum_l y_l\right) -\left(\sum_k v_k\right)\left(\sum_l j_l\right),
     \end{equation} since the product is zero if the tail and head vertices do not match.
 
     Notice that the relations of $\A^\fr$ preserve the length of the characters. Thus, the length of characters on the right is always one greater than the left. Hence, $\sum_k a_k'=\sum b_k'=0$, which further implies $\sum_k v_k=0$.
 
     Notice that if $\sum_k a_k'=0$, $a_k'$ is zero for all $k$, since $a_k'$ has tails at different vertex for each $k$. Similarly, $b_k'$ and $v_k$ also possess this property. Hence, $a_k'=b_k'=v_k=0$ for all $k$. The complex is exact in degree one. 

    The only nontrivial cohomology appears at the top degree, which gives the framed equiver representation. In particular, this implies Theorem \ref{thm:frep}, the self-plumbing case, when $k=0$. 
\end{proof}

\begin{cor}
	In the affine $A_n$ case, $\cF^{\U} \circ \cF^{(\bL^{\textrm{fr}},\bb)}(\bL^{\textrm{fr}},\cE)$ is quasi-isomorphic to $\cF^{(\bL,\bb)}((\bL^{\textrm{fr}},\cE)).$
\end{cor}

	Using the spectral sequence, the proof is similar to the Corollary \ref{cor:compare}.

	

\subsection{Stable family of affine $A_n$ Lagrangian immersions and special Lagrangians}
The affine $A_n$ surface is particularly well-studied.  This local model admits a special Lagrangian fibration constructed by $\bS^1$-symplectic reduction \cite{Gm01,Gol01}.  Its Fukaya category was known by \cite{Sei08} and stability conditions were studied by \cite{Thomas,TY02}.  SYZ and homological mirror symmetry in this case is well-known \cite{CLL12, LLW12, Cha13, AAK16, Kanazawa-Lau}.  An explicit expression of the non-Archimedean dual fibration was recently found by \cite{yuan}.

In this subsection, we revisit mirror symmetry and stability conditions for the affine $A_n$ surface.  The new discovery is the gluing relation for the affine $A_n$ Lagrangian skeleton (which is a union of vanishing spheres) with SYZ fibers via the language of quiver algebroid stacks, and the identification between its stable deformations and special Lagrangian torus objects.




Now we consider the surface 
$$X=\{(x,y,z)\in \mathbb{C}^3\mid xy = p(z) \text{ and }z\neq 0\}$$ where $p(z)=\prod_{k=1}^{n+1}(z-a_{k})$ is a monic polynomial and $-\infty < a_{n+1}<\cdots < a_{1}<0$ are distinct real numbers. This is a smoothing of the $A_n$ singularity. $X$ is equipped with the K\"ahler form inherited from the standard form on $\mathbb{C}^3.$ As is known, the projection map $\pi: (x,y,z)\mapsto z$ can be viewd as a Lefschetz fibration. One may regard Example \ref{exmp:conic} as its special case. It admits a natural Hamiltonian $\bS^1$-action given by \begin{align*}
	e^{it}\cdot(x,y,z) =  (e^{it} x, e^{-it} y, z)
\end{align*}
whose associated moment map is $\mu(x,y,z)=\frac{1}{2}(|x|^2-|y|^2)$. 

$X$ admits a special Lagrangian torus fibration 
$$F:X \to \R^2, \, F(x,y,z)=(\mu(x,y,z), \ln \, |z|^2).$$ 
We denote the singular fibers $F^{-1}(0, \ln |a_i|^2)$ by $\cS_i$, and the special Lagrangian tori $F^{-1}(0, \ln \, |r|^2)$ by $T_{r,0}$ for $r \neq a_i$ and $i=1,\cdots, n+1$.  In particular, we have the families of special Lagrangian tori $T_{r,0}$ for $|a_i|<r<|a_{i+1}|$. For simplicity, we denote these families by $T_i$.

Let $C_{q}=\left\{(x,y,z)\in X\mid \mu(x,y,z) = 0 \text{ and }z= q\right\}$.  Consider $\bL := \bigcup_{q\in c}C_{q}$, where $c$ is a loop passing through all the zeros of $p(z)$ as shown in Figure \ref{fig:An}. $\bL$ is a union of $n+1$ Lagrangian spheres $\bigcup_{i=1}^{n+1} \bS^2_i$.  $\bL$ corresponds to the affine $A_n$ quiver shown in Figure \ref{fig:quiver-An}.

\begin{figure}[hpt!]
\begin{tikzcd}
	&                                                      & {\bullet} \arrow[rrd, harpoon', shift left] \arrow[lld, harpoon', shift right=2] &                                                      &                                                                        \\
{\bullet} \arrow[r, harpoon] \arrow[r, harpoon] \arrow[rru, harpoon', shift left] & {\bullet} \arrow[r, harpoon] \arrow[l, harpoon, shift left] & {\cdots \cdots} \arrow[l, harpoon, shift left] \arrow[r, harpoon]                      & {\bullet} \arrow[r, harpoon] \arrow[l, harpoon, shift left] & {\bullet} \arrow[l, harpoon, shift left] \arrow[llu, harpoon', shift right=2]
\end{tikzcd}
	\caption{Quiver of affine $A_{n}$ type}
	\label{fig:quiver-An}
	\end{figure}
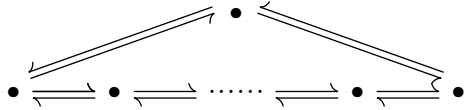

The generators of $\CF^1(\bL,\bL)$ are the immersed sectors $U_i^L \in \CF^1(\bS^2_i, \bS^2_{i+1}), V_i^L \in \CF^1(\bS^2_{i+1}, \bS^2_{i})$; the generators of $\CF^1(\cS_i,\cS_i)$ are the self-immersed sectors $U_i,V_i$ for $i=1,\cdots n+1$ as shown in Figure \ref{fig:An}. Besides, the holonomy variables of $T_i$ will be denoted by $x_i, y_i$. Let's denote the deformation parameter of $\bL$ (resp. $\cS_i$) by $\bb_0= \sum_i u_i^LU_i^L+ v_i^LV_i^L$ (resp. $\bb_i= u_iU_i + v_iV_i$) and the formal deformation space by $\A_0$ (resp. $\cA_i$).

We find that there exists a quiver algebroid stack $\hat{\cY}$, over which the Fukaya isomorphism equations between $(\bL,\bb_0)$ and immersed spheres $(\cS_i,\bb_i)$ as well as $(T_i,\nabla^{(x_i,y_i)})$ can be solved. This induces a map from the Maurer-Cartan space of $\bL$ to that of $\cS_i$ and $T_i.$


\begin{figure}[htb!]
	\includegraphics[scale=2]{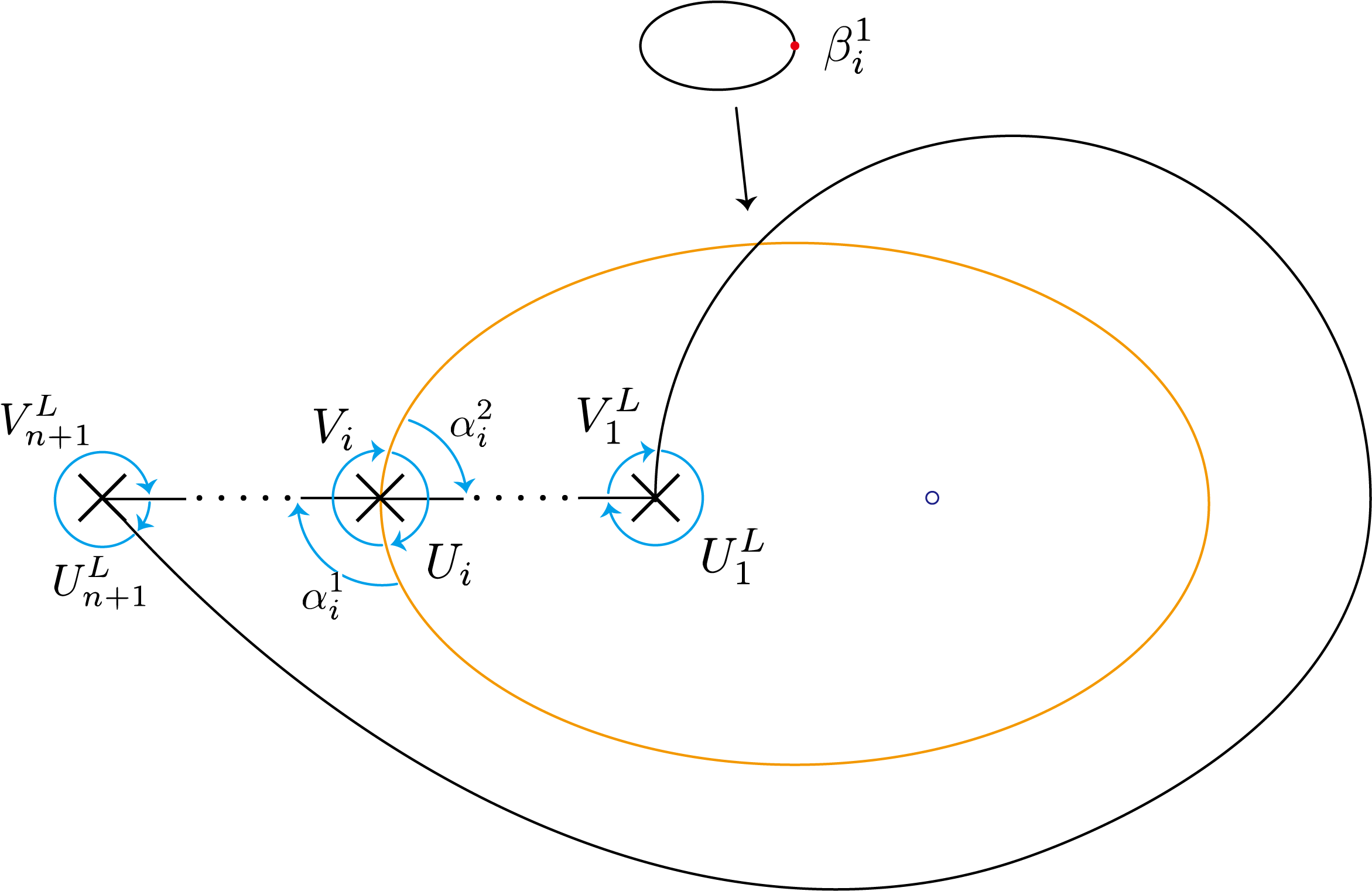}
	\caption{Base of the conic fibration in affine $A_{n}$ case}
	\label{fig:An}
  \end{figure}

\begin{thm}\label{thm:qstack}
	There exist preisomorphism pairs between $(\bL,\bb_0)$ and $(\cS_i,\bb_i),i=1,\cdots, n+1$:
	$$\alpha_i \in \CF_{\A_0(U_{0i})\otimes \cA_i}((\bL,\bb_0),(\cS_i,\bb_i)),\quad \beta_i \in \CF_{\cA_i \otimes \A_0(U_{0i})}((\cS_i,\bb_i),(\bL,\bb_0))$$
	and a quiver stack $\hat{\cY}$ corresponding to the minimal resolution of $A_n$-singularity over the Novikov field $\Lambda$, whose charts are $\A_0$ and $\cA_i,i=1,\cdots,n+1$, that solves the isomorphism equations for $(\alpha_i,\beta_i)$:
	\begin{align}
		\label{equ:stackrelation1}
		m_{1,\hat{\cY}}^{\bb_0,\bb_i}(\alpha_i) = 0,& \quad  m_{1,\hat{\cY}}^{\bb_i,\bb_0}(\beta_i) = 0;\\
		\label{equ:stackrelation2}
		m_{2,\hat{\cY}}^{\bb_0,\bb_i,\bb_0}(\alpha_i,\beta_i)  = \one_\bL,&\quad  m_{2,\hat{\cY}}^{\bb_i,\bb_0,\bb_i}(\beta_i,\alpha_i) = \one_{\cS_i}.
	\end{align}
	In above, $\A_0(U_{0i})$ is the localization of $\A_0$ at the set of arrows
	$\{v_1^{L},\cdots, v_{i-1}^L, u_{i+1}^{L}, \cdots, u_{n+1}^L\}$ for $i=1,\cdots,n+1$ respectively. In particular, the existence of the isomorphism pairs imposes commutativity on the local affine charts of $\A_0$.
\end{thm}

\begin{proof}
	Let $A_{i}$ and $A_{i}'$  denote areas of the polygons with vertices $\alpha_{i}^{1},u_{i+1}^{L},\cdots ,u_{n+1}^L,\beta_{i}^1$ and $\alpha_{i}^2,v_{i-1}^{L},\cdots, v_{1}^L,\beta_{i}^1$. The isomorphism pairs are defined using normalized immersed sectors $\alpha_{i}^1,\alpha_{i}^2$ and $\beta_{i}^1$(see Figure \ref{fig:An}), 
	$$
	\alpha_{i}= T^{-A_{i}}(u_{n+1}^L\cdots u_{i+1}^L)^{-1}\alpha_{i}^1+ T^{-A_{i}'} (v_{1}^L\cdots v_{i-1}^L)^{-1}\alpha_{i}^2, \quad \beta_i = \beta_{i}^1.$$ 
	It's easy to check they satisfy Equations \eqref{equ:stackrelation1} and \ref{equ:stackrelation2}.
	
	For simplicity, we assume that the area terms $A_i=A_i'=0$. The only difference this assumption makes is that the transition map will be defined up to a constant $T^{-A}$. However, this assumption suffices for the purpose of constructing the quiver algebroid stack. It is important to keep track of these area terms when considering the maps between Maurer-Cartan spaces.
	
	We now construct a quiver stack $\hat{\mathcal{Y}}$ as follows:
	\begin{enumerate}
		\item The underlying topological space of $\hat{\mathcal{Y}}$ is the dual polyhedral set  $B$ of the fan of $A_{n}$-resolution.
		\begin{figure}[htb!]
			\includegraphics[scale=0.7]{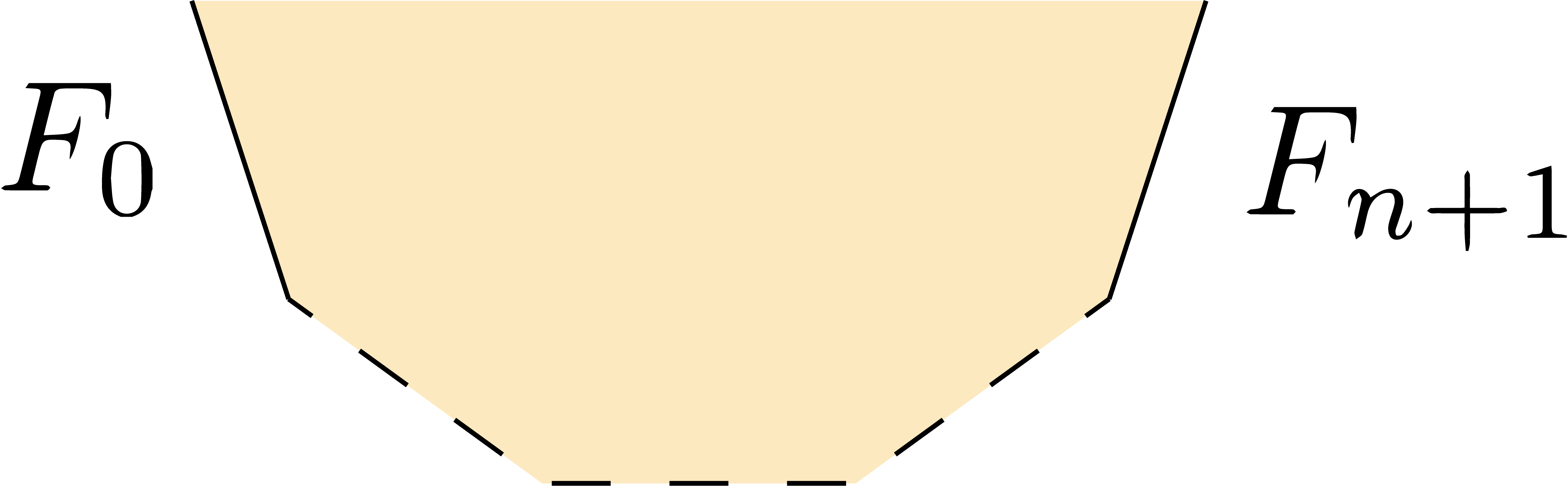}
			\caption[$A_{n}$ polyhedral set]{An illustration of the polyhedral set for $A_{n}$ resolution.}
			\label{fig:An_polyhedral}
		\end{figure}
		
		We endow $B$ with a topology $\left\{\emptyset,B, F_{1}^c,\cdots, F_{n+1}^c\right\}$  where $F_{i}^c$ is the complement of the $i$-th facet $F_{i}$. We take $\left\{U_{i}=F_0^c \cap \cdots \cap F_{i-2}^c\cap F_{i+1}^c\cap \cdots F_{n+1}^c \right\}_{1\leq i \leq n+1}$ together with $U_0:=B$,  which serve as the open cover of the quiver stack.

		\item We associate each open set $U_{i}$ a sheaf by localizing the quiver algebras $\mathbb{A}_{0}$ for $i=0$, and $\mathscr{A}_{i}$ for $1\leq i\leq n+1$. So for convenience, if there is no ambiguity we will use $\mathbb{A}_{0}$ or $\mathscr{A}_{i}$ to denote the sheaves. Then $\mathbb{A}_0(U_0\cap U_{i})$ is defined by localizing $\mathbb{A}_0$ at $\{v_{1}^L,\cdots, v_{i-1}^L,u_{i+1}^L,\cdots, u_{n+1}^L\}$; on the intersection it's the algebra localized at the union of corresponding elements.  
		
		\item The representation maps  $G_{i0}:\mathbb{A}_{0}|_{U_{0i}}\rightarrow\mathscr{A}_{i}|_{U_{0i}}$ for $1\leq i \leq n+1$ and $G_{0i}$ are defined as $$G_{i0}:= \begin{cases} 
			u_{j}^L\rightarrow  v_{i}u_{i}& \forall 1\leq j \leq i-1 \\
			u_{i}^L\rightarrow u_{i} &  \\ 
			u_{j}^L\rightarrow 1 & \forall i+1 \leq j \leq n+1  \\
			v_{j}^L\rightarrow 1 & \forall 1 \leq j \leq i-1  \\
			v_{i}^L\rightarrow v_{i} &  \\
			v_{j}^L \rightarrow  u_{i}v_{i}& \forall i+1 \leq j \leq n+1 
			\end{cases},\quad G_{0i}:=\begin{cases}
			u_{i}\rightarrow (u_{n+1}^L\cdots u_{i+1}^L)u_{i}^L(v_{1}^L\cdots v_{i-1}^L)^{-1} &  \\
			v_{i}\rightarrow (v_{1}^L\cdots v_{i-1}^L)v_{i}^L(u_{n+1}^L\cdots u_{i+1}^L)^{-1}
			\end{cases}.$$
		\item The gerbe terms of the form $c_{0i0}$  is defined by  $$c_{0i0}(h(u_{j}^L))= \begin{cases}
			v_{1}^L\cdots v_{j}^L & 1\leq j \leq i-1   \\
			u_{n+1}^L\cdots u_{i+1}^L &  j=i  \\
			u_{n+1}^L\cdots u_{j+1}^L & i+1\leq j\leq n+1 
			\end{cases}  \quad c_{0i0}(t(u_{j}^L))=\begin{cases}
			v_{1}^L\cdots v_{j-1}^L & 1\leq j \leq i-1  \\
			v_{1}^L\cdots v_{i-1}^L & j=i \\
			u_{n+1}^L\cdots u_{j}^L & i+1 \leq j\leq n+1 
			\end{cases}.$$Since $h(u_{j}^L) =t(v_{j}^L)$  and $t(u_{j}^L)=h(v_{j}^L)$,  we have the following identities $c_{0i0}(h(u_{j}^L)) =c_{0i0}(t(v_{j}^L))$  and $c_{0i0}(t(u_{j}^L))=c_{0i0}(h(v_{j}^L)).$ 
			
			\item We define the transition maps between $i$-th chart and $k$-th chart to be $G_{ik}(x):=G_{i0}\circ G_{0k}(x)$, so the corresponding gerbe terms are always trivial.
	\end{enumerate}
	In particular, the transition map between $i$-th chart and $(i+1)$-th chart is the following:
	\begin{align*}
		G_{i,i+1}=G_{i0}\circ G_{0,i+1}=\begin{cases}
			u_{i+1}\rightarrow v_{i}^{-1} &  \\
			v_{i+1}\rightarrow v_{i}u_{i}v_{i}=u_{i}v_{i}^2
		\end{cases},
	\end{align*}
	which are exactly the transition functions of the minimal resolution of $A_n$-singularity as a complex variety.

    An explicit computation shows the above data indeed defines a quiver algebroid stack over $P$. We take the composition $G_{0i}\circ G_{i0}$ as an example: 
	\begin{enumerate}
		\item $1\leq j \leq i-1:$  Using the substitution, $v_{j}^Lu_{j}^L = u_{j-1}^Lv_{j-1}^L$, we find \begin{align*}G_{0i}\circ G_{i0}(u_{j}^L)& =(v_{1}^L\cdots v_{i-1}^L)v_{i}^Lu_{i}^L(v_{1}^L\cdots v_{i-1}^L)^{-1}\\
			& = v_{1}^L\cdots v_{j}^Lu_{j}^L (v_{j-1}^L)^{-1}\cdots (v_{1}^L)^{-1}\\
			& = (v_{1}^L\cdots v_{j}^L)u_{j}^L(v_{1}^L\cdots v_{j-1}^L)^{-1}.\end{align*} 
		\item $j=i:$ \begin{align*}G_{0i}\circ G_{i0}(u_{i}^L)& =(u_{n+1}^L\cdots u_{i+1}^L)u_{i}^L(v_{1}^L\cdots v_{i-1}^L)^{-1}.\end{align*}
		\item $i+1\leq j \leq n+1:$ \begin{align*}G_{0i}\circ G_{i0}(u_{j}^L)& =1 \\
			& = (u_{n+1}^L\cdots u_{j+1}^L)u_{j}^L(u_{n+1}^L\cdots u_{j}^L)^{-1}.\end{align*}
	\end{enumerate}
	Therefore, they satisfy the equality $G_{0i}\circ G_{i0}(x) = c_{0i0}(h(x))\cdot G_{00}(x)\cdot c_{0i0}^{-1}(t(x))$. 
	
	In the above construction, we didn't assume $\cA_i$ is commutative. However, The existence of representations $G_{0i}$ imposes commutativity on $\cA_i$. This arises because $G_{0i}$ possesses a left inverse and they satisfy $G_{0i}(u_iv_i)=G_{0i}(v_iu_i)$ for all $i$.  
\end{proof}

Furthermore, there exists an isomorphism pair between the family of special Lagrangian tori $(T_i, \nabla^{(x_i,y_i)})$ and $(\bL,\bb_0)$ over the same quiver algebroid stack $\hat{\cY}$.

\begin{thm}\label{thm:qstack2}
	There exist preisomorphism pairs between the family of special Lagrangian tori $(T_i,\bb_i:=\nabla^{(x_i,y_i)}),i=1,\cdots, n$ and $(\bL,\bb_0)$:
	$$\alpha_i \in \CF_{\A_0(U_{0i}')\otimes \cA_i}((T_i,\bb_i),(\bL,\bb_0)),\beta_i \in \CF_{\cA_i \otimes \A_0(U_{0i}')}((\bL,\bb_0),(T_i,\bb_i ))$$
	that solves the Fukaya isomorphism equations  over the quiver algebroid stack $\hat{\cY}$ corresponding to the minimal resolution of $A_n$ singularity (see Theorem \ref{thm:qstack}):
	\begin{align}
		\label{equ:stare1}
		m_{1,\hat{\cY}}^{\bb_i,\bb_0}(\alpha_i) =& 0, m_{1,\hat{\cY}}^{\bb_0,\bb_i}(\beta_i) = 0;\\
		\label{equ:stare2}
		m_{2,\hat{\cY}}^{\bb_i,\bb_0,\bb_i}(\alpha_i,\beta_i) =& \one_{T_i}, m_{2,\hat{\cY}}^{\bb_0,\bb_i,\bb_0}(\beta_i,\alpha_i) = \one_{\bL},
	\end{align}
	where $U_{0i}':= U_{0i} \cap U_{0(i+1)}$, and $\A_0(U_{0i}')$ is the localization of $\A_0$ at the set of arrows
	$$\{v_1^{L},\cdots, v_{i}^L, u_{i+1}^{L}, u_{i+2}^{L}, \cdots, u_{n+1}^L\}$$ 
	for $i=1,\cdots,n$ respectively. 
\end{thm}

\begin{figure}[htb!]
	\includegraphics[scale=2]{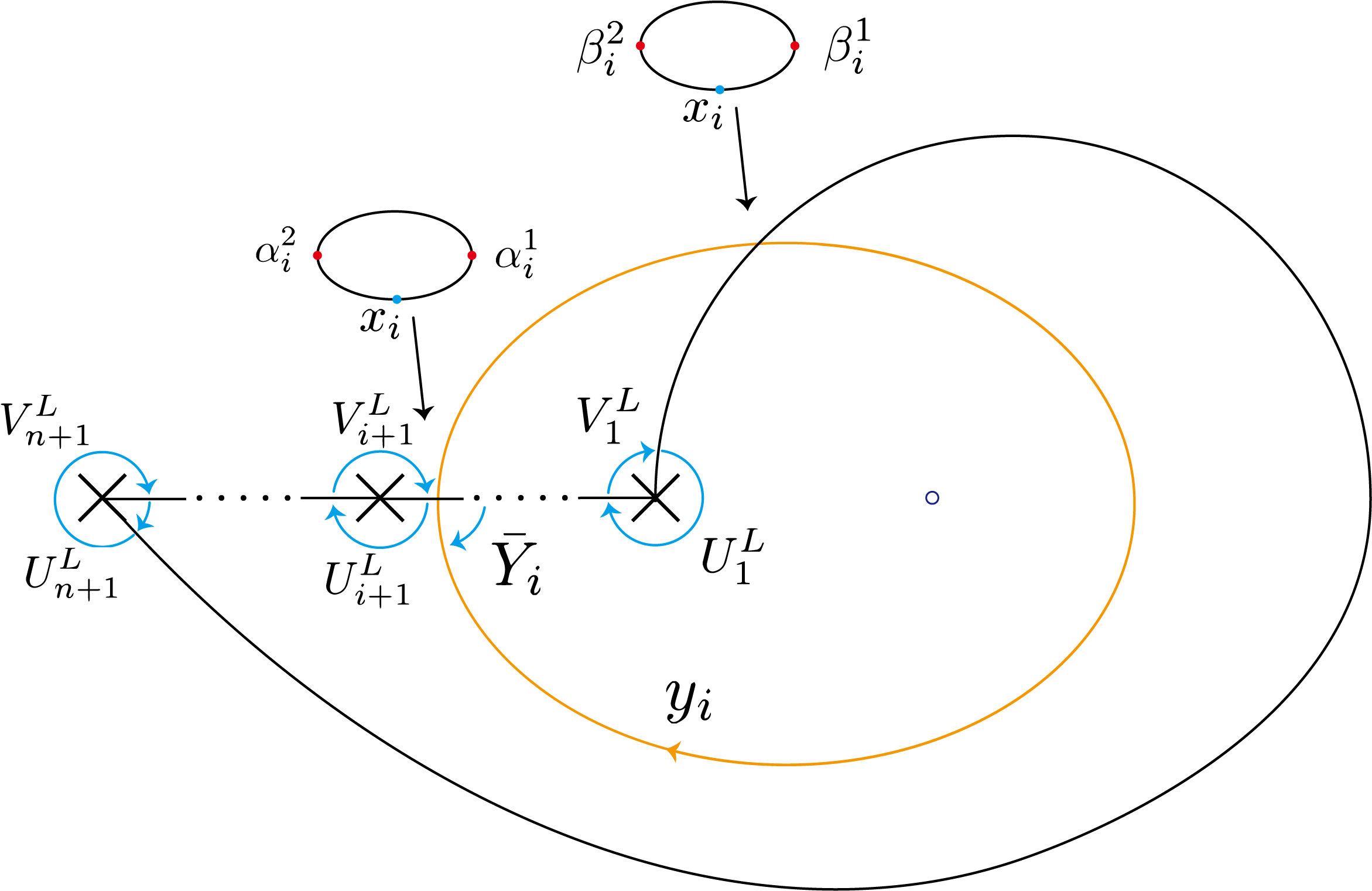}
	\caption{Base of the conic fibration in affine $A_{n}$ case. The holonomy variables of $T_i$ are denoted by $x_i,y_i$. $\bar{Y}_i \in \CF^1(\bS^2_i,(T_i,\bb_i)) \subset \CF^1((\bL,\bb_0),(T_i,\bb_i)).$}
	\label{fig:Anclean}
  \end{figure}
\begin{proof}
	In this proof, we will use the same notations for the quiver algebroid stack as in Theorem \ref{thm:qstack} and denote the deformation space of the tori $T_i$ by $\cA_i:=\Lambda_0[x_i,x_i^{-1},y_i,y_i^{-1}].$
	
	The torus $T_i$ cleanly intersects $\bL$ at two connected circles. Using Morse model, the critical points of the Morse functions generate the Lagrangian Floer complexes $\CF((T_i,\bb_i),(\bL,\bb_0))$ and $\CF((\bL,\bb_0),(T_i,\bb_i))$, among which the maximal points $\alpha_{i}^1 \in \CF^0((T_i,\bb_i),(\bL,\bb_0))$ and $\beta_{i}^1 \in \CF^0((\bL,\bb_0),(T_i,\bb_i))$ have degree zero. We observe that the normalized intersection points form a Fukaya isomorphism pair: $$\alpha_i:= T^{-A_i'} (v_1^L \cdots v_{i}^L)^{-1} \alpha_i^1, \quad \beta_i:= \beta_i^1.$$
	
	Similar as before, we set $A_i=A_i'=0$. The quiver algebroid stack $\hat{\cY}$ obtained in theorem \ref{thm:qstack} also provides a solution to the Fukaya isomorphism equations:
	\begin{enumerate}
		\item The underlying topological space of $\hat{\mathcal{Y}}$ is the polyhedral set $P$ of the minimal resolution of $A_{n}$ singularity, which is the same as in Theorem \ref{thm:qstack}. But in this proof, we will focus on $U_{0i}':= U_{0i} \cap U_{0(i+1)}$. 
		\item We associate each open set $U_i$ a sheaf of path algebras as above, for $i=0,1,\cdots, n.$ In addition, we also assign the sheaf $\tilde{\cA_i}$ associated to $\cA_i$ to $U_{0i}'$. 
		\item The transition maps  $G_{i0}:\mathbb{A}_{0}|_{U_{0i}'}\rightarrow\cA_{i}|_{U_{0i}'}$ and $G_{0i}: \cA_{i}|_{U_{0i}'} \to \mathbb{A}_{0}|_{U_{0i}'}$ for $1\leq i \leq n$  are defined by $$G_{i0}:= \begin{cases} 
			u_{j}^L\rightarrow  x_i-1 & \forall 1\leq j \leq i \\
			u_{i+1}^L\rightarrow y_i^{-1} &  \\ 
			u_{j}^L\rightarrow 1 & \forall i+2 \leq j \leq n+1  \\
			v_{j}^L\rightarrow 1 & \forall 1 \leq j \leq i  \\
			v_{i+1}^L\rightarrow (x_i-1)y_i &  \\
			v_{j}^L \rightarrow  x_i-1& \forall i+2 \leq j \leq n +1
		\end{cases},\quad G_{0i}:=\begin{cases}
			x_{i}\rightarrow (v_{1}^L\cdots v_{i+1}^L)u_{i+1}^L(v_{1}^L\cdots v_{i}^L)^{-1} + e_1 &  \\
			y_{i}\rightarrow (v_{1}^L\cdots v_{i}^L)(u_{n+1}^L\cdots u_{i+1}^L)^{-1}
		\end{cases}.$$ Recall that $(\bL,\bb_0)$ is weakly unobstructed. Namely, we have $v_j^Lu_j^L= u_{j+1}^L v_{j+1}^L \, \forall 1 \leq j \leq n$ and $u_{n+1}^L v_{n+1}^L=v_1^L u_1^L$.  Hence, $G_{0i}(x_i)=(v_{1}^L\cdots v_{i+1}^L)u_{i+1}^L(v_{1}^L\cdots v_{i}^L)^{-1} + e_1= v_1^L u_1^L+e_1$ by the weakly unobstructed equations.
		\item The gerbe terms $c_{0i0}$ at the vertices of $Q$  is defined by  $$c_{0i0}(a_j)= \begin{cases}
			e_1 & j=1\\
			v_{1}^L\cdots v_{j-1}^L & \forall 2\leq j \leq i+1   \\
			u_{n}^L\cdots u_{j}^L &  \forall i+2 \leq j \leq n+1  	 
		\end{cases} $$ where $a_j$ stands for the $j-$th vertex of $Q$. The other gerbe terms are trivial.
	\end{enumerate}
Notice that $G_{0i}$ is well-defined in the sense that $G_{0i}$ is intertwining, i.e. $G_{0i}(x_iy_i)=G_{0i}(y_ix_i).$ We have \begin{align*}
	G_{0i}(x_iy_i)=& (v_{1}^L\cdots v_{i+1}^L) (u_{n+1}^L \cdots u_{i+2}^L)^{-1} + (v_1^L \cdots v_{i}^{L})(u_{n+1}^L \cdots u_{i+1}^L)^{-1}.
\end{align*}
On the other hand, \begin{align*}
	G_{0i}(y_ix_i)=& (v_{1}^L\cdots v_{i}^L) (u_{n+1}^L \cdots u_{i+1}^L)^{-1}v_1^L u_1^L + (v_1^L \cdots v_{i}^{L})(u_{n+1}^L \cdots u_{i+1}^L)^{-1} \\
	=& (v_{1}^L\cdots v_{i}^L) (u_{n+1}^L \cdots u_{i+1}^L)^{-1} u_{n+1}^L v_{n+1}^L + (v_{1}^L\cdots v_{i}^L) (u_{n+1}^L \cdots u_{i+1}^L)^{-1} \\
	=& (v_{1}^L\cdots v_{i}^L) (u_{n-1}^L \cdots u_{i+1}^L)^{-1} (u_{n}^{L})^{-1} v_{n+1}^L + (v_{1}^L\cdots v_{i}^L) (u_{n+1}^L \cdots u_{i+1}^L)^{-1}\\
	=& (v_{1}^L\cdots v_{i}^L) (u_{n-1}^L \cdots u_{i+1}^L)^{-1} v_{n}^{L} (u_{n+1}^L)^{-1} + (v_{1}^L\cdots v_{i}^L) (u_{n+1}^L \cdots u_{i+1}^L)^{-1} \\
	=& (v_{1}^L\cdots v_{i+1}^L) (u_{n+1}^L \cdots u_{i+2}^L)^{-1} + (v_{1}^L\cdots v_{i}^L) (u_{n+1}^L \cdots u_{i+1}^L)^{-1}.
\end{align*} On the first line, we use another expression of $G_{0i}(x_i).$ On the fifth line, we apply the weakly unobstructed equations until it gets to the desired form. 

One can check that $G_{0i}\circ G_{i0}(x) = c_{0i0}(h(x))\cdot G_{00}(x)\cdot c_{0i0}^{-1}(t(x))$ for all $x$.  For instance, $$G_{0i}\circ G_{i0}(u_{i+1}^L)= (u_{n+1}^L \cdots u_{i+2}^L) u_{i+1}^L(v_1^L \cdots v_{i}^L)^{-1}= c_{0i0}(a_{i+2}) u_{i+1}^L c_{0i0}(a_{i+1}).$$

Furthermore, we can solve the isomorphism equations for $(\alpha_i,\beta_{i})$ over the quiver stack explicitly. More precisely, we get \begin{align*}
	m_{1,\hat{\cY}}^{\bb_0,\bb_i}(\beta_i) &= ((v_1^L \cdots v_{i}^L)+ y_i(u_{n+1}^L \cdots u_{i+1}^L))\bar{Y}_i^1 + ((v_1^Lu_1^L+1)+x_i) \beta_{i}^2
\end{align*}
\begin{align*}
	m_{2,\hat{\cY}}^{\bb_0,\bb_i,\bb_0}(\beta_i,\alpha_i)&=(v_1^L \cdots v_{i}^L (v_1^L \cdots v_{i}^L)^{-1})\one_{\bS^2_1}+ \cdots + (v_{i}^L(v_1^L \cdots v_{i}^L)^{-1}v_1^L \cdots v_{i-1}^L)\one_{\bS^2_{i}}       \\
	&((v_1^L \cdots v_{i}^L)^{-1}y_i(u_{n+1}^L \cdots u_{i+1}^L))\one_{\bS^2_{i+1}} + \cdots + ((u_{n}^L \cdots u_{i+1}^L)(v_1^L \cdots v_{i}^L)^{-1}y_iu_{n+1}^L)\one_{\bS^2_{n+1}}.
\end{align*}
Using the transition representations $G_{0i}: \cA_{i}|_{U_{0i}'} \to \mathbb{A}_{0}|_{U_{0i}'},$ 
\begin{align*}
	m_{1,\hat{\cY}}^{\bb_0,\bb_i}(\beta_i) &= ((v_1^L \cdots v_{i}^L)+ y_i(u_{n+1}^L \cdots u_{i+1}^L))\bar{Y}_i^1 + ((v_1^Lu_1^L+1)+x_i) \beta_{i}^2=0
\end{align*}
\begin{align*}
	m_{2,\hat{\cY}}^{\bb_0,\bb_i,\bb_0}(\beta_i,\alpha_i)&= \one_{\bS^2_1} + \cdots \one_{\bS^2_{i-1}}+ \one_{\bS^2_i} + \cdots \one_{\bS^2_{n+1}}= \one_{\bL}.
\end{align*}
The computations of the remaining isomorphism equations are similar. Hence, $(\alpha_i,\beta_i)$ forms a Fukaya isomorphism pair. 

\end{proof}



By the above theorem, the transition maps for the Fukaya isomorphism pair $(\alpha_i,\beta_i)$ give a map between (subsets of) the Maurer-Cartan spaces of the Lagrangian branes $\bL$ and special Lagrangian tori $T_i$.  Note that it is important to take localizations for the Maurer-Cartan space of $\bL$ for existence of such maps.  Indeed, these localizations come from a stability condition.  Below, we make this more explicit.

Implicitly the Lagrangian brane $(\bL,\cE)$ has rank $\vec{\delta}=(1,1,\cdots,1)$. Let $\zeta \in \R^{n+1}$ be a generic character such that $\zeta \cdot \vec{\delta}=0$ and only $\zeta_1$ is positive. In addition, recall that the subvariety of $\A_0$-representations in $\mathrm{Rep}(Q,\vec{\delta})$ is denoted by $M(Q,\vec{\delta})$ .

Let's denote the moduli space of $\zeta$-stable deformation of $(\bL,\cE)$ by $\textrm{Def}(\bL),$ which is a subset of $MC((\bL,\cE)).$ More precisely, $\textrm{Def}(\bL)$ is defined as a coset with $\Lambda_0$-values
\begin{center}
	$\textrm{Def}(\bL):= \{b \in M(Q,\vec{\delta}) |$ $b$ is $\zeta$-stable $\}/GL(\vec{\delta}).$ 
\end{center} 

Let's define $\Lambda_U:= \C^\times \oplus \Lambda_+.$ $\Lambda_U$ is a valuation zero subset of $\Lambda$ which forms a multiplicative group. Besides, we will consider the following subsets of the Maurer-Cartan spaces:
\begin{align*}
	\textrm{Def}(S_1):=& \{(u_1,v_1) \in (\Lambda_0)^2 \mid \, \val(u_1) \geq A_1 \} \\	\textrm{Def}(S_j):=& \{(u_j,v_j) \in (\Lambda_0)^2 \mid \, \val(u_j), \val(v_j) \geq \max\{|A_j-A_j'|,0\} \textrm{ and } \val(u_jv_j)>0\} \textrm{ for } j=2, \cdots, n\\
	\textrm{Def}(S_{n+1}):=& \{(u_{n+1},v_{n+1}) \in (\Lambda_0)^2 \mid \, \val(v_{n+1}) \geq A_{n+1}' \} \\
	\textrm{Def}(T_i):=& \{(x_i,y_i) \in (\Lambda_U)^2 \mid \, x_i \in -1+ \Lambda_{\geq |A_i-A_i'|} \}.
\end{align*}

Here we use the same notations as in the quiver algebroid stack $\hat{\cY}$, but the variables stand for the $\Lambda_0$-valued solutions of the unobstructed equations.

\begin{prop}
	There exists a correspondence between the isomorphism class of $\zeta$-stable deformation of $\bL$ and the deformation spaces of special Lagrangians $T_{r,0}.$ In other words, $$\textrm{Def}(\bL) \cong \bigcup_{r \in [a_1,a_{n+1}]} \textrm{Def} (T_{r,0})$$ In particular, when $r= a_i $, $T_{r,0}$ is the immersed sphere $\cS_i$ for $i=1, \cdots, n+1$. 
\end{prop}

\begin{proof}
	Let $U_i:=U_{0i}$ be the open subsets introduced in the quiver algebroid stack $\hat{\cY}$. Notice that $\textrm{Def}(\bL)$ can be covered by \begin{center}
	 $\textrm{Def}(\bL,U_{i}):=\{(u^L,v^L)\in M(Q,\vec{\delta})|$  $v_1^{L},\cdots, v_{i-1}^L, u_{i+1}^{L}, \cdots, u_{n+1}^L\in \Lambda_0^\ast \}/GL(\vec{\delta}),$
	\end{center} for $i=1, \cdots,n+1$, because the representation is $\zeta$-stable. Thus, given an element $[b] \in \textrm{Def}(\bL)$, there exists some $i$ such that $[b] \in \textrm{Def}(\bL,U_i)$. Using the group action, $[b]$ has a representative satisfies $v_1^L=\cdots=v_{i-1}^L=u_{i+1}^L=\cdots= u_{n+1}^L=1$. In the following, we will use the restriction of transition maps, found in Theorem \ref{thm:qstack}, to give a correspondence between the isomorphism classes of Maurer-Cartan subspaces.
	
	First, since $\val(u_i^L v_i^L)>0$, $\textrm{Def}(\bL,U_i) \cap \textrm{Def}(\bL,U_j)= \emptyset$ for $j \neq i-1,i,i+1$.   We can define
	the correspondence as follow. Given $[b] \in \textrm{Def}(\bL,U_i) \setminus (\textrm{Def}(\bL,U_{i-1}) \cup \textrm{Def}(\bL,U_{i+1}))$ for $i=1, \cdots, n+1$, the correspondence $G$ is defined as the restriction of $G_{i0}: \mathbb{A}_{0}|_{U_{0i}}\rightarrow\mathscr{A}_{i}|_{U_{0i}}$, where $\mathscr{A}_{i}$ is the formal deformation space of $\cS_i$. Namely,  $$G_{i0}:= \begin{cases} 
		u_{j}^L\mapsto v_{i}u_{i}& \forall 1\leq j \leq i-1 \\
		u_{i}^L\mapsto T^{A_i'-A_i}u_{i} &  \\ 
		u_{j}^L\mapsto 1 & \forall i+1 \leq j \leq n+1  \\
		v_{j}^L\mapsto 1 & \forall 1 \leq j \leq i-1  \\
		v_{i}^L\mapsto T^{A_i-A_i'}v_{i} &  \\
		v_{j}^L \mapsto  u_{i}v_{i}& \forall i+1 \leq j \leq n+1 
	\end{cases}, \quad G_{0i}:=\begin{cases}
		u_{i}\mapsto T^{A_i-A_i'}(u_{n+1}^L\cdots u_{i+1}^L)u_{i}^L(v_{1}^L\cdots v_{i-1}^L)^{-1} &  \\
		v_{i}\mapsto T^{A_i'-A_i}(v_{1}^L\cdots v_{i-1}^L)v_{i}^L(u_{n+1}^L\cdots u_{i+1}^L)^{-1}
	\end{cases}.$$

For $[b] \in \textrm{Def}(\bL,U_i) \cap \textrm{Def}(\bL,U_{i+1})$, the correspondence $G$ is defined as the restriction $ G_{i0}:\mathbb{A}_{0}|_{U_{0i}'}\rightarrow\cA_{i}|_{U_{0i}'}$, where $\cA_i$ is the formal deformation space of $T_i$ for $i=1,\cdots,n$. This gives a well-defined map from $\textrm{Def}(\bL)$ to $\bigcup_{r \in [a_1,a_{n+1}]} \textrm{Def} (T_{r,0})$. Notice that $G_{0i} \circ G_{0i}= id$, $G_{0i} \circ G_{i0}(a)= c_{0i0}(h(a)) \cdot a \cdot c_{0i0}^{-1}(t(a))$ and the gerbe terms are elements in $GL(\vec{\delta})$ when considering $\Lambda_0$-values deformations. Hence, $G_{0i}$ and $G_{i0}$ descend to well-defined mutually inverse representations. This provides the desired correspondence. 
\end{proof}

\begin{rem}
	The transition map, as constructed in Theorems \ref{thm:qstack} and \ref{thm:qstack2}, identifies the $\zeta$-stable deformations of $\bL$ with special Lagrangian tori on the moment map level zero with a flat connection. By restricting to complex-valued deformations, the moduli space of $\zeta$-stable Lagrangian branes $(\bL,\cE)$ forms a minimal resolution of $A_n$-singularity.  If we enlarge the deformation space of $\bL$ over $\Lambda$, the transition map identifies the $\zeta$-stable deformations of $\bL$ with special Lagrangian tori in other levels as well.
\end{rem}


\subsection{Extended Lagrangian deformations and derived Nakajima schemes}

In \cite[Chapter 9]{CHL21}, the noncommutative mirror functor was extended to the space of formal deformations in general degrees for a Lagrangian immersion.  The resulting mirror is a curved dg algebra $(\tilde{\A},d,W)$, where $\tilde{\A}$ is a quiver algebra that has arrows in non-positive degrees (while the quiver appeared in previous sections has arrows in degree $0$), $W \in \tilde{\A}$ is the extended disc potential, and $d: \A \to \A$ is a derivation of degree $1$, which satisfy the equalities 
$$dW = 0 \textrm{ and } d^2 = [W,\cdot].$$

We apply this construction to the two-dimensional framed Lagrangian immersion in Section \ref{sec:preproj}.  Our situation is graded, and hence $\tilde{W}=0$.  The result is similar to \cite[Prop. 9.10]{CHL21} for CY3 Ginzburg algebras.  In \cite{EL17}, dg algebras were constructed for Legendrian cohomology.  In our situation, the Lagrangians are framed.  Moreover, constant disc bubbles with more than two inputs of immersed sectors occur in the dimension-two situation, and we have to perform a change of coordinates as in Section \ref{sec:preproj} to standardize the resulting dg algebra.

Recall that the dg algebra is constructed from the generators of $\CF^\bullet(\bL^{\textrm{fr}})$ in all degrees, other than the unit of each component.  In our case, we have the degree one generators $X_a, X_{\bar{a}}, I_v, J_v$, and the degree two generators $T_v$ (the minimal points of the Morse functions on the sphere components).  Note that the Morse function on the framing $\R^2$ only has a single maximal point and does not have any degree two generators.  The resulting quiver $\tilde{Q}$ is the framed framed Nakajima quiver added with one self loop $\theta_v$ of degree $(-1)$ at each unframed vertex $v$.  Figure \ref{fig:extended-ADHM} shows $\tilde{Q}$ in the case of ADHM quiver.

\begin{figure}[ht]
	\includegraphics[scale=0.6]{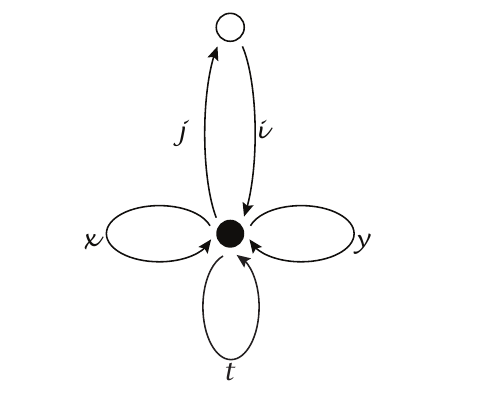}
	\caption{}
	\label{fig:extended-ADHM}
\end{figure}

Let
$$\tilde{b} = \sum_a (x_a X_a + x_{\bar{a}}X_{\bar{a}}) + \sum_v t_v T_v  $$
be the formal deformations.  Formally $\tilde{b}$ is made to be degree one by setting $\deg x_a = \deg x_{\bar{a}} = 0$ and $\deg t_v = -1$.  
Then $m_0^{\tilde{b}}$ is in degree $2$ and takes the form
$$m_0^{\tilde{b}} = \sum_v p_v T_v.  $$
Note that $\one_{L_v}$ and $X_a,X_{\bar{a}}$ do not appear in the above expression by degree reason.  Moreover, $p_v$ must have degree $0$ and hence are series only in $x_a,x_{\bar{a}},i,j$.

The derivation $d$ is defined by $d x_a = d x_{\bar{a}} = 0$ and $d t_v = p_v$, and then extended by Leibniz rule.  The $A_\infty$ equations imply that$(\Lambda \tilde{Q}, d)$ is a dg algebra.

The only non-trivial term $dt_v = p_v$ counts Floer contributions to the minimal point of the Morse function in each sphere component.  Its expression follows from definition and the change of coordinates in Section \ref{sec:preproj}.  Thus we can summarize as follows.

\begin{prop} \label{prop:extended}
	Let $\bL^{\mathrm{fr}}$ be the framed Lagrangian immersion as in Section \ref{sec:preproj}, whose (possibly immersed) sphere components are equipped with perfect Morse functions.   Then its extended localized mirror is $(\tilde{Q},d)$, where $\tilde{Q}$ is the corresponding framed Nakajima quiver with one more self loop $t_v$ of degree $(-1)$ attached to each unframed vertex $v$; 
	$$d: \widehat{\Lambda \tilde{Q}} \to \widehat{\Lambda \tilde{Q}}$$ 
	is the differential, which sends the degree-zero generator $\tilde{x}_a,\tilde{x}_{\bar{a}}, i_v,j_v$ of the path algebra to $0$ and
	$$ d t_v = \sum_{h(\tilde{x}_a)=v} \epsilon(a) \tilde{x}_a \tilde{x}_{\bar{a}} + i_vj_v$$
	where $\tilde{x}_a, \tilde{x}_{\bar{a}}$ are given in Equation \eqref{eq:coord} as a change of coordinates on $x_a,x_{\bar{a}}$.
\end{prop}

In \cite{DA22}, the dga $(\tilde{Q},d)$ was used as an explicit cofibrant resolution of the affine Nakajima variety and serves as a derived representation scheme model.  In a general dimension $d$, $(2-d)$-shifted symplectic structures \cite{MR2294224, BKR13, PTVV} are important in this subject.



\bibliography{mybib}{}
\bibliographystyle{alpha} 

\end{document}